\documentclass[letterpaper, 10 pt, draftclsnofoot, onecolumn]{ieeeconf}  

 \IEEEoverridecommandlockouts                              
\usepackage{graphicx} 
\usepackage{amsmath,bm,times} 
\usepackage{amssymb}  
\usepackage{tabularx}
\usepackage[ruled,vlined, procnumbered]{algorithm2e}
\usepackage{subfigure}
\usepackage{tikz}
\usetikzlibrary{shapes,arrows,backgrounds,fit,positioning,snakes,decorations.markings}
\let\labelindent\relax
\usepackage{enumitem}
\usepackage{cite}

\newtheorem{mydef}{\textbf{Definition}} [section]
\newtheorem{assump}[mydef] {\textbf{Assumption}}
\newtheorem{thm}[mydef] {\textbf{Theorem}}
\newtheorem{lemma}[mydef]{\textbf{Lemma}}
\newtheorem{prop}[mydef]{\textbf{Proposition}}
\newtheorem{cor}[mydef]{\textbf{Corollary}}
\newtheorem{remark}[mydef]{\textbf{Remark}}

\newtheorem{prob}[mydef]{\textbf{Problem}}

\renewcommand{\themydef}{\thesection.\arabic{mydef}}

\title{\LARGE \bf
Optimal Remote State Estimation for Self-Propelled Particle Models}

\author{Shinkyu Park and Nuno C. Martins
\thanks{Shinkyu Park and Nuno C. Martins are with the Department of Electrical and Computer Engineering, University of Maryland College Park, College Park, MD 20742-4450, USA.
        {\tt\small \{skpark, nmartins\}@umd.edu}}%
}

\begin{document}

\maketitle

\begin{abstract}
  We investigate the design of a remote state estimation system for a self-propelled particle (SPP). Our framework consists of a sensing unit that accesses the full state of the SPP and an estimator that is remotely located from the sensing unit. The sensing unit must pay a cost when it chooses to transmit information on the state of the SPP to the estimator; and the estimator computes the best estimate of the state of the SPP based on received information. In this paper, we provide methods to design transmission policies and estimation rules for the sensing unit and estimator, respectively, that are optimal for a given cost functional that combines state estimation distortion and communication costs. We consider two notions of optimality: joint optimality and person-by-person optimality.\footnote{The precise definitions of joint optimality and person-by-person optimality are given in Definition \ref{def_joint_optimal} and Definition~\ref{def_pbp_optimal}, respectively.} Our main results show the existence of a jointly optimal solution and describe an iterative procedure to find a person-by-person optimal solution. In addition, we explain how the remote estimation scheme can be applied to tracking of animal movements over a costly communication link. We also provide experimental results to show the effectiveness of the scheme.
\end{abstract}

\section{Introduction}
Consider a self-propelled particle (SPP) moving in a two-dimensional plane whose state $\mathbf x_k$ is represented as follows:
$${\mathbf x_k = \begin{pmatrix} \mathbf p_{1,k} & \mathbf p_{2,k} & \boldsymbol{\theta}_{k} \end{pmatrix}^T \in \mathbb R^2 \times [0, 2\pi)}$$ 
where $\left( \mathbf p_{1,k}, \mathbf p_{2,k} \right)$ and $\boldsymbol{\theta}_{k} $ represent the location in the plane and the orientation at time $k$, respectively. The state of the SPP evolves according to the following model:
\begin{align} \label{eq_spp_model}
  \begin{pmatrix} \mathbf{p}_{1,k+1} \\ \mathbf{p}_{2,k+1} \\ \boldsymbol \theta_{k+1} \end{pmatrix} = \begin{pmatrix} \mathbf{p}_{1,k} + \mathbf v_k \cos \left( \boldsymbol \theta_{k} + \boldsymbol \phi_k \right) \\ \mathbf{p}_{2,k} + \mathbf v_k \sin \left( \boldsymbol \theta_{k} + \boldsymbol \phi_k \right) \\ \boldsymbol \theta_{k} +  \boldsymbol \phi_k \end{pmatrix}, ~ k \geq 0
\end{align}
with the initial condition $\mathbf x_0 = x_0 = \begin{pmatrix} p_{1,0} & p_{2,0} & \theta_{0} \end{pmatrix}^T$. The random processes $\mathbf v_k$ and $\boldsymbol{\phi}_k$ represent the translational and angular velocities, respectively.

In this paper, we consider a remote estimation system formed by a \textit{sensing unit} and \textit{remotely located estimator}: The sensing unit accesses $\mathbf x_k$ and has the authority to decide whether to transmit it to the estimator. The sequence of decisions on whether to transmit is represented by $\mathbf R_k$, for which $\mathbf R_k=1$ if the sensing unit decides to transmit and $\mathbf R_k=0$ otherwise. The cost of each transmission is represented by $c_k$. The estimator computes a state estimate $\hat{\mathbf x}_k = \begin{pmatrix} \hat{\boldsymbol{p}}_{1,k} & \hat{\boldsymbol{p}}_{2,k} & \hat{\boldsymbol \theta}_{k} \end{pmatrix}^T$ based on received information. The diagram in Fig. \ref{figure_remote_estimation} depicts the overall framework adopted here.

\subsection{Outline of Main Results}
Let a transmission policy $\boldsymbol{\mathcal T}_k$ and an estimation rule $\mathcal E_k$ for the sensing unit and estimator at time $k$, respectively, be defined as follows:
\begin{align*}
  \boldsymbol{\mathcal T}_k&: \left( \mathbb R^2 \times [0, 2\pi) \right)^{k+1} \times \{0, 1\}^{k-1} \to \{0, 1\} \\
  \mathcal E_k&: \left( \mathbb R^2 \times [0, 2\pi) \right)^{\left| \mathcal I^k \right|+1} \times \{0, 1\}^k \to \mathbb R^2 \times [0, 2\pi)
\end{align*}
\normalsize
where the variable ${\mathcal I^k = \left\{ \mathbf x_j \,\Big|\, \mathbf R_j = 1, \, 1 \leq j \leq k\right\}}$ represents information transmitted to the estimator up to time $k$.

\textbf{Our main goal} is to obtain methods to design transmission policies $\left( \boldsymbol{\mathcal T}_1, \cdots, \boldsymbol{\mathcal T}_N\right)$ and estimation rules $\left( \mathcal E_1, \cdots, \mathcal E_N \right)$ that are optimal for the following cost functional:
\begin{align} \label{eq_cost_functional}
   & \mathcal J \left( x_0, \left( \boldsymbol{\mathcal T}_{1}, \cdots, \boldsymbol{\mathcal T}_{N} \right), \left(\mathcal E_{1}, \cdots, \mathcal E_{N}\right) \right) \nonumber \\
   & =\sum_{k=1}^N \mathbb E \left[ d^2 \left( \mathbf x_k, \hat{\mathbf x}_k \right) + c_k \cdot \mathbf R_k \,\Big|\, \mathbf x_0=x_0, \left( \boldsymbol{\mathcal T}_{1}, \cdots, \boldsymbol{\mathcal T}_{N} \right), \left(\mathcal E_{1}, \cdots, \mathcal E_{N}\right) \right]
\end{align}
\normalsize
subject to the SPP model \eqref{eq_spp_model} and 
\begin{subequations} \label{eq_policy_rule}
  \begin{align} 
    \mathbf R_k &= \boldsymbol{\mathcal T}_k \left( \left( \mathbf x_0, \cdots, \mathbf x_k \right), \left(\mathbf R_1, \cdots, \mathbf R_{k-1} \right) \right) \label{eq_transmission_policy} \\
    \hat{\mathbf x}_k &= \mathcal E_k \left( \left( \mathbf x_0, \mathcal I^k \right), \left(\mathbf R_1, \cdots, \mathbf R_k \right) \right) \label{eq_estimation_rule}
  \end{align}
\end{subequations}
for each $k$ in $\{1, \cdots, N\}$, where we use Frobenius norm to define the metric $d$ as follows:
\begin{align*}
  & d \left( x_k, \hat x_k \right) = \left\| \begin{pmatrix} \cos \theta_k & -\sin \theta_k & p_{1,k} \\ 
      \sin \theta_k & \cos \theta_k & p_{2,k} \\
      0 & 0 & 1 \end{pmatrix} - \begin{pmatrix} \cos \hat \theta_k & -\sin \hat \theta_k & \hat p_{1,k} \\ 
      \sin \hat \theta_k & \cos \hat \theta_k & \hat p_{2,k} \\
      0 & 0 & 1 \end{pmatrix} \right\|_F
\end{align*}

\tikzstyle{plant} = [draw, text centered, circle, minimum height=2.5em, minimum width=2.5em, thick]
\tikzstyle{block} = [draw, text centered, rounded corners, minimum height=2.5em, minimum width=2.5em, thick]
\tikzstyle{dotted_block} = [draw, dashed, text centered, circle, minimum height=3em, minimum width=3em, thick]
\tikzstyle{block_empty} = [text centered, minimum height=3em, minimum width=3em]
\tikzstyle{vertex} = [circle, draw, inner sep=0, minimum width=.3em, thick, fill=black]

\def\blockdist{2em}

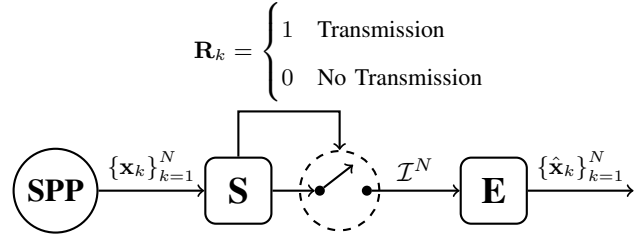
\begin{figure} [t]
  \centering
  \begin{tikzpicture} 
    \node (plant) [plant] {\textbf{\large SPP}};
    \node (sensing_unit) [block, right = 2.0*\blockdist of plant] {\textbf{\Large S}};
    \node (switch) [dotted_block, right = .5*\blockdist of sensing_unit] {};
    \node (estimator) [block, right = 1.5*\blockdist of switch] {\textbf{\Large E}};
    \node (dummy_block) [block_empty, right = 2.0*\blockdist of estimator] {};
    \node (vertex1) [vertex, right = .8*\blockdist of sensing_unit] {};
    \node (vertex2) [vertex, right = .7*\blockdist of vertex1] {};

    \draw [->, thick] (plant) -- node [above] {\small $\left\{ \mathbf x_k \right\}_{k=1}^N$} (sensing_unit);
    \draw [->, thick] (sensing_unit) -- (vertex1);
    \draw [->, thick] (vertex1) -- ++(.63*\blockdist, .49*\blockdist);
    \draw [->, thick] (vertex2) -- node [above] {$\mathcal I^N$} (estimator);
    \draw [->, thick] (sensing_unit) -- ++(0, 1.5*\blockdist) -| node [above] {\small $\mathbf{R}_k = \begin{cases} 1 & \text{Transmission} \\ 0 & \text{No Transmission} \end{cases}$} (switch);
    \draw [->, thick] (estimator) -- node [above] {\small $\left\{ \hat{\mathbf x}_k \right\}_{k=1}^N$} (dummy_block);
  \end{tikzpicture}

  \caption{A remote estimation framework comprised of a self-propelled particle (SPP), a sensing unit (S), and an estimator (E), where ${\mathcal I^N = \left\{ \mathbf x_k \,\Big|\, \mathbf R_k = 1, \, 1 \leq k \leq N\right\}}$.}
  \label{figure_remote_estimation}
\end{figure}

Our problem is non-trivial because \eqref{eq_cost_functional} is in general non-convex and searching for a solution that achieves the minimum over a function space is computationally complex. We adopt a team decision framework in which the sensing unit and the estimator are viewed as players. The following are our main contributions:
\begin{enumerate}
  \item First, we show that there is a jointly optimal solution which minimizes the cost functional \eqref{eq_cost_functional}. As joint optimality implies person-by-person optimality, this result ensures that the set of person-by-person optimal solutions is non-empty.
  
  \item We propose an iterative procedure, which is inspired by Lloyd's algorithm \cite{qiang_du1999_siam_review}, to compute a person-by-person optimal solution. The procedure alternates between finding the best transmission policies for \eqref{eq_cost_functional} with the estimation rules fixed, and vice versa; and it generates a sequence of sub-optimal solutions. Our analysis will show that the sequence has a convergent subsequence; and the limit of any convergent subsequence is a person-by-person optimal solution.

  \item We illustrate the performance of the optimal remote estimation scheme in the context of tracking of animal movements over a costly communication link. Our numerical results use GPS data collected from a monitoring device mounted on an African buffalo.
\end{enumerate}

\subsection{Paper Organization}
In Section \ref{section_formulation}, we describe the problem formulation considered throughout the paper, and briefly describe our methodology to find a solution. The main strategy is to decompose the problem into sub-problems, which we can solve sequentially. In Section \ref{section_two-player_stopping_problem}, we examine the existence of a jointly optimal solution to each sub-problem. We also describe an iterative procedure for finding a person-by-person optimal solution. Section \ref{section_experimental_results} discusses an application of our results to tracking of animal movements and also presents experimental results. 

\section{Problem Formulation} \label{section_formulation}
\subsection{Notation and Terminology}
\begin{itemize}
\item For a finite sequence of elements $a_1, \cdots, a_N$ belonging to a set, we adopt the shorthand notation ${a_{1:N} = \left( a_1, \cdots, a_N \right)}$.

\item For a finite sequence of functions $\mathcal A_1, \cdots, \mathcal A_N $ defined on a set, we adopt the shorthand notation ${\mathcal A_{1:N} = \left( \mathcal A_1, \cdots, \mathcal A_N \right)}$.


\item For $\left\{ \mathbf R_j \right\}_{j=1}^{k-1}$, we define\footnote{We adopt a convention that $\boldsymbol \tau_k = 0$ if $\mathbf R_j=0$ for all $j$ in ${\{1, \cdots, k-1\}}$.}
$${\boldsymbol \tau_k = \max \left\{1 \leq j \leq k-1 \,\Big|\, \mathbf R_j = 1\right\}}$$
We refer to $\boldsymbol \tau_k$ as the \textit{last transmission time} before time $k$.


\end{itemize}

\subsection{Problem Description}
We start by assuming that transmission policies and estimation rules have the following structure\footnote{We do not lose any optimality from imposing these structures. This can be verified by similar arguments as in Lemma 1 and Lemma 3 of \cite{molin2012_ifac}.}: The transmission policy, which may be randomized
\footnote{See Appendix \ref{remark_randomized_transmission_policies} for a detailed description of randomized transmission policies policies.}, 
at time $k$ depends on the last transmission time $\boldsymbol{\tau}_k$, the information $\mathbf x_{\boldsymbol{\tau}_k}$ transmitted to the estimator at time $\boldsymbol \tau_k$, and the current state $\mathbf x_k$ of the SPP. The estimation rule at time $k$ depends on the last transmission time $\boldsymbol{\tau}_k$ and the information $\mathbf x_{\boldsymbol{\tau}_k}$ received from the sensing unit at time $\boldsymbol{\tau}_k$. 

According to \eqref{eq_policy_rule} and the structural assumptions mentioned above, the variable $\mathbf R_k$ and estimate $\hat{\mathbf x}_k$ are determined by a transmission policy $\boldsymbol{\mathcal T}_k$ and an estimation rule $\mathcal E_k$ as follows:
\begin{subequations} \label{eq_policy_rule_02}
  \begin{align}
    \mathbf R_k &= \boldsymbol{\mathcal T}_k \left( \boldsymbol{\tau}_k, \mathbf x_{\boldsymbol\tau_k},  \mathbf x_k \right) \label{eq_transmission_policy_02} \\
    \hat{\mathbf x}_k &= \begin{cases} \mathcal E_k \left( \boldsymbol{\tau}_k, \mathbf x_{\boldsymbol\tau_k}\right) & \text{if } \mathbf R_k=0 \\ \mathbf x_k & \text{otherwise} \end{cases}
  \end{align}
\end{subequations}
\normalsize
We formally state our main problem as follows.

\noindent \hrulefill

\begin{prob} \label{prob_remote_estimation}
  Find transmission policies $\boldsymbol{\mathcal T}_{1:N}$ and estimation rules $\mathcal E_{1:N}$ that are optimal for the cost functional \eqref{eq_cost_functional} subject to the SPP model \eqref{eq_spp_model} with the initial condition $\mathbf x_0 = x_0$ and \eqref{eq_policy_rule_02}.\footnote{The underlying SPP model and the initial condition $\mathbf x_0 = x_0$ are common knowledge to both the sensing unit and estimator.}
\end{prob}

\noindent \hrulefill

We consider the following two notions of optimality for Problem \ref{prob_remote_estimation}.
\begin{mydef} \label{def_joint_optimal}
  We say that transmission policies $\boldsymbol{\mathcal T}_{1:N}^\ast$ and estimation rules $\mathcal E_{1:N}^\ast$ are \textit{jointly optimal} for \eqref{eq_cost_functional} if they achieve the global minimum for every $x_0$ in $\mathbb R^2 \times [0, 2\pi)$.
\end{mydef}

\begin{mydef} \label{def_pbp_optimal}
  We say that transmission policies $\boldsymbol{\mathcal T}_{1:N}^\ast$ and estimation rules $\mathcal E_{1:N}^\ast$ are \textit{person-by-person optimal} for \eqref{eq_cost_functional} if the following relations hold for every $x_0$ in $\mathbb R^2 \times [0, 2\pi)$:
  \begin{align} \label{eq_def_pbp_optimal}
    \mathcal J \left( x_0, \boldsymbol{\mathcal T}_{1:N}^\ast, \mathcal E_{1:N}^\ast \right) &= \min_{\boldsymbol{\mathcal T}_{1:N}} \mathcal J \left( x_0, \boldsymbol{\mathcal T}_{1:N}, \mathcal E_{1:N}^\ast \right) \nonumber \\
    &= \min_{\mathcal E_{1:N}} \mathcal J \left( x_0, \boldsymbol{\mathcal T}_{1:N}^\ast, \mathcal E_{1:N} \right)
  \end{align}
  \normalsize
  Equation \eqref{eq_def_pbp_optimal} implies that with the transmission policies $\boldsymbol{\mathcal T}_{1:N}^\ast$ fixed, the estimation rules $\mathcal E_{1:N}^\ast$ minimize the cost functional \eqref{eq_cost_functional}, and vice versa.
\end{mydef}

We maintain the following assumption throughout the paper.
\begin{assump} \label{assump_transition_probability}
  Let $\mathfrak B$ be a Borel $\sigma$-algebra on ${\mathbb R^2 \times [0, 2\pi)}$. We assume that $\mathbf v_k$ and $\boldsymbol{\phi}_k$ in \eqref{eq_spp_model} are random processes for which the following hold for every $k$ in $\{1, \cdots, N\}$:
  \begin{enumerate}
  \item For every Borel set $\mathbb A$ in $\mathfrak B$, the function ${x \mapsto \mathbb P \left( \mathbf x_k \in \mathbb A \,\Big|\, \mathbf x_{k-1} = x \right)}$ is well-defined and continuous.

  \item For every non-empty open set $\mathbb O$, the function ${x \mapsto \mathbb P \left( \mathbf x_k \in \mathbb O \,\Big|\, \mathbf x_{k-1} = x \right)}$ is positive for all $x$ in $\mathbb R^2 \times [0, 2\pi)$.
  \end{enumerate}
\end{assump}

To find a solution to Problem \ref{prob_remote_estimation}, we decompose the problem into a set of $N$ sub-problems, which we solve sequentially. We start by describing the so-called \textit{Two-Player Optimal Stopping Problem}, which we use to decompose Problem \ref{prob_remote_estimation} into sub-problems. We then describe how to obtain a solution to Problem \ref{prob_remote_estimation} by solving the sub-problems.

\noindent\hrulefill

\begin{prob} [Two-Player Optimal Stopping Problem] \label{prob_two-player_stopping}
  Given positive real numbers $\left\{ c_j'\right\}_{j=k}^N$, find policies $\boldsymbol{\mathcal T}_{k:N}^{<k-1>}$ and rules $\mathcal E_{k:N}^{<k-1>}$ that are optimal for the following cost functional:
  \begin{align} \label{eq_cost_functional_two-player_stopping}
    &\mathcal J_k \left( x_{k-1}, \boldsymbol{\mathcal T}_{k:N}^{<k-1>}, \mathcal E_{k:N}^{<k-1>} \right) \nonumber \\
    &= \mathbb E \left[ \sum_{j=k}^{\mathbf K} d^2 \left( \mathbf x_j, \hat{\mathbf x}_j \right) + c_{\mathbf K}' \cdot \mathbf R_{\mathbf K} \,\Bigg|\, \mathbf x_{k-1} = x_{k-1}, \boldsymbol{\mathcal T}_{k:N}^{<k-1>}, \mathcal E_{k:N}^{<k-1>} \right]
  \end{align}
  \normalsize
  subject to \eqref{eq_spp_model} with the initial condition $\mathbf x_{k-1} = x_{k-1}$ and 
  \begin{subequations} \label{eq_policy_rule_03}
    \begin{align}
      \mathbf R_j &= \boldsymbol{\mathcal T}_j^{<k-1>} \left( x_{k-1}, \mathbf x_j \right) \\
      \hat{\mathbf x}_j &= \begin{cases} \mathcal E_j^{<k-1>} \left( x_{k-1} \right) & \text{if } \mathbf R_j=0 \\ \mathbf x_j & \text{otherwise} \end{cases}
    \end{align}
  \end{subequations}
  \normalsize
  for each $j$ in $\{k, \cdots, N\}$, where\footnote{We adopt a convention that $\mathbf K = N$ if $\mathbf R_j=0$ for all $j$ in $\{k, \cdots, N\}$.}
  {$$\mathbf K = \min \left\{ k \leq j \leq N \,\Big|\, \mathbf R_j = 1 \right\}$$}
\end{prob}

\noindent\hrulefill

Note that the total expected cost \eqref{eq_cost_functional_two-player_stopping} consists of \textit{running costs} $d^2 \left( \mathbf x_j, \hat{\mathbf x}_j \right)$ and \textit{stopping costs} $c_j'$.

Similar to Definitions \ref{def_joint_optimal} and \ref{def_pbp_optimal}, we adopt two notions of optimality for Problem \ref{prob_two-player_stopping} as follows.
\begin{mydef} \label{def_joint_optimal_stopping}
  We say that policies ${\boldsymbol{\mathcal T}^\ast}_{k:N}^{<k-1>}$ and rules ${\mathcal E^\ast}_{k:N}^{<k-1>}$ are \textit{jointly optimal} for \eqref{eq_cost_functional_two-player_stopping} if they achieve the global minimum for every $x_{k-1}$ in $\mathbb R^2 \times [0, 2\pi)$.
\end{mydef}

\begin{mydef} \label{def_pbp_optimal_stopping}
  We say that policies ${\boldsymbol{\mathcal T}^\ast}_{k:N}^{<k-1>}$ and rules ${\mathcal E^\ast}_{k:N}^{<k-1>}$ are \textit{person-by-person optimal} for \eqref{eq_cost_functional_two-player_stopping} if the following relations hold for every $x_{k-1}$ in $\mathbb R^2 \times [0, 2\pi)$:
  \begin{align} 
    \mathcal J_k \left( x_{k-1}, {\boldsymbol{\mathcal T}^\ast}_{k:N}^{<k-1>}, {\mathcal E^\ast}_{k:N}^{<k-1>} \right) &= \min_{\boldsymbol{\mathcal T}_{k:N}^{<k-1>}} \mathcal J_k \left( x_{k-1}, {\boldsymbol{\mathcal T}}_{k:N}^{<k-1>}, {\mathcal E^\ast}_{k:N}^{<k-1>} \right) \nonumber \\
    &= \min_{{\mathcal E}_{k:N}^{<k-1>}} \mathcal J_k \left( x_{k-1}, {\boldsymbol{\mathcal T}^\ast}_{k:N}^{<k-1>}, {\mathcal E}_{k:N}^{<k-1>} \right)
  \end{align}
  \normalsize
\end{mydef}




To explain how to decompose Problem \ref{prob_remote_estimation} into sub-problems, we consider recursive computations of constants $\left\{ c_j' \right\}_{j=k}^N$ described as follows: Suppose that policies $\boldsymbol{\mathcal T}_{j+1:N}^{<j>}$ and rules $\mathcal E_{j+1:N}^{<j>}$ are given for all $j$ in ${\{k, \cdots, N\}}$. By proceeding backwards from $j=N$ to $j=k$, for each step $j$, let us compute
\begin{align} \label{eq_stopping_cost}
  c_j' &= c_j + \mathcal J_{j+1} \left( 0, \boldsymbol{\mathcal T}_{j+1:N}^{<j>}, \mathcal E_{j+1:N}^{<j>} \right)
\end{align}
\normalsize
with $c_N' = c_N$, where $c_j$ is given in \eqref{eq_cost_functional} and $\mathcal J_{j+1}$ is defined in \eqref{eq_cost_functional_two-player_stopping}. In computing $\mathcal J_{j+1} \left( 0, \boldsymbol{\mathcal T}_{j+1:N}^{<j>}, \mathcal E_{j+1:N}^{<j>} \right)$, we use the constants $\left\{ c_l' \right\}_{l=j+1}^N$ that are obtained in preceding steps.

We describe the $k$-th sub-problem of Problem \ref{prob_remote_estimation} as follows.

\noindent\hrulefill

\noindent\textbf{Sub-problem $k$: } Given policies $\boldsymbol{\mathcal T}_{j+1:N}^{<j>}$ and rules $\mathcal E_{j+1:N}^{<j>}$ for all $j$ in $\{k, \cdots, N\}$, let us compute constants $\left\{ c_j' \right\}_{j=k}^N$ according to \eqref{eq_stopping_cost}. Using $\left\{c_j' \right\}_{j=k}^N$, find a solution $\boldsymbol{\mathcal T}_{k:N}^{<k-1>}$ and $\mathcal E_{k:N}^{<k-1>}$ to Problem~\ref{prob_two-player_stopping}.

\noindent\hrulefill \newline

\textbf{Our main strategy for solving} Problem \ref{prob_remote_estimation} is as follows:  We solve each Sub-problem $k$ backwards in time from ${k=N}$ to $k=1$, where for each Sub-problem $k$ we use solutions to all preceding sub-problems, i.e., $\boldsymbol{\mathcal T}_{j+1:N}^{<j>}$ and $\mathcal E_{j+1:N}^{<j>}$ for all $j$ in $\{k, \cdots, N\}$, to compute the constants $\left\{ c_j' \right\}_{j=k}^N$. Once solutions to all the sub-problems are found, we determine transmission policies $\boldsymbol{\mathcal T}_{1:N}$ and estimation rules $\mathcal E_{1:N}$ for Problem~\ref{prob_remote_estimation} in the following way:
\begin{subequations} \label{eq_solutions}
  \begin{align} 
    \boldsymbol{\mathcal T}_j \left(k-1, x_{k-1}, x_j \right) = \boldsymbol{\mathcal T}_j^{<k-1>} \left(x_{k-1}, x_j \right) \\
    \mathcal E_j \left(k-1, x_{k-1} \right) = \mathcal E_j^{<k-1>} \left( x_{k-1} \right)
  \end{align}
\end{subequations}
for each $j$ in $\{k, \cdots, N\}$ and $k$ in $\{1, \cdots, N\}$. It can be verified that the transmission policies and estimation rules determined by \eqref{eq_solutions} are a solution to Problem \ref{prob_remote_estimation}.


\subsection{Brief Survey of Related Work}

Finite time-horizon problem formulations are considered in \cite{lipsa2011_ieee_tac, molin2012_ifac, nayyar2013_ieee_tac, spark7040013}. The authors of \cite{lipsa2011_ieee_tac} found a jointly optimal solution for a remote estimation problem under first-order linear processes driven by Gaussian noise where it is shown that transmission policies of jointly optimal solutions are of threshold-type. An iterative procedure for finding transmission policies and estimation rules was proposed in \cite{molin2012_ifac}. The authors performed a convergence analysis on the proposed procedure in the same problem formulation of \cite{lipsa2011_ieee_tac}, which essentially leads to an alternative proof of the main results of \cite{lipsa2011_ieee_tac}. The work of \cite{nayyar2013_ieee_tac} considered a problem setting in which the sensing unit has an energy harvesting capability. Preliminary results of our work were presented in \cite{spark7040013} under some technical assumptions.\footnote{The analysis and results presented in this work can be readily extended to the problem formulation considered in \cite{spark7040013}. Due to the space constraints, we will focus on SPP models.}

Infinite time-horizon formulations are considered in \cite{yonggang_xu2004_ieee_cdc, cogil2007_acc, lichun_li2011_ieee_cdc-ecc, lichun_li2013_acc}.  The authors of \cite{yonggang_xu2004_ieee_cdc} studied the structure of optimal transmission policies for a remote estimation problem under linear processes driven by Gaussian noise, and proposed a procedure based on the value iteration algorithm to compute an optimal policy. In \cite{cogil2007_acc}, an algorithm for finding a sub-optimal solution was proposed. The authors showed that when the underlying process is linear and driven by Gaussian noise, the proposed algorithm incurs a cost that is within a constant factor of the optimum. While the question of whether transmission policies of jointly optimal solutions are of threshold-type for the problems under multi-dimensional linear processes remains unanswered, the authors of \cite{lichun_li2011_ieee_cdc-ecc} analyzed the performance of threshold-type transmission policies for such problems. In \cite{lichun_li2013_acc}, the authors proposed a polynomial approximation-based method to find sub-optimal transmission policies.

Our problem formulation and methods are distinguished from previous ones found in literature by the following facts: 

\begin{enumerate}
\item We adopt a random process model that is nonlinear.

\item We do not impose any structural assumptions on transmission policies and estimation rules that result in the loss of optimality.

\item We investigate optimization of a given performance criterion over both transmission policies and estimation rules.
\end{enumerate}

\section{Two-Player Optimal Stopping Problem} \label{section_two-player_stopping_problem}
In this section, we investigate Sub-problem $k$ in which, to determine the constants $\left\{ c_j' \right\}_{j=k}^N$, we use solutions to the preceding sub-problems -- Sub-problem $N$ to Sub-problem~$k+1$. We start by re-writing \eqref{eq_cost_functional_two-player_stopping} into a suitable form using the following definition.
\begin{mydef} \label{def_solutions}
  For each $j$ in $\{k, \cdots, N\}$, we define a (random) function ${\boldsymbol{\mathcal P}_j:\mathbb R^2 \times [0, 2\pi) \to \{0, 1\}}$ and a variable $\hat x_j $ in $\mathbb R^2 \times [0, 2\pi)$ as follows:
  \begin{subequations} \label{eq_decision_sets_estimates}
    \begin{align}
      \boldsymbol{\mathcal P}_j \left( x_j \right) &= \boldsymbol{\mathcal T}_j^{<k-1>} \left( 0, x_j  \right) \\
      \hat x_j &= \mathcal E_j^{<k-1>} \left( 0 \right)
    \end{align}
  \end{subequations}
  We refer to $\boldsymbol{\mathcal P}_j$ and $\hat x_j$ as the \textit{(randomized) policy} and \textit{estimate} at time $j$ (for the initial condition $\mathbf x_{k-1}=0$), respectively.
\footnote{See Appendix \ref{remark_randomized_transmission_policies} for a detailed description of randomized policies.} 
\end{mydef}

Given that $\mathbf x_{k-1}=0$, we can re-write \eqref{eq_cost_functional_two-player_stopping} as follows:\footnote{For concise presentation, we will omit the dependence of the cost functional \eqref{eq_cost_functional_two-player_stopping_02} on the initial condition unless it is necessary.}
\begin{align} \label{eq_cost_functional_two-player_stopping_02}
  \mathbb E_{\mathbf x_k} \left[ J_k \left(\mathbf x_{k}, \boldsymbol{\mathcal P}_{k:N}, \hat x_{k:N} \right) \right]
\end{align}
subject to \eqref{eq_spp_model} with the initial condition $\mathbf x_{k-1} = 0$ and 
\begin{align} \label{eq_decision_variables_subproblem}
  \mathbf R_j = \boldsymbol{\mathcal P}_j \left( \mathbf x_j  \right)
\end{align}
for each $j$ in $\{k, \cdots, N\}$, where $J_k$ is recursively defined as follows:
\begin{align*}
  &J_j \left(x_{j}, \boldsymbol{\mathcal P}_{j:N}, \hat x_{j:N} \right) \nonumber \\ 
  &= \left( d^2\left( x_{j}, \hat x_{j} \right) + \mathbb E_{\mathbf x_{j+1}} \left[ J_{j+1} \left( \mathbf x_{j+1}, \boldsymbol{\mathcal P}_{j+1:N}, \hat x_{j+1:N} \right) \,\Big|\, \mathbf x_j = x_j \right] \right) \cdot \left(1-\mathbf R_j\right) + c_j' \cdot \mathbf R_j
\end{align*}
\normalsize
for each $j$ in $\{k, \cdots, N\}$ with $J_{N+1} = 0$. Note that $J_j$ satisfies the following for every $j$ in $\{k, \cdots, N\}$:
\begin{align} \label{eq_cost-to-go_01}
  &\mathbb E_{\mathbf x_j} \left[ J_j \left(\mathbf x_j, \boldsymbol{\mathcal P}_{j:N}, \hat x_{j:N} \right) \,\Big|\, \mathbf R_k=0, \cdots, \mathbf R_{j-1}=0 \right] \nonumber \\
  &= \Big( \mathbb E_{\mathbf x_j} \left[ d^2\left( \mathbf x_j, \hat x_j \right) \,\Big|\, \mathbf R_k=0, \cdots, \mathbf R_j=0  \right] + \mathbb E_{\mathbf x_{j+1}} \left[ J_{j+1} \left( \mathbf x_{j+1}, \boldsymbol{\mathcal P}_{j+1:N}, \hat x_{j+1:N} \right) \,\Big|\, \mathbf R_k=0, \cdots, \mathbf R_j=0 \right] \Big) \nonumber \\
  &\quad \cdot \mathbb P\left(\mathbf R_j=0 \,\Big|\, \mathbf R_k=0, \cdots, \mathbf R_{j-1}=0\right) + c_j' \cdot \mathbb P\left(\mathbf R_j=1 \,\Big|\, \mathbf R_k=0, \cdots, \mathbf R_{j-1}=0\right)
\end{align}

We will proceed with finding an optimal solution $\boldsymbol{\mathcal P}_{k:N}^\ast$ and $\hat x_{k:N}^\ast$ for \eqref{eq_cost_functional_two-player_stopping_02}. Remark \ref{remark_solution_transformation} given below explains how we can derive a solution to Sub-problem $k$ from $\boldsymbol{\mathcal P}_{k:N}^\ast$ and $\hat x_{k:N}^\ast$.

\begin{remark} \label{remark_solution_transformation}
  Consider the transformations given below:
  \begin{align*}
    M \left( x_{k-1}, x_{j} \right) &= \begin{pmatrix} \cos \theta_{k-1} & \sin \theta_{k-1} & 0 \\ 
      -\sin \theta_{k-1} & \cos \theta_{k-1} & 0 \\
      0 & 0 & 1 \end{pmatrix} \cdot \begin{pmatrix} p_{1,j} - p_{1,k-1} \\ p_{2,j} - p_{2,k-1}\\ \theta_j - \theta_{k-1} \end{pmatrix} \\
    M^\dagger \left( x_{k-1}, x_{j} \right) &= \begin{pmatrix} \cos \theta_{k-1} & -\sin \theta_{k-1} & 0 \\ 
      \sin \theta_{k-1} & \cos \theta_{k-1} & 0 \\
      0 & 0 & 1 \end{pmatrix} \cdot \begin{pmatrix} p_{1,j} \\ p_{2,j} \\ \theta_j \end{pmatrix} + \begin{pmatrix} p_{1,k-1} \\ p_{2,k-1}\\ \theta_{k-1} \end{pmatrix}
  \end{align*}
  \normalsize
  Suppose that $\boldsymbol{\mathcal P}_{k:N}^\ast$ and $\hat x_{k:N}^\ast$ are optimal policies and estimates for \eqref{eq_cost_functional_two-player_stopping_02}, respectively, and that a solution to Sub-problem $k$ are determined as follows: For each $j$ in $\{k, \cdots, N\}$,
  \begin{subequations} \label{eq_remark_solution_transformation_03}
    \begin{align}
      {\boldsymbol{\mathcal T}^\ast}_j^{<k-1>} \left( x_{k-1}, x_j \right) &= \boldsymbol{\mathcal P}_j^\ast \left( M\left( x_{k-1}, x_j \right) \right) \\
      {\mathcal E^\ast}_j^{<k-1>} \left( x_{k-1} \right) &= M^\dagger \left( x_{k-1}, \hat x_j^\ast \right)
    \end{align}
  \end{subequations}

  Based on Definition \ref{def_solutions}, it can be verified that the following holds for all $x_{k-1}$ in $\mathbb R^2 \times [0, 2\pi)$:
  \begin{align*}
    \mathcal J_k \left( x_{k-1}, {\boldsymbol{\mathcal T}^\ast}_{k:N}^{<k-1>}, {\mathcal E^\ast}_{k:N}^{<k-1>} \right) =\mathbb E_{\mathbf x_k} \left[J_k \left(\mathbf x_k, \boldsymbol{\mathcal P}_{k:N}^\ast, \hat x_{k:N}^\ast \right) \right]
  \end{align*}
  where $\mathcal J_k$ is defined in \eqref{eq_cost_functional_two-player_stopping}. This implies that the value of \eqref{eq_cost_functional_two-player_stopping} evaluated at an optimal solution does not depend on the initial condition; and by finding an optimal solution for the sub-problem with the initial condition ${\mathbf x_{k-1}=0}$, we can derive a solution to Sub-problem $k$ using \eqref{eq_remark_solution_transformation_03}.   \hfill $\blacksquare$
\end{remark}

\subsection{Definitions and Preliminary Results}
We restate Definition \ref{def_joint_optimal_stopping} and Definition \ref{def_pbp_optimal_stopping} as follows.
\begin{mydef} \label{def_joint_optimal_sub}
  We say that policies $\boldsymbol{\mathcal P}_{k:N}^\ast$ and estimates $\hat x_{k:N}^\ast$ are \textit{jointly optimal} for \eqref{eq_cost_functional_two-player_stopping_02} if they achieve the global minimum.
\end{mydef}

\begin{mydef} \label{def_pbp_optimal_sub}
  We say that policies $\boldsymbol{\mathcal P}_{k:N}^\ast$ and estimates $\hat x_{k:N}^\ast$ are \textit{person-by-person optimal} for \eqref{eq_cost_functional_two-player_stopping_02} if the following relations hold:
  \begin{align*}
    \mathbb E_{\mathbf x_k} \left[ J_k \left(\mathbf x_k, \boldsymbol{\mathcal P}_{k:N}^\ast, \hat x_{k:N}^\ast \right) \right] &= \min_{\boldsymbol{\mathcal P}_{k:N}} \mathbb E_{\mathbf x_{k}} \left[ J_k\left(\mathbf x_k, \boldsymbol{\mathcal P}_{k:N}, \hat x_{k:N}^\ast \right) \right] \nonumber \\ 
    & = \min_{\hat x_{k:N}} \mathbb E_{\mathbf x_k} \left[ J_k \left(\mathbf x_k, \boldsymbol{\mathcal P}_{k:N}^\ast, \hat x_{k:N} \right) \right]
  \end{align*}
\end{mydef}

In what follows, we define \textit{best response mappings} $\boldsymbol{\mathfrak P}$ and $\mathfrak X$.

\begin{mydef} \label{def_best_response_mapping_P}
  Given estimates $\hat x_{k:N}$, we define $\boldsymbol{\mathfrak P} \left( \hat x_{k:N} \right)$ as the collection of policies $\boldsymbol{\mathcal P}_{k:N}$ satisfying
  \begin{align*}
    \mathbb E_{\mathbf x_k} \left[ J_k \left(\mathbf x_k, \boldsymbol{\mathcal P}_{k:N}, \hat x_{k:N} \right) \right] = \min_{\boldsymbol{\mathcal P}_{k:N}'} \mathbb E_{\mathbf x_{k}} \left[ J_k\left(\mathbf x_k, \boldsymbol{\mathcal P}_{k:N}', \hat x_{k:N} \right) \right]
  \end{align*}
\end{mydef}

\begin{mydef} \label{def_best_response_mapping_X}
  Given policies $\boldsymbol{\mathcal P}_{k:N}$, we define $\mathfrak X \left( \boldsymbol{\mathcal P}_{k:N} \right)$ as the collection of estimates $\hat x_{k:N}$ satisfying
  \begin{align*}
    \mathbb E_{\mathbf x_k} \left[ J_k \left(\mathbf x_k, \boldsymbol{\mathcal P}_{k:N}, \hat x_{k:N} \right) \right] = \min_{\hat x_{k:N}'} \mathbb E_{\mathbf x_{k}} \left[ J_k\left(\mathbf x_k, \boldsymbol{\mathcal P}_{k:N}, \hat x_{k:N}' \right) \right]
  \end{align*}
\end{mydef}

\begin{mydef} \label{def_degeneracy}
  Policies $\boldsymbol{\mathcal P}_{k:N}$ are said to be \textit{degenerate} if there exists $j_0$ in $\{k, \cdots, N\}$ for which it holds that
  \begin{align} \label{eq_def_degeneracy_01}
    \mathbb P \left( \mathbf R_{j_0}=0 \,\Big|\, \mathbf R_k=0, \cdots, \mathbf R_{j_0-1}=0 \right) = 0
  \end{align}
\end{mydef}

\begin{remark} \label{remark_degeneracy}
  Let $\boldsymbol{\mathcal P}_{k:N}$ be degenerate policies such that \eqref{eq_def_degeneracy_01} holds for $j_0$ in $\{k, \cdots, N\}$. From \eqref{eq_cost-to-go_01}, we can derive that
  \begin{align*}
    \mathbb E_{\mathbf x_{j_0}} \left[ J_{j_0} \left(\mathbf x_{j_0}, \boldsymbol{\mathcal P}_{j_0:N}, \hat x_{j_0:N} \right) \,\Big|\, \mathbf R_k=0, \cdots, \mathbf R_{j_0-1}=0 \right] = c_{j_0}'
  \end{align*}
  \normalsize
  from which we can infer that the cost \eqref{eq_cost_functional_two-player_stopping_02} does not depend on the choice of estimates $\hat x_{j_0:N}$.
\end{remark}

\begin{prop} \label{prop_best_response_mapping_P}
  Suppose that non-degenerate policies $\boldsymbol{\mathcal P}_{k:N}$ and estimates $\hat x_{k:N}$ are given. The policies $\boldsymbol{\mathcal P}_{k:N}$ belong to $\boldsymbol{\mathfrak P} \left( \hat x_{k:N} \right)$ if and only if the following holds for all $j$ in $\{k, \cdots, N\}$:
  \begin{align} \label{eq_prop_best_response_mapping_P_01}
    &\mathbb E_{\mathbf x_j} \left[J_j \left( \mathbf x_j, \boldsymbol{\mathcal P}_{j:N}, \hat x_{j:N} \right) \,\Big|\, \mathbf R_k=0, \cdots, \mathbf R_{j-1}=0 \right] \nonumber \\
    &= \mathbb E_{\mathbf x_j} \left[J_j^\ast \left( \mathbf x_j, \hat{x}_{j:N} \right) \,\Big|\, \mathbf R_k=0, \cdots, \mathbf R_{j-1}=0 \right]
  \end{align}
  where for each $j$ in $\{k, \cdots, N\}$,
  \begin{align} \label{eq_cost-to-go_02}
    &J_j^\ast \left( x_j, \hat{x}_{j:N} \right) = \min \left\{ d^2 \left( x_j, \hat x_j \right) + \mathbb E_{\mathbf x_{j+1}} \left[ J_{j+1}^\ast \left( \mathbf x_{j+1}, \hat x_{j+1:N} \right) \,\Big|\, \mathbf x_j=x_j \right], c_j' \right\}
  \end{align}
  with $J_{N+1}^\ast = 0$
  \normalsize
\end{prop}
The proof follows from \eqref{eq_cost-to-go_01}, Definition \ref{def_best_response_mapping_P}, and the fact that
\begin{align*}
  \min_{\boldsymbol{\mathcal P}_{k:N}'} \mathbb E_{\mathbf x_{k}} \left[ J_k\left(\mathbf x_k, \boldsymbol{\mathcal P}_{k:N}', \hat x_{k:N} \right) \right] = \mathbb E_{\mathbf x_{k}} \left[ J_k^\ast \left(\mathbf x_k, \hat x_{k:N} \right) \right]
\end{align*}
We omit the detail for brevity.

\begin{cor} \label{cor_best_response_mapping_P}
  Given estimates $\hat x_{k:N}$, for each $j$ in $\{k, \cdots, N\}$, let us define sets $\overline{\mathbb D}_j$ and $\underline{\mathbb D}_j$ as follows:
  \begin{subequations} \label{eq_best_decision_sets}
    \begin{align} 
      \overline{\mathbb D}_j &= \left\{ x_j \in \mathbb R^2 \times [0, 2\pi) \,\Big|\, d^2 \left( x_j, \hat x_j \right) + \mathbb E_{\mathbf x_{j+1}} \left[ J_{j+1}^\ast \left( \mathbf x_{j+1}, \hat x_{j+1:N} \right) \,\Big|\, \mathbf x_j = x_j \right] \leq c_j' \right\} \label{eq_best_decision_sets_upper_limit} \\
      \underline{\mathbb D}_j &= \left\{ x_j \in \mathbb R^2 \times [0, 2\pi) \,\Big|\, d^2 \left( x_j, \hat x_j \right) + \mathbb E_{\mathbf x_{j+1}} \left[ J_{j+1}^\ast \left( \mathbf x_{j+1}, \hat x_{j+1:N} \right) \,\Big|\, \mathbf x_j = x_j \right] < c_j' \right\} \label{eq_best_decision_sets_lower_limit}
    \end{align}
  \end{subequations}
  Consider (deterministic) policies $\mathcal P_{k:N}$ defined by
  \begin{align} \label{eq_deterministic_policy}
    \mathcal P_j(x_j) = \begin{cases} 0 & \text{if } x_j \in \mathbb D_j \\
      1 & \text{otherwise} \end{cases}
  \end{align}
  for each $j$ in $\{k, \cdots, N\}$, where $\mathbb D_j$ is a measurable set satisfying $\underline{\mathbb D}_j \subseteq \mathbb D_j \subseteq \overline{\mathbb D}_j$. The policies $\mathcal P_{k:N}$ belong to $\boldsymbol{\mathfrak P} \left( \hat x_{k:N} \right)$.
\end{cor}

\begin{prop} \label{prop_best_response_mapping_X}
  Consider that non-degenerate policies $\boldsymbol{\mathcal P}_{k:N}$ and estimates $\hat x_{k:N}$ are given. The estimates $\hat x_{k:N}$ belong to $\mathfrak X \left( \boldsymbol{\mathcal P}_{k:N} \right)$ if and only if the following holds for all $j$ in $\{k, \cdots, N\}$:
  \begin{align*}
    &\mathbb{E}_{\mathbf{x}_j} \left[ d^2 \left( \mathbf x_j, \hat x_j \right) \,\Big|\, \mathbf R_k=0, \cdots, \mathbf R_j=0 \right] \nonumber \\
    	&= \min_{\hat x_j' \in \mathbb R^2 \times [0, 2\pi)  }\mathbb{E}_{\mathbf{x}_j} \left[ d^2 \left( \mathbf x_j, \hat x_j' \right) \,\Big|\, \mathbf R_k=0, \cdots, \mathbf R_j=0 \right]
  \end{align*}
\end{prop}
The proof follows from \eqref{eq_cost-to-go_01} and Definition \ref{def_best_response_mapping_X}. We omit the detail for brevity.

\begin{cor} \label{cor_best_response_mapping_X}
  Given non-degenerate policies $\boldsymbol{\mathcal P}_{k:N}$, for each $j$ in $\{k, \cdots, N\}$, let us consider an estimate ${\hat x_j = \begin{pmatrix} \hat p_{1,j} & \hat p_{2,j} & \hat \theta_j \end{pmatrix}^T}$ determined as follows:
  \begin{subequations} \label{eq_optimal_estimate}
    \begin{align} 
      \hat p_{1,j} &= \mathbb E \left[ \mathbf p_{1,j} \,\big|\, \mathbf R_k=0, \cdots, \mathbf R_j=0 \right] \\
      \hat p_{2,j} &= \mathbb E \left[ \mathbf p_{2,j} \,\big|\, \mathbf R_k=0, \cdots, \mathbf R_j=0 \right]
    \end{align}
  \end{subequations}
  and $\hat \theta_j$ takes a value in $[0, 2\pi)$ that satisfies
  \begin{subequations} \label{eq_optimal_estimate_02}
    \begin{align}  
      \sin \hat \theta_j &= \alpha^{-1} \cdot \mathbb E \left[ \sin \boldsymbol \theta_j \,\big|\, \mathbf R_k=0, \cdots, \mathbf R_j=0 \right] \\
      \cos \hat \theta_j &= \alpha^{-1} \cdot \mathbb E \left[ \cos \boldsymbol \theta_j \,\big|\, \mathbf R_k=0, \cdots, \mathbf R_j=0 \right]
    \end{align}
  \end{subequations}
  provided 
  \begin{align*}
    \alpha &= \mathbb E^2 \left[ \sin \boldsymbol \theta_j \,\big|\, \mathbf R_k=0, \cdots, \mathbf R_j=0 \right] + \mathbb E^2 \left[ \cos \boldsymbol \theta_j \,\big|\, \mathbf R_k=0, \cdots, \mathbf R_j=0 \right]
  \end{align*}
  \normalsize
  is non-zero; otherwise $\hat \theta_j$ takes any value in $[0, 2\pi)$.
  The estimates $\hat x_{k:N}$ belong to $\mathfrak X \left( \boldsymbol{\mathcal P}_{k:N} \right)$.
\end{cor}

\begin{prop} \label{prop_mapping_G_j_continuity}
  Consider functions $\left\{ \mathcal G_j \right\}_{j=k}^N$ defined as follows:\footnote{Note that $\mathcal G_j$ is a function defined on $\left( \mathbb R^2 \times [0, 2\pi) \right)^{N-j+2}$. See Appendix \ref{section_product_metric_space} for some remarks on the continuity of functions on a product space.} For each $j$ in $\{k, \cdots, N\}$,
  \begin{align} \label{eq_mapping_G_j}
    \mathcal G_j \left( x_{j-1}, \hat x_{j:N}\right) \overset{def}{=} \mathbb E_{\mathbf x_j} \left[ J_j^\ast \left( \mathbf x_j, \hat x_{j:N} \right) \,\Big|\, \mathbf x_{j-1} = x_{j-1} \right]
  \end{align}
  where $J_j^\ast$ is given in \eqref{eq_cost-to-go_02}. The functions $\left\{ \mathcal G_j \right\}_{j=k}^N$ are all continuous.
\end{prop}
The proof is given in Appendix \ref{proof_prop_mapping_G_j_continuity}. 

\subsection{Existence of a Jointly Optimal Solution} \label{section_existence_global_minimizer}
\begin{prop} \label{prop_optimal_solution_non-degeneracy}
  Let policies $\boldsymbol{\mathcal P}_{k:N}^\ast$ and estimates $\hat x_{k:N}^\ast$ are jointly optimal for \eqref{eq_cost_functional_two-player_stopping_02}. The policies $\boldsymbol{\mathcal P}_{k:N}^\ast$  are not degenerate in the sense of Definition \ref{def_degeneracy}.
\end{prop}
The proof is given in Appendix \ref{proof_section_2}.

\begin{thm} \label{thm_existence_minimizer}
  There exist policies $\boldsymbol{\mathcal P}_{k:N}^\ast$ and estimates $\hat x_{k:N}^\ast$ that are jointly optimal for \eqref{eq_cost_functional_two-player_stopping_02}.
\end{thm}

To prove Theorem \ref{thm_existence_minimizer}, we need the following lemma.
\begin{lemma} \label{lemma_optimal_solution_containment}
  Let us define
  \begin{align} \label{eq_mapping_G}
    \mathcal G \left( \hat x_{k:N} \right) &\overset{def}{=} \mathbb E_{\mathbf x_k} \left[ J_k^\ast \left( \mathbf x_k, \hat x_{k:N} \right) \right]
  \end{align}
  with the initial condition $\mathbf x_{k-1}=0$, where $J_k^\ast$ is defined in \eqref{eq_cost-to-go_02}. There exists a compact set ${\mathbb K \subset \left( \mathbb R^2 \times [0, 2\pi) \right)^{N-k+1}}$ for which the following holds for all $\hat x_{k:N}$ in $\left(\mathbb R^2 \times [0, 2\pi) \right)^{N-k+1}$:
  $$\inf_{\hat x_{k:N}' \in \mathbb K} \mathcal G \left( \hat x_{k:N}' \right) \leq \mathcal G \left( \hat x_{k:N} \right)$$
\end{lemma}
The proof is given in Appendix \ref{proof_section_2}.

\textit{Proof of Theorem \ref{thm_existence_minimizer}:}
Recall the definitions of $\mathcal G_k$ and $\mathcal G$ given in \eqref{eq_mapping_G_j} and \eqref{eq_mapping_G}, respectively. According to Proposition~\ref{prop_mapping_G_j_continuity} and by the fact that $\mathcal G \left( \hat x_{k:N} \right) = \mathcal G_k \left( 0, \hat x_{k:N}\right)$, we can see that $\mathcal G$ is a continuous function. Note that Lemma~\ref{lemma_optimal_solution_containment} implies that, if it exists, a global minimizer of $\mathcal G$ resides in a compact set. In what regards to finding a global minimizer, without loss of generality, we may assume that the domain of $\mathcal G$ is compact. Hence, by the continuity of $\mathcal G$ and compactness of its domain, there exist estimates $\hat x_{k:N}^\ast$ that achieve the global minimum of $\mathcal G$. 

Next, let us choose policies $\boldsymbol{\mathcal P}_{k:N}^\ast$ belonging to $\boldsymbol{\mathfrak P} \left( \hat x_{k:N}^\ast \right)$ using, for instance, Corollary \ref{cor_best_response_mapping_P}. By the definition of $\mathcal G$ given as in \eqref{eq_mapping_G} and by the fact that $\hat x_{k:N}^\ast$ is a global minimizer of $\mathcal G$, we conclude that the policies $\boldsymbol{\mathcal P}_{k:N}^\ast$ and the estimates $x_{k:N}^\ast$ are jointly optimal for \eqref{eq_cost_functional_two-player_stopping_02}. \hfill \QED


\subsection{Iterative Procedure for Finding a Person-by-Person Optimal Solution} \label{section_convergence_pbp_optimal_solutions}
As numerically illustrated in \cite{molin2012_ifac}, the function $\mathcal G$ in \eqref{eq_mapping_G} may be non-convex; consequently, finding a jointly optimal solution for \eqref{eq_cost_functional_two-player_stopping_02} would be computationally intractable. Instead, we seek a person-by-person optimal solution based on Procedure \ref{alg_solution} described below. In the procedure, $\eta$ is a pre-selected non-negative constant that determines a stopping criterion (Line 17), and the function $\mathcal G$ is defined in \eqref{eq_mapping_G}.

\SetAlFnt{\small}
\begin{procedure} \label{alg_solution}
  \LinesNumbered
  \DontPrintSemicolon
  \SetKwInOut{Input}{input}\SetKwInOut{Output}{output}
  \Input{$\eta \geq 0$, $\hat{x}_{k:N}^{(0)}$}
  \Output{$\boldsymbol{\mathcal P}_{k:N}^{(i+1)}, \hat{x}_{k:N}^{(i)}$}
  \Begin{
    $j \gets N$
    
    \While {$j \geq k$} {
      Choose $\boldsymbol{\mathcal P}_j^{(1)}$ according to Corollary \ref{cor_best_response_mapping_P} \newline using $\hat x_{k:N}^{(0)}$
      
      $j \gets j-1$
    }
  
    $i \gets 0$

    \Repeat {$\left| \mathcal{G} \left( \hat{x}_{k:N}^{(i)} \right) - \mathcal{G} \left( \hat{x}_{k:N}^{(i-1)} \right) \right| \leq \eta$} {

      $i \gets i+1$
      
      $j \gets k$

      \While {$j \leq N$} {
        Choose $\hat x_j^{(i)}$ according to Corollary \ref{cor_best_response_mapping_X} \newline using $\boldsymbol{\mathcal P}_{k:N}^{(i)}$

        $j \gets j+1$
      }      

      $j \gets N$

      \While {$j \geq k$} {
        Choose $\boldsymbol{\mathcal P}_j^{(i+1)}$ according to Corollary \ref{cor_best_response_mapping_P} \newline using $\hat x_{k:N}^{(i)}$

        $j \gets j-1$
      }
    }
  }  
  \caption{() Finding a Person-by-Person Optimal Solution}
\end{procedure}


Let $\left\{ \left( \boldsymbol{\mathcal P}_{k:N}^{(i)}, \hat{x}_{k:N}^{(i)} \right) \right\}_{i \in \mathbb N}$ be a sequence of solutions generated through repeated computations of policies and estimates by Procedure \ref{alg_solution} (Line $2-16$). In the rest of this section, we discuss convergence of the sequence to a person-by-person optimal solution. We first define convergence of policies and estimates. For notational convenience, we adopt the following: Let $\mathfrak B$ be a Borel $\sigma$-algebra on $\mathbb R^2 \times [0, 2\pi)$. Given $\left\{ \boldsymbol{\mathcal P}_{k:N}^{(i)} \right\}_{i \in \mathbb N}$ and $\boldsymbol{\mathcal P}_{k:N}$, let us define the following: For each $\mathbb A$ in $\mathfrak B$,
\begin{subequations}  \label{eq_probablity_measures}
  \begin{align}
    \mu_{j|j}^{(i)} \left( \mathbb A \right) &= \mathbb P \left( \mathbf x_j \in \mathbb A \,\Big|\, \mathbf R_k^{(i)}=0, \cdots, \mathbf R_j^{(i)}=0 \right) \\
    \mu_{j|j} \left( \mathbb A \right) &= \mathbb P \left( \mathbf x_j \in \mathbb A \,\Big|\, \mathbf R_k=0, \cdots, \mathbf R_j=0 \right)
  \end{align}
\end{subequations}
subject to
\begin{subequations} \label{eq_decision_variables_sequence_limit}
  \begin{align}
    \mathbf R_j^{(i)} &= \boldsymbol{\mathcal P}_j^{(i)} \left( \mathbf x_j  \right) \label{eq_decision_variables_sequence_limit_a} \\
    \mathbf R_j &= \boldsymbol{\mathcal P}_j \left( \mathbf x_j  \right) \label{eq_decision_variables_sequence_limit_b}
  \end{align}
\end{subequations}
for all $i$ in $\mathbb N$ and $j$ in $\{k, \cdots, N\}$.
\begin{mydef} \label{def_policy_convergence}
  Let $\left\{ \boldsymbol{\mathcal P}_{k:N}^{(i)} \right\}_{i \in \mathbb N}$ be a sequence of policies. We say that the sequence \textit{converges} to $\boldsymbol{\mathcal P}_{k:N}$ if the following hold for all $j$ in $\{k, \cdots, N\}$:
  \begin{subequations}
    \begin{align} \label{eq_policies_convergence_01}
      \mu_{j|j}^{(i)} \rightarrow \mu_{j|j}
    \end{align}
    and
    \begin{align} \label{eq_policies_convergence_02}
      \mathbb P \left( \mathbf R_j^{(i)}=0 \,\Big|\, \mathbf R_k^{(i)}=0, \cdots, \mathbf R_{j-1}^{(i)}=0 \right) \rightarrow \mathbb P \left( \mathbf R_j=0 \,\Big|\, \mathbf R_k=0, \cdots, \mathbf R_{j-1}=0 \right)
    \end{align} 
  \end{subequations}
  subject to \eqref{eq_decision_variables_sequence_limit}.\footnote{Equation \eqref{eq_policies_convergence_01} implies that the sequence of the probability measures $\left\{ \mu_{j|j}^{(i)}\right\}_{i \in \mathbb N}$ \textit{converges} to the probability measure $\mu_{j|j}$. See Chapter 9.3 of \cite{dudley_9780511755347} for the definition of convergence of probability measures.} We denote the convergence by ${\boldsymbol{\mathcal P}_{k:N}^{(i)} \Rightarrow \boldsymbol{\mathcal P}_{k:N}}$. In addition, we say that two sets of policies $\boldsymbol{\mathcal P}_{k:N}$ and $\boldsymbol{\mathcal P}_{k:N}'$ are \textit{equal} if the following hold for all $j$ in $\{k, \cdots, N\}$:
  \begin{subequations} \label{eq_equivalent_policies}
    \begin{align}
      \mu_{j|j} = \mu_{j|j}'
    \end{align} 
    and
    \begin{align}
      \mathbb P \left( \mathbf R_j=0 \,\Big|\, \mathbf R_k=0, \cdots, \mathbf R_{j-1}=0 \right) = \mathbb P \left( \mathbf R_j'=0 \,\Big|\, \mathbf R_k'=0, \cdots, \mathbf R_{j-1}'=0 \right)
    \end{align} 
  \end{subequations}
  subject to $\mathbf R_j = \boldsymbol{\mathcal P}_j \left( \mathbf x_j  \right)$ and $\mathbf R_j' = \boldsymbol{\mathcal P}_j' \left( \mathbf x_j  \right)$ for all $j$ in $\{k, \cdots, N\}$.
\end{mydef}

\begin{remark} [Uniqueness of the Limit of Policies]
  Suppose that a sequence of policies $\left\{ \boldsymbol{\mathcal P}_{k:N}^{(i)} \right\}_{i \in \mathbb N}$ converges to both $\boldsymbol{\mathcal P}_{k:N}$ and $\boldsymbol{\mathcal P}_{k:N}'$. Then the two sets of the policies $\boldsymbol{\mathcal P}_{k:N}$ and $\boldsymbol{\mathcal P}_{k:N}'$ are equal. To see this, using the definition of convergence of probability measures, we can derive that
  \begin{align}
    \int_{\mathbb R^2 \times [0, 2\pi)} g \,\mathrm d\mu_{j|j} = \int_{\mathbb R^2 \times [0, 2\pi)} g \,\mathrm d\mu_{j|j}'
  \end{align}
  for every bounded, continuous function ${g: \mathbb R^2 \times [0, 2\pi) \to \mathbb R}$. Based on Lemma 9.3.2 in \cite{dudley_9780511755347}, we can see that \eqref{eq_equivalent_policies} holds for all $j$ in $\{k, \cdots, N\}$.
\end{remark}

\begin{mydef} \label{def_estimate_convergence}
  Let $\left\{ \hat x_{k:N}^{(i)} \right\}_{i \in \mathbb N}$ be a sequence of estimates. We say that the sequence \textit{converges} to $\hat x_{k:N}$ if the following holds for all $j$ in $\{k, \cdots, N\}$:
  \begin{align}
    \lim_{i \to \infty} d\left( \hat x_j^{(i)}, \hat x_j \right) = 0
  \end{align}
  We denote the convergence by $\hat x_{k:N}^{(i)} \Rightarrow \hat x_{k:N}$.
  In addition, we say that two sets of estimates $\hat x_{k:N}$ and $\hat x_{k:N}'$ are \textit{equal} if the following holds for all $j$ in $\{k, \cdots, N\}$:
  \begin{align}
    d\left( \hat x_j, \hat x_j'\right)=0
  \end{align}
\end{mydef}

\begin{mydef}
  Let $\left\{ \boldsymbol{\mathcal P}_{k:N}^{(i)}\right\}_{i \in \mathbb N}$ be a sequence of policies. We say that the policies are \textit{strictly non-degenerate} if there exists a positive constant $\epsilon$ for which the following holds for all $i$ in $\mathbb{N}$ and $j$ in $\{k, \cdots, N\}$:
  \begin{align} \label{eq_decision_sets_non-degeneracy}
    \mathbb P \left( \mathbf R_j^{(i)}=0 \,\Big|\, \mathbf R_k^{(i)}=0, \cdots, \mathbf R_{j-1}^{(i)}=0 \right) \geq \epsilon
  \end{align}
  subject to \eqref{eq_decision_variables_sequence_limit_a}.
\end{mydef}

Recall that the sequence of solutions $\left\{ \left( \boldsymbol{\mathcal P}_{k:N}^{(i)}, \hat{x}_{k:N}^{(i)} \right) \right\}_{i \in \mathbb N}$ generated by Procedure \ref{alg_solution} satisfies the following for all $i$ in $\mathbb N$:
\begin{subequations} \label{eq_best_response_alternation}
  \begin{align}
    \boldsymbol{\mathcal P}_{k:N}^{(i)} &\in \boldsymbol{\mathfrak P}\left( \hat x_{k:N}^{(i-1)} \right) \label{eq_best_response_alternation_a} \\
    \hat x_{k:N}^{(i)} &\in \mathfrak X \left( \boldsymbol{\mathcal P}_{k:N}^{(i)} \right) \label{eq_best_response_alternation_b}
  \end{align}
\end{subequations}
The following theorem states convergence of the sequence to a person-by-person optimal solution.

\begin{thm} \label{thm_procedure_convergence}
  Consider a sequence of solutions $\left\{ \left( \boldsymbol{\mathcal P}_{k:N}^{(i)}, \hat{x}_{k:N}^{(i)} \right) \right\}_{i \in \mathbb N}$ satisfying \eqref{eq_best_response_alternation}. Suppose that the policies $\left\{ \boldsymbol{\mathcal P}_{k:N}^{(i)}\right\}_{i \in \mathbb N}$ are strictly non-degenerate. Then, the sequence has a convergent subsequence, and the limit of any convergent subsequence is a person-by-person optimal solution.
\end{thm}

To prove Theorem \ref{thm_procedure_convergence}, we need the following three lemmas.


\begin{lemma} \label{lemma_bounded_estimates_s}
  Consider a sequence of solutions $\left\{ \left( \boldsymbol{\mathcal P}_{k:N}^{(i)}, \hat{x}_{k:N}^{(i)} \right) \right\}_{i \in \mathbb N}$ satisfying \eqref{eq_best_response_alternation}. Suppose that the policies $\left\{ \boldsymbol{\mathcal P}_{k:N}^{(i)}\right\}_{i \in \mathbb N}$ are strictly non-degenerate. Then the sequence $\left\{ \hat x_j^{(i)} \right\}_{i \in \mathbb N}$ is bounded for all $j$ in $\{k, \cdots, N\}$.
\end{lemma}
Lemma \ref{lemma_bounded_estimates_s} is a special case of Lemma \ref{lemma_bounded_estimates} given in Appendix \ref{proof_section_2}.

\begin{lemma} \label{lemma_estimates_convergence_spp}
  Consider a sequence of solutions $\left\{ \left( \boldsymbol{\mathcal P}_{k:N}^{(i)}, \hat{x}_{k:N}^{(i)} \right) \right\}_{i \in \mathbb N}$ satisfying \eqref{eq_best_response_alternation}. Suppose that for an infinite subset $\left\{ i_l \right\}_{l \in \mathbb N}$ of $\mathbb N$, the following hold: $\boldsymbol{\mathcal P}_{k:N}^{(i_l)} \Rightarrow \boldsymbol{\mathcal P}_{k:N}$, $\hat x_{k:N}^{(i_l)} \Rightarrow \hat x_{k:N}$, and $\hat x_{k:N}^{(i_l-1)} \Rightarrow \hat x_{k:N}'$. Then the estimates $\hat x_{k:N}$ belong to $\mathfrak X \left( \boldsymbol{\mathcal P}_{k:N} \right)$.
\end{lemma}

\begin{lemma} \label{lemma_P_closed}
  Consider a sequence of solutions $\left\{ \left( \boldsymbol{\mathcal P}_{k:N}^{(i)}, \hat{x}_{k:N}^{(i)} \right) \right\}_{i \in \mathbb N}$ satisfying \eqref{eq_best_response_alternation}. Suppose that the policies $\left\{ \boldsymbol{\mathcal P}_{k:N}^{(i)} \right\}_{i \in \mathbb N}$ are strictly non-degenerate and that for an infinite subset $\left\{ i_l \right\}_{l \in \mathbb N}$ of $\mathbb N$, it holds that $\hat x_{k:N}^{(i_l-1)} \Rightarrow \hat x_{k:N}'$. Then, the sequence $\left\{ \boldsymbol{\mathcal P}_{k:N}^{(i_l)} \right\}_{l \in \mathbb N}$ has a convergent subsequence, and the limit $\boldsymbol{\mathcal P}_{k:N}$ of any convergent subsequence belongs to $\boldsymbol{\mathfrak P} \left( \hat x_{k:N}' \right)$.
\end{lemma}
The proofs of Lemmas \ref{lemma_estimates_convergence_spp} and \ref{lemma_P_closed} are given in Appendix \ref{proof_lemma_P_closed}.

\textit{Proof of Theorem \ref{thm_procedure_convergence}:} 
We first note that according to Lemma \ref{lemma_bounded_estimates_s}, the sequence $\left\{ \hat x_j^{(i)} \right\}_{i \in \mathbb{N}}$ is contained in a compact set for every $j$ in $\{k, \cdots, N\}$.\footnote{The metric space $\left( \mathbb R^2 \times [0, 2\pi), d\right)$ is proper; hence for any bounded subset of $\mathbb R^2 \times [0, 2\pi)$, we can find a compact set that contains the subset.} Hence, by the compactness, there exists an infinite subset $\mathbb I$ of $\mathbb N$ for which the subsequences $\left\{ \hat x_{k:N}^{(i)} \right\}_{i \in \mathbb I}$ and $\left\{ \hat x_{k:N}^{(i-1)} \right\}_{i \in \mathbb I}$ are both convergent. Let $\hat x_{k:N}$ and $\hat x_{k:N}'$ be the respective limits of the subsequences. Also, according to Lemma \ref{lemma_P_closed}, there is an infinite subset $\mathbb I'$ of $\mathbb I$ for which the subsequence $\left\{ \boldsymbol{\mathcal P}_{k:N}^{(i)} \right\}_{i \in \mathbb I'}$ is convergent. Let $\boldsymbol{\mathcal P}_{k:N}$ be the limit of this subsequence.

To complete the proof, it remains to show that $\boldsymbol{\mathcal P}_{k:N}$ and $\hat x_{k:N}$ constitute a person-by-person optimal solution, i.e., it holds that
\begin{subequations}
  \begin{align}
    \boldsymbol{\mathcal P}_{k:N} &\in \boldsymbol{\mathfrak P} \left( \hat x_{k:N}\right) \label{eq_pbp_requirement_01} \\
    \hat x_{k:N} &\in \mathfrak X \left( \boldsymbol{\mathcal P}_{k:N} \right) \label{eq_pbp_requirement_02}
  \end{align}
\end{subequations}
Equation \eqref{eq_pbp_requirement_02} is ensured by Lemma \ref{lemma_estimates_convergence_spp}, hence it remains to show that \eqref{eq_pbp_requirement_01} is true. 

By contradiction, suppose that the policies $\boldsymbol{\mathcal P}_{k:N}$ do not belong to $\boldsymbol{\mathfrak P} \left( \hat x_{k:N}\right)$. Note that by Lemma~\ref{lemma_P_closed}, $\boldsymbol{\mathcal P}_{k:N}$ belong to $\boldsymbol{\mathfrak P} \left( \hat x_{k:N}' \right)$. We can see that the following relations hold for any policies $\boldsymbol{\mathcal P}_{k:N}'$ belonging to $\boldsymbol{\mathfrak P} \left( \hat x_{k:N}\right)$:
\begin{align} \label{eq_proof_thm_procedure_convergence_01}
  \mathcal G \left( \hat x_{k:N} \right) &= \mathbb E_{\mathbf x_k} \left[ J_k \left(\mathbf x_k, \boldsymbol{\mathcal P}_{k:N}', \hat x_{k:N} \right) \right] \nonumber \\
                                         &\overset{\mathbf{(i)}}{<} \mathbb E_{\mathbf x_k} \left[ J_k \left(\mathbf x_k, \boldsymbol{\mathcal P}_{k:N}, \hat x_{k:N} \right) \right] \nonumber \\
                                         &\overset{\mathbf{(ii)}}{\leq} \mathbb E_{\mathbf x_k} \left[ J_k \left(\mathbf x_k, \boldsymbol{\mathcal P}_{k:N}, \hat x_{k:N}' \right) \right] = \mathcal G \left( \hat x_{k:N}' \right)
\end{align}
$\mathbf{(i)}$ follows from the hypothesis that $\boldsymbol{\mathcal P}_{k:N} \notin \boldsymbol{\mathfrak P} \left( \hat x_{k:N}\right)$; and $\mathbf{(ii)}$ is due to \eqref{eq_pbp_requirement_02}. On the other hand, since $\mathcal G$ is non-negative and decreasing along the sequence $\left\{ \hat x_{k:N}^{(i)} \right\}_{i \in \mathbb N}$, i.e., $\mathcal G \left( \hat x_{k:N}^{(i+1)} \right) \leq \mathcal G \left( \hat x_{k:N}^{(i)} \right)$ holds for all $i$ in $\mathbb N$, it holds that $\lim_{i \to \infty} \mathcal G \left( \hat x_{k:N}^{(i)} \right) = \alpha$ for some real number $\alpha$. In conjunction with the continuity of $\mathcal G$ (see Proposition \ref{prop_mapping_G_j_continuity}), this implies that $\mathcal G \left( \hat x_{k:N} \right) = \mathcal G \left( \hat x_{k:N}' \right) = \alpha$ which contradicts \eqref{eq_proof_thm_procedure_convergence_01}. Therefore we conclude that the policies $\boldsymbol{\mathcal P}_{k:N}$ belong to $\boldsymbol{\mathfrak P} \left( \hat x_{k:N} \right)$. \hfill \QED

\section{Application to Tracking of Animal Movements} \label{section_experimental_results}
In this section, we apply our results to estimation of animal movements over a costly communication link where the performance of the optimal scheme are illustrated using GPS data collected from a monitoring device mounted on an African buffalo.\footnote{The development and deployment of animal-borne monitoring devices were performed under a research grant NSF ECCS 1135726. The GPS data were collected at the Gorongosa National Park, Mozambique.} Fig. \ref{figure_gps} shows the GPS track of the buffalo. To represent the movement of the buffalo, as described in \cite{mcclintock2012_em}, we adopt the SPP model \eqref{eq_spp_model} in which $\mathbf v_k$ and $\boldsymbol{\phi}_k$ are the Weibull and Wrapped Cauchy random processes, respectively. Note that the probability density functions of $\mathbf v_k$ and $\boldsymbol{\phi}_k$ are given as follows:
\begin{subequations}  \label{eq_pdf}
  \begin{align}
    f_{\mathbf v_k} (v) &=  \frac{a_\mathbf{v}}{s_\mathbf{v}} \left( \frac{v}{s_\mathbf{v}} \right)^{a_\mathbf{v}-1} e^{-\left( \frac{v}{s_\mathbf{v}} \right)^{a_\mathbf{v}}}, \text{ for } v \geq 0 \\
    f_{\boldsymbol{\phi}_k} (\phi) &=  \frac{1}{2\pi} \cdot \frac{1-a_{\boldsymbol{\phi}}^2}{ 1+a_{\boldsymbol{\phi}}^2 - 2 a_{\boldsymbol{\phi}} \cos\left(\phi-m_{\boldsymbol{\phi}} \right)}
  \end{align}
\end{subequations}
Using the collected GPS data, we compute the maximum likelihood estimates of the parameters for \eqref{eq_pdf} as follows:
\begin{subequations} \label{eq_pdf_params}
  \begin{align}
    (a_{\mathbf v}, s_{\mathbf v}) &= (1.35, 4.66) \\
    (a_{\boldsymbol{\phi}}, m_{\boldsymbol{\phi}}) &= (0.65, 0.00)
  \end{align}
\end{subequations}
The graphs in Fig. \ref{figure_pdf_fitting} show comparisons between the resulting probability density functions and the histograms obtained from the GPS data.

We have selected the communication costs $c_k=10$ for all $k$ in $\{1, \cdots, N\}$ and the length of the time-horizon ${N=100}$. Using Procedure \ref{alg_solution}, \eqref{eq_solutions}, and \eqref{eq_remark_solution_transformation_03}, we have found the optimal remote estimation scheme where Fig. \ref{figure_estimation_error_ch03} illustrates the performance of the scheme in terms of the state estimation distortion computed by the metric $d\left( x_k, \hat x_k \right)$. Note that the (red) circles on the time axis ($x$-axis) represents the time steps at which the sensing unit transmitted information on the full state $x_k$ to the estimator, and the state estimate $\hat x_k$ was set to $\hat x_k=x_k$ (hence $d\left( x_k, \hat x_k \right)=0$). Our experimental results show that the optimal scheme achieved the error of location estimation less than $5$ meters compared to the total traveled distance of $372.53$ meters; and the information transmissions occurred $32$ times over $100$ time steps.

\begin{figure}
  \centering
  \includegraphics[scale=0.25]{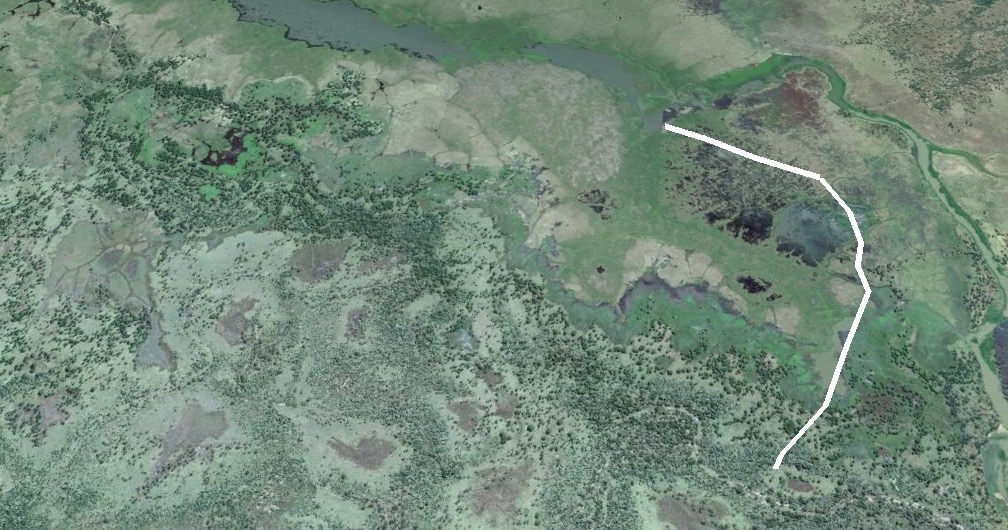}
  \caption {A screenshot of a GPS track (the white trajectory) of an African buffalo in the Google Earth.}
  \label{figure_gps}
\end{figure}

\begin{figure}
  \centering
  \includegraphics[scale=.25]{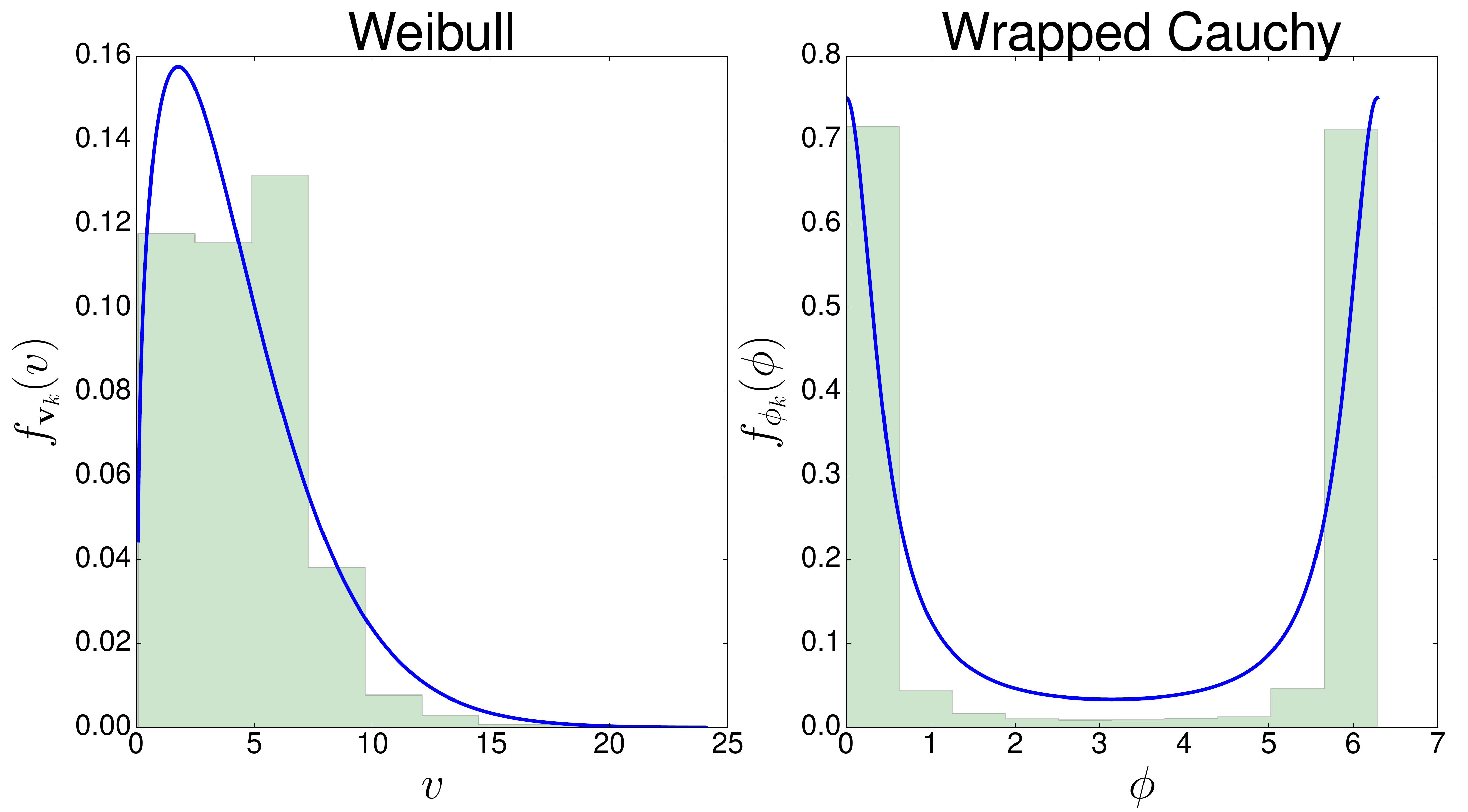}
  \caption {Comparisons between the probability density functions of $\mathbf v_k$ and $\boldsymbol{\phi}_k$ under \eqref{eq_pdf_params} and the histograms obtained from the GPS data.}
  \label{figure_pdf_fitting}
\end{figure}

\begin{figure}
  \centering
  \includegraphics[scale=0.35]{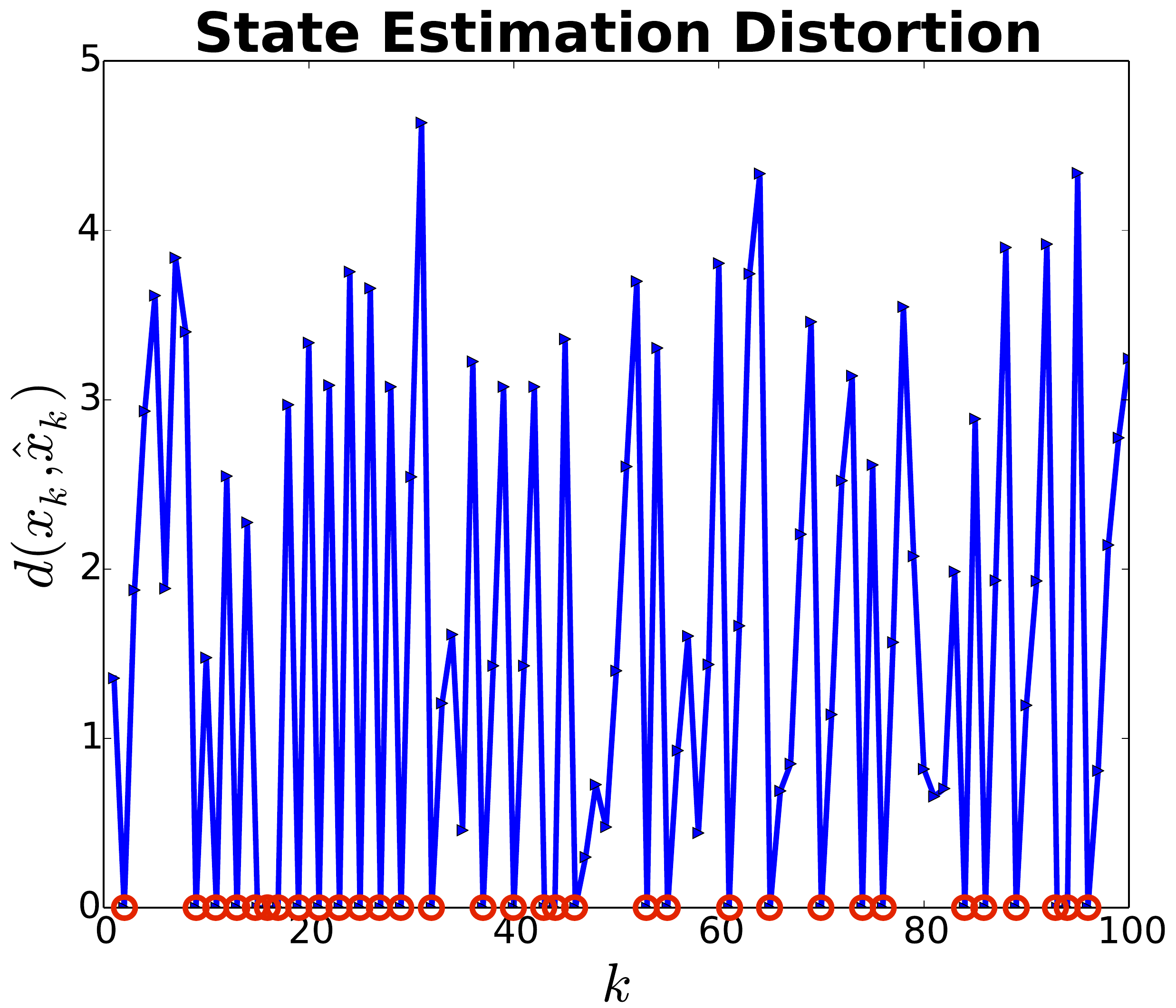}
  \caption {State estimation distortion of the optimal remote estimation scheme}
  \label{figure_estimation_error_ch03}
\end{figure}


\newpage
\begin{appendix}
 \renewcommand{\themydef}{\thesubsection.\arabic{mydef}}
 
   \subsection{On Product Metric Space} \label{section_product_metric_space}
 Given a metric space $\left( \mathbb X, d \right)$, we define a metric $\overline{d}$ on the product space $\mathbb X^k$ as follows: For $x_{1:k} = ( x_{1}, ~ \cdots, ~ x_{k} )$ and $y_{1:k} = ( y_{1}, ~ \cdots, ~ y_{k} )$ in $\mathbb X^k$,
 \begin{align}
   \overline{d}(x_{1:k}, y_{1:k}) = \left[d^2\left( x_{1}, y_{1} \right) + \cdots + d^2\left( x_{k}, y_{k} \right) \right]^{1/2}
 \end{align}
 Note that $\left( \mathbb X^k, \overline{d} \right)$ is a (product) metric space. For $\left( \mathbb X, d \right)$ and $\left( \mathbb X^k, \overline{d} \right)$, the following are true:
 \begin{enumerate} [label=\bfseries (F\arabic*)]
 \item Let $\left\{ \mathbb K_j \right\}_{j=1}^k$ be a collection of compact subsets of $\mathbb X$. The product set $\mathbb K_1 \times \cdots \times \mathbb K_k$ is a compact subset of $\mathbb X^k$.

 \item Consider a sequence $\left\{ x_{1:k}^{(i)} \right\}_{i \in \mathbb N}$ in $\mathbb X^k$. The sequence converges to $x_{1:k}$ in $\mathbb X^k$, i.e., $$\lim_{i \to \infty} \overline{d} \left( x_{1:k}^{(i)}, x_{1:k}\right) = 0$$ if and only if $\left\{ x_j^{(i)} \right\}_{i \in \mathbb N}$ converges to $x_j$ for all $j$ in $\{1, \cdots, k\}$, i.e., $$\lim_{i \to \infty} d \left( x_j^{(i)}, x_j\right) = 0$$ holds for all $j$ in $\{1, \cdots, k\}$. \label{enum_sequence_convergence}
 \end{enumerate}
 
 As a consequence of \ref{enum_sequence_convergence}, a function $\mathcal G: \mathbb X^k \to \mathbb R$ is continuous at $x_{1:k} \in \mathbb X^k$ if for any sequence $\left\{ x_{1:k}^{(i)} \right\}_{i \in \mathbb N}$ for which $\lim_{i \to \infty} d\left( x_j^{(i)}, x_j \right) = 0$ holds for all $j$ in $\{1, \cdots, k\}$, it holds that $\lim_{i \to \infty} \left| \mathcal G \left( x_{1:k}^{(i)} \right) - \mathcal G \left( x_{1:k} \right) \right| = 0$.

 \subsection{On Randomized Policies} \label{remark_randomized_transmission_policies}
Let $\left( \boldsymbol \tau_k, \mathbf x_{\boldsymbol \tau_k}, \mathbf x_k\right) \mapsto \boldsymbol{\mathcal T}_k\left( \boldsymbol \tau_k, \mathbf x_{\boldsymbol \tau_k}, \mathbf x_k\right)$ be a randomized transmission policy defined in Section \ref{section_formulation} that dictates the random variable $\mathbf R_k$ as in \eqref{eq_policy_rule_02}. Given a realization $\left(\tau_k, x_{\tau_k}, x_k \right)$ of $\left( \boldsymbol{\tau}_k, \mathbf x_{\boldsymbol{\tau}_k}, \mathbf x_k \right)$, the variable $\mathbf R_k$ satisfies
 \begin{align*}
   \mathbf R_k = \begin{cases} 0 & \text{with probability } \mathbb P \left( \boldsymbol{\mathcal T}_k \left( \boldsymbol{\tau}_k, \boldsymbol{x}_{\boldsymbol{\tau}_k}, \mathbf x_k \right) = 0 \,\Big|\, \boldsymbol{\tau}_k = \tau_k, \boldsymbol{x}_{\boldsymbol{\tau}_k} = x_{\tau_k}, \mathbf x_k = x_k \right) \\ 1 & \text{with probability } \mathbb P \left( \boldsymbol{\mathcal T}_k \left( \boldsymbol{\tau}_k, \boldsymbol{x}_{\boldsymbol{\tau}_k}, \mathbf x_k \right) = 1 \,\Big|\, \boldsymbol{\tau}_k = \tau_k, \boldsymbol{x}_{\boldsymbol{\tau}_k} = x_{\tau_k}, \mathbf x_k = x_k \right) \end{cases}
 \end{align*}

Let $\mathbf x_j \mapsto \boldsymbol{\mathcal P}_j \left(\mathbf x_j \right)$ be a randomized policy defined in Section \ref{section_two-player_stopping_problem} that dictates the random variable $\mathbf R_j$ as in \eqref{eq_decision_variables_subproblem}. Given a realization $x_j$ of $\mathbf x_j$, the variable $\mathbf R_j$ satisfies
 \begin{align*}
   \mathbf R_j = \begin{cases} 0 & \text{with probability } \mathbb P \left( \boldsymbol{\mathcal P}_j \left( \mathbf x_j \right) = 0 \,\Big|\, \mathbf x_j = x_j \right) \\ 
     1 & \text{with probability } \mathbb P \left( \boldsymbol{\mathcal P}_j \left( \mathbf x_j \right) = 1 \,\Big|\, \mathbf x_j = x_j  \right) \end{cases}
 \end{align*}

 Throughout the work, we restrict our attention to the policies in which 
 $$\mathbb P \left( \boldsymbol{\mathcal P}_j \left( \mathbf x_j \right) = 0 \,\Big|\, \mathbf x_j = x_j  \right)$$
 is a measurable function of $x_j$ on the measurable space $\left( \mathbb R^2 \times [0, 2\pi) , \mathfrak B \right)$ where $\mathfrak B$ is a Borel $\sigma$-algebra on $\mathbb R^2 \times [0, 2\pi)$. As a case in point, consider a (deterministic) policy defined by 
 $$\boldsymbol{\mathcal P}_j(\mathbf x_j) = \begin{cases} 0 & \text{if }\mathbf x_j \in \mathbb D_j \\ 1 & \text{otherwise} \end{cases}$$
 where $\mathbb D_j \in \mathfrak B$. It can be verified that 
 \begin{align*}
   \mathbb P \left( \boldsymbol{\mathcal P}_j \left( \mathbf x_j \right) = 0 \,\Big|\, \mathbf x_j = x_j \right) = \begin{cases} 0 & \text{if } x_j \in \mathbb D_j \\ 1 & \text{otherwise} \end{cases}
 \end{align*}
 is a measurable function of $x_j$.
 
 \subsection{Preliminary Concepts in Probability Theory}
 We first review some of key definitions and results from probability theory \cite{billingsley1995_wiley, dudley_9780511755347}. Let $\left( \mathbb X, d\right)$ be a complete separable metric space, and let $\mathcal T$ and $\mathfrak B$ be a topology and a Borel $\sigma$-algebra derived from the metric, respectively.

 \begin{mydef} \label{def_tightness}
   Let $\mu$ be a probability measure on $\left(\mathbb X, \mathfrak B \right)$. The probability measure is said to be \textit{tight} if for every positive constant $\epsilon$, there exists a compact subset $\mathbb K$ of $\left( \mathbb X, \mathcal T \right)$ for which $\mu \left( \mathbb K \right) > 1 - \epsilon$ holds.
 \end{mydef}

 The following is adopted from Theorem 7.1.4 in \cite{dudley_9780511755347}.
 \begin{lemma} \label{lemma_tightness}
   Any probability measure $\mu$ on $\left( \mathbb X, \mathfrak B\right)$ is tight.
 \end{lemma}

 \begin{mydef} \label{def_closed_regularity}
   A probability measure $\mu$ defined on $\left(\mathbb X, \mathfrak B \right)$ is said to be \textit{closed regular} if for every $\mathbb A$ in $\mathfrak B$, it holds that
   \begin{align}
     \mu(\mathbb A) = \sup \left\{ \mu(\mathbb F) \,\big|\, \mathbb F \in \mathfrak B \text{ closed}, \mathbb F \subset \mathbb A \right\}
   \end{align}
 \end{mydef}

 From Theorem 7.1.3 in \cite{dudley_9780511755347}, we can state the following Lemma.
 \begin{lemma} \label{lemma_closed_regularity}
   Any probability measure $\mu$ on $\left( \mathbb X, \mathfrak B\right)$ is closed regular.
 \end{lemma}

 \begin{remark} \label{remark_outer_regularity}
   Let $\mu$ be a closed regular probability measure defined on $\left( \mathbb X, \mathfrak B\right)$ and let $\mathbb A$ be a measurable subset in $\mathfrak B$. For every positive constant $\epsilon$, there exists a closed set $\mathbb F$ for which $\mathbb F \subset \mathbb X \setminus \mathbb A$ and $\mu \left( \mathbb X \setminus \mathbb A \right) < \mu \left( \mathbb F \right) + \epsilon$. Let us define an open set $\mathbb O = \mathbb X \setminus \mathbb F$. We can see that $\mathbb O$ satisfies $\mathbb O \supset \mathbb A$ and $\mu \left( \mathbb O \right) < \mu \left( \mathbb A \right) + \epsilon$. Hence we conclude that
   \begin{align}
     \mu \left( \mathbb A \right) = \inf \left\{ \mu \left( \mathbb O \right) \,\big|\, \mathbb O \in \mathfrak B \text{ open}, \mathbb O \supset \mathbb A \right\}
   \end{align}
 \end{remark}

 \begin{mydef} [Convergence of Probability Measures] \label{def_probability_measures_convergence}
   Let $\left\{ \mu^{(i)} \right\}_{i \in \mathbb N}$ and $\mu$ be a sequence of probability measures and a probability measure defined on $\left(\mathbb X, \mathfrak B \right)$, respectively, and let $\mathcal C_b \left( \mathbb X \right)$ be the set of all bounded, continuous, real-valued functions on $\mathbb X$. The sequence is said to \textit{converge} to $\mu$ if it holds that 
   $$\lim_{i \to \infty} \int g \,\mathrm d\mu^{(i)} = \int g \,\mathrm d\mu$$
   for every $g$ in $\mathcal C_b \left( \mathbb X \right)$. We denote the convergence by $\mu^{(i)} \rightarrow \mu$.
 \end{mydef}

 \begin{mydef} \label{def_uniform_tightness}
   Let $\left\{ \mu^{(i)} \right\}_{i \in \mathbb N}$ be a sequence of probability measures defined on $\left(\mathbb X, \mathfrak B \right)$. The probability measures are said to be \textit{uniformly tight} if for every positive constant $\epsilon$, there exists a compact subset $\mathbb K$ of $\left( \mathbb X, \mathcal T \right)$ for which $\mu^{(i)} \left( \mathbb K \right) > 1 - \epsilon$ holds for all $i$ in $\mathbb N$.
 \end{mydef}

 A measurable subset $\mathbb A$ of $\mathbb X$ is said to be a \textit{$\mu$-continuity set} if its boundary set has the zero measure with respect to $\mu$, i.e., $\mu \left( \mathrm{bd} \left( \mathbb A\right) \right) = 0$. The following is the \textit{portmanteau theorem} (See Theorem 11.1.1 in \cite{dudley_9780511755347}).
 \begin{thm} \label{thm_portmanteau}
   For a sequence $\left\{ \mu^{(i)} \right\}_{i \in \mathbb N}$ of probability measures and a probability measure $\mu$ on $\left( \mathbb X, \mathfrak B \right)$, the following are equivalent:
   \begin{enumerate}
   \item $\mu^{(i)} \rightarrow \mu$
     
   \item $\limsup_{i \to \infty} \mu^{(i)} \left( \mathbb F \right) \leq \mu \left( \mathbb F \right)$ for any closed subset $\mathbb F$ of $\mathbb X$

   \item $\liminf_{i \to \infty} \mu^{(i)} \left( \mathbb O \right) \geq \mu \left( \mathbb O \right)$ for any open subset $\mathbb O$ of $\mathbb X$

   \item $\lim_{i \to \infty} \mu^{(i)} \left( \mathbb A \right) = \mu \left( \mathbb A \right)$ for any $\mu$-continuity subset $\mathbb A$ of $\mathbb X$.
   \end{enumerate}
 \end{thm}

 \subsection{Preliminary Results}
 \begin{lemma} \label{lemma_metric_space}
   The metric space $\left( \mathbb R^2 \times [0, 2\pi), d \right)$ defined in Section \ref{section_formulation} is complete, separable, and proper.
 \end{lemma}

 \begin{lemma} \label{lemma_optimal_policies}   
   Given estimates $\hat x_{k:N}$, let $\boldsymbol{\mathcal P}_{k:N}$ be non-degenerate policies that belong to $\boldsymbol{\mathfrak P} \left( \hat x_{k:N} \right)$. The following are true for all $j$ in $\{k, \cdots, N\}$:
   \begin{enumerate}
   \item $\mathbb P \left( \mathbf x_j \in \overline{\mathbb D}_j \,\Big|\, \mathbf R_k=0, \cdots, \mathbf R_j = 0 \right) = 1$
     
   \item $\mathbb P \left( \mathbf x_j \in \underline{\mathbb D}_j \,\Big|\, \mathbf R_k=0, \cdots, \mathbf R_j = 1 \right) = 0$
   \end{enumerate}
   subject to $\mathbf R_j = \boldsymbol{\mathcal P}_j \left( \mathbf x_j \right)$ for each $j$ in $\{k, \cdots, N\}$, where $\overline{\mathbb D}_j$ and $\underline{\mathbb D}_j$ are defined in \eqref{eq_best_decision_sets}.
 \end{lemma}
The proof directly follows from Proposition \ref{prop_best_response_mapping_P}.

 Based on Lemma \ref{lemma_optimal_policies}, we can state the following proposition.
 \begin{prop} \label{prop_decision_set_containment}
   Given estimates $\hat x_{k:N}$, let $\boldsymbol{\mathcal P}_{k:N}$ be non-degenerate policies that belong to $\boldsymbol{\mathfrak P} \left( \hat x_{k:N} \right)$. Consider compact sets $\left\{ \mathbb K_j \right\}_{j=k}^N$ given by\footnote{Due to the properness of the metric space $\left( \mathbb R^2 \times [0, 2\pi), d \right)$ (see Lemma \ref{lemma_metric_space}), every closed ball is a compact set.}
   \begin{align} \label{eq_lemma_decision_set_containment}
     \mathbb K_j = \left\{ x \in \mathbb R^2 \times [0, 2\pi) \,\Big|\, d^2\left( x, \hat x_j \right) \leq c_j' \right\}
   \end{align}
   The following holds for all $j$ in $\{k, \cdots, N\}$:
   \begin{align}
     \mathbb P \left( \mathbf x_j \in \mathbb K_j \,\Big|\, \mathbf R_k=0, \cdots, \mathbf R_j=0 \right) = 1
   \end{align}
   subject to $\mathbf R_j = \boldsymbol{\mathcal P}_j \left( \mathbf x_j \right)$ for each $j$ in $\{k, \cdots, N\}$.
 \end{prop}
 The proof follows from the fact that  $\mathbb K_j$ contains the set $\overline{\mathbb D}_j$ defined in \eqref{eq_best_decision_sets_upper_limit} and Lemma \ref{lemma_optimal_policies}.

 \begin{prop} \label{prop_time_evolution}
   Given policies $\boldsymbol{\mathcal P}_{k:N}$, for each $j$ in $\{k, \cdots, N\}$, let us define
   \begin{subequations} \label{eq_probability_measures}
     \begin{align}
       \mu_{j|j} \left( \mathbb A \right) &= \mathbb P \left( \mathbf x_j \in \mathbb A \,\Big|\, \mathbf R_k=0, \cdots, \mathbf R_j=0 \right) \\
       \mu_{j|j-1} \left( \mathbb A \right) &= \mathbb P \left( \mathbf x_j \in \mathbb A \,\Big|\, \mathbf R_k=0, \cdots, \mathbf R_{j-1}=0 \right)
     \end{align}
   \end{subequations}
   subject to $\mathbf R_j = \boldsymbol{\mathcal P}_j \left( \mathbf x_j \right)$ for each $j$ in $\{k, \cdots, N\}$, where $\mathbb A$ is a Borel-measurable subset. The probability measures \eqref{eq_probability_measures} evolve according to the following update rules:
   \begin{enumerate}
   \item Policy update rule: 
     $$\mu_{j|j} \left( \mathbb A \right) = \frac{\int_{\mathbb A} \mathbb P \left( \boldsymbol{\mathcal P}_j \left( \mathbf x_j \right) = 0 \,\Big|\, \mathbf x_j=x \right) \,\mathrm d\mu_{j|j-1}}{\mathbb P \left( \mathbf R_j=0 \,\Big|\, \mathbf R_k=0, \cdots, \mathbf R_{j-1}=0 \right)}$$
     provided that $\mathbb P \left( \mathbf R_j=0 \,\Big|\, \mathbf R_k=0, \cdots, \mathbf R_{j-1}=0 \right)$ is positive.

   \item Process update rule:
     $$\mu_{j|j-1} \left( \mathbb A \right) = \int_{\mathbb R^2 \times [0, 2\pi)} \mathbb P \left( \mathbf x_j \in \mathbb A \,\Big|\, \mathbf x_{j-1} = x \right) \,\mathrm d\mu_{j-1|j-1}$$

   \end{enumerate}
 \end{prop}

 \subsection{Proof of Proposition \ref{prop_mapping_G_j_continuity}} \label{proof_prop_mapping_G_j_continuity}
 To start with, we note that for each $j$ in $\{k, \cdots, N\}$, the function $\mathcal G_j$ can be written as follows:
 \begin{align} \label{eq_proof_prop_mapping_G_j_continuity_01}
   \mathcal G_j \left( x_{j-1}, \hat x_{j:N} \right) = \mathbb E_{\mathbf x_j} \left[ g_j \left( \mathbf x_j, \hat x_{j:N} \right) \,\big|\, \mathbf x_{j-1} = x_{j-1} \right]
 \end{align}
 with $\mathcal G_{N+1}=0$, where $g_j(x_j, \hat x_{j:N}) = \min \left\{ d^2 \left( x_j, \hat x_j \right) + \mathcal G_{j+1} \left( x_j, \hat x_{j+1:N} \right), c_j' \right\}$.

 We prove the statement using mathematical induction starting from $j = N+1$. Since $\mathcal G_{N+1}$ is constant, e.g., $\mathcal G_{N+1} = 0$, it is a continuous function. Now suppose that $\mathcal{G}_{j+1}$ is a continuous function. Note that $g_j$ in \eqref{eq_proof_prop_mapping_G_j_continuity_01} is a continuous function. To verify the continuity of $\mathcal G_j$, let $\left\{ x_{j-1}^{(i)} \right\}_{i \in \mathbb{N}}$ and $\left\{ \hat x_{j:N}^{(i)} \right\}_{i \in \mathbb N}$ be sequences that converge to $x_{j-1}$ and $\hat x_{j:N}$, respectively. For each set $\mathbb A$ in $\mathfrak B$, let us define
 \begin{align}
   \mu_j^{(i)} \left( \mathbb A \right) &= \mathbb P \left( \mathbf x_j \in \mathbb A \,\Big|\, \mathbf x_{j-1} = x_{j-1}^{(i)} \right) \\
   \mu_j \left( \mathbb A \right) &= \mathbb P \left( \mathbf x_j \in \mathbb A \,\Big|\, \mathbf x_{j-1} = x_{j-1} \right)
 \end{align}
 By Assumption \ref{assump_transition_probability} and Theorem \ref{thm_portmanteau}, we can see that $\left\{ \mu_j^{(i) }\right\}_{i \in \mathbb N}$ converges to $\mu_j$. 

Since $\left( \mathbb R^2 \times [0, 2\pi), d \right)$ is a complete, separable metric space by Lemma \ref{lemma_metric_space}, using the Skorokhod representation theorem \cite{skorokhod1956_tpa}, we can see that there is a sequence of random variables $\left\{ \mathbf y_j^{(i)} \right\}_{i \in \mathbb{N}}$ and a random variable $\mathbf y_j$ all defined on a common probability space $\left( \Omega, \mathfrak{F}, \nu \right)$ in which the following three facts are true:
 \begin{enumerate} [label=\bfseries (F\arabic*)]
 \item $\mu_j^{(i)}$ is the probability measure of $\mathbf y_j^{(i)}$, i.e., $\nu \left( \left\{ \omega \in \Omega \,\Big|\, \mathbf y_j^{(i)}(\omega) \in \mathbb A \right\} \right) = \mu_j^{(i)} \left( \mathbb A\right)$ for each $\mathbb A$ in $\mathfrak B$. \label{enum_01}

 \item $\mu_j$ is the probability measure of $\mathbf y_j$, i.e., $\nu \left( \left\{ \omega \in \Omega \,\Big|\, \mathbf y_j(\omega) \in \mathbb A \right\} \right) = \mu_j \left( \mathbb A\right)$ for each $\mathbb A$ in $\mathfrak B$. \label{enum_02}

 \item $\left\{ \mathbf y_j^{(i)} \right\}_{i \in \mathbb{N}}$ converges to $\mathbf y_j$ almost surely. \label{enum_03}
 \end{enumerate}

 From \ref{enum_01} and \ref{enum_02}, we can derive
 \begin{align}
   &\mathcal G_j \left( x_{j-1}^{(i)}, \hat x_{j:N}^{(i)} \right) - \mathcal G_j \left( x_{j-1}, \hat x_{j:N} \right) \nonumber \\
   &= \mathbb E_{\mathbf x_j} \left[ g_j \left( \mathbf x_j, \hat x_{j:N}^{(i)} \right) \,\Big|\, \mathbf x_{j-1} = x_{j-1}^{(i)} \right] - \mathbb E_{\mathbf x_j} \left[ g_j \left( \mathbf x_j, \hat x_{j:N} \right) \,\Big|\, \mathbf x_{j-1} = x_{j-1} \right] \nonumber \\
   &= \int_{\Omega} g_j \left( \mathbf y_j^{(i)}(\omega), \hat x_{j:N}^{(i)} \right) \mathrm d\nu - \int_{\Omega} g_j \left( \mathbf y_j(\omega), \hat x_{j:N} \right) \mathrm d\nu
 \end{align}

 Notice that by the fact that $c_j'$ is a fixed constant and $g$ is a non-negative function, it holds that $${0 \leq g_j \left( \mathbf y_j^{(i)}(\omega), \hat x_{j:N}^{(i)} \right) \leq c_j'}$$ for every $i$ in $\mathbb N$ and every $\omega \in \Omega$. Hence, the sequence of functions $\left\{ g_j \left( \mathbf y_j^{(i)}(\cdot), \hat x_{j:N}^{(i)} \right) \right\}_{i \in \mathbb N}$ is uniformly bounded. Also, by the continuity of $g_j$ and \ref{enum_03}, it holds that
 \begin{align*}
   \lim_{i \to \infty} g_j \left( \mathbf y_j^{(i)}(\omega), \hat x_{j:N}^{(i)} \right) = g_j \left( \mathbf y_j(\omega), \hat x_{j:N} \right)
 \end{align*}
 for almost every $\omega$ in $\Omega$. Using the bounded convergence theorem (see Theorem 16.5 in \cite{billingsley1995_wiley}), we have that 
 \begin{align*}
   &\lim_{i \to \infty} \left| \mathcal G_j \left( x_{j-1}^{(i)}, \hat x_{j:N}^{(i)} \right) - \mathcal G_j \left( x_{j-1}, \hat x_{j:N} \right) \right| \nonumber \\
   &=\lim_{i \to \infty} \left| \int_{\Omega} g_j \left( \mathbf y_j^{(i)}(\omega), \hat x_{j:N}^{(i)} \right) \mathrm d\nu - \int_{\Omega} g_j \left( \mathbf y_j(\omega), \hat x_{j:N} \right) \mathrm d\nu\right| = 0
 \end{align*}
 which proves that the function $\mathcal G_j$ is continuous.
 
 Finally, by induction, we conclude that the functions $\left\{ \mathcal G_j \right\}_{j=k}^N$ are all continuous. \hfill \QED

 \subsection{Proofs of Proposition \ref{prop_optimal_solution_non-degeneracy} and Lemma \ref{lemma_optimal_solution_containment}} \label{proof_section_2}
 \begin{lemma} \label{lemma_non-degenerate_policy}
   For each $j$ in $\{k, \cdots, N\}$, there exist estimates $\hat x_{j:N}$ for which the set given by
   \begin{align}
     \underline{\mathbb D}_{j} = \left\{ x_{j} \in \mathbb R^2 \times [0, 2\pi) \,\Big|\, d^2 \left(x_{j}, \hat x_{j} \right) + \mathbb E_{\mathbf x_{j+1}} \left[ J_{j+1}^\ast \left( \mathbf x_{j+1}, \hat x_{j+1:N} \right) \,\Big|\, \mathbf x_{j} = x_{j} \right] < c_{j}' \right\}
   \end{align}
   is non-empty, where $J_{j+1}^\ast$ is defined in \eqref{eq_cost-to-go_02}.
 \end{lemma}
 \begin{proof}
   Recall how $c_{j}'$ is determined by \eqref{eq_stopping_cost} with the solutions $\boldsymbol{\mathcal T}_{j+1:N}^{<j>}$ and $\mathcal E_{j+1:N}^{<j>}$ to Sub-problem $j+1$. Let us select a fixed point $x_{j}^o$ in $\mathbb R^2 \times [0, 2\pi)$. Under the choice of $\hat x_{j} = x_{j}^o$ and $\hat x_l = \mathcal E_l^{<j>} (x_{j})$ for each $l$ in $\{j+1, \cdots, N\}$, by a similar argument as in Remark \ref{remark_solution_transformation}, we can see that 
   \begin{align*}
     \mathbb E_{\mathbf x_{j+1}} \left[ J_{j+1}^\ast \left( \mathbf x_{j+1}, \hat x_{j+1:N} \right) \,\Big|\, \mathbf x_{j} = x_{j} \right] = \mathbb E_{\mathbf x_{j+1}} \left[ J_{j+1}^\ast \left( \mathbf x_{j+1}, \hat x_{j+1:N}' \right) \,\Big|\, \mathbf x_{j} = 0 \right]
   \end{align*}
   where $\hat x_l' = \mathcal E_l^{<j>} (0)$ for each $l$ in $\{j+1, \cdots, N\}$. Hence, at $x_{j} = x_j^o$, it holds that
   \begin{align*}
     &d^2 \left( x_{j}, \hat x_{j} \right) + \mathbb E_{\mathbf x_{j+1}} \left[ J_{j+1}^\ast \left( \mathbf x_{j+1}, \hat x_{j+1:N} \right) \,\Big|\, \mathbf x_{j} = x_{j} \right]\\
     &= \mathbb E_{\mathbf x_{j+1}} \left[ J_{j+1}^\ast \left( \mathbf x_{j+1}, \hat x_{j+1:N} \right) \,\Big|\, \mathbf x_{j} = x_{j} \right] \\
     &< c_{j} + \mathbb E_{\mathbf x_{j+1}} \left[ J_{j+1}^\ast \left( \mathbf x_{j+1}, \hat x_{j+1:N} \right) \,\Big|\, \mathbf x_{j} = x_{j} \right] \\
     &= c_{j} + \mathbb E_{\mathbf x_{j+1}} \left[ J_{j+1}^\ast \left( \mathbf x_{j+1}, \hat x_{j+1:N}' \right) \,\Big|\, \mathbf x_{j} = 0 \right] = c_{j}'
   \end{align*}
   This proves the Lemma.
 \end{proof}

 \begin{lemma} \label{lemma_mapping_continuity}
   Given estimates $\hat x_{k:N}$, the sets $\overline{\mathbb D}_j$ and $\underline{\mathbb D}_j$ defined in \eqref{eq_best_decision_sets} are closed and open, respectively, for all $j$ in $\{k, \cdots, N\}$.
 \end{lemma}
 The proof directly follows from the continuity of the functions $\left\{ \mathcal G_j \right\}_{j=k}^N$, each defined in \eqref{eq_mapping_G_j} (See Proposition~\ref{prop_mapping_G_j_continuity}).

 \textit{Proof of Proposition \ref{prop_optimal_solution_non-degeneracy}:}
 By contradiction, suppose that the degenerate policies $\boldsymbol{\mathcal P}_{k:N}^\ast$ and estimates $\hat x_{k:N}^\ast$ are jointly optimal for \eqref{eq_cost_functional_two-player_stopping_02}. Let $j_0 \in \{k, \cdots, N\}$ be the smallest integer for which 
\begin{equation} \label{eq_prop_optimal_solution_non-degeneracy_01}
 \mathbb P \left( \mathbf R_{j_0}^\ast=0 \,\Big|\, \mathbf R_k^\ast=0, \cdots, \mathbf R_{j_0-1}^\ast=0 \right) = 0
\end{equation}
holds subject to $\mathbf R_j^\ast = \boldsymbol{\mathcal P}_j^\ast\left( \mathbf x_j \right)$ for each $j$ in $\{k, \cdots, N\}$. Since $j_0$ is the smallest such integer, by Proposition~\ref{prop_time_evolution}, the probability measure $\mu_{j_0|j_0-1}$ of $\mathbf x_{j_0}$ is well-defined.

Using Lemma \ref{lemma_non-degenerate_policy}, let us choose  $\hat x_{j_0:N}^{\circ} \in \left( \mathbb R^2 \times [0, 2\pi) \right)^{N-j_0+1}$ for which the set given by 
\begin{align}
 \underline{\mathbb D}_{j_0}^\circ = \left\{ x_{j_0} \in \mathbb R^2 \times [0, 2\pi) \,\Big|\, d^2 \left(x_{j_0}, \hat x_{j_0}^\circ \right) + \mathbb E_{\mathbf x_{j_0+1}} \left[ J_{j_0+1}^\ast \left( \mathbf x_{j_0+1}, \hat x_{j_0+1:N}^\circ \right) \,\Big|\, \mathbf x_{j_0} = x_{j_0} \right] < c_{j_0}' \right\}
\end{align}
is non-empty. Note that according to Lemma \ref{lemma_mapping_continuity}, the set $\underline{\mathbb D}_{j_0}^\circ$ is open; hence, from Assumption~\ref{assump_transition_probability} and Proposition~\ref{prop_time_evolution}, we have that
\begin{align} \label{eq_prop_optimal_solution_non-degeneracy_03}
 \mathbb P \left( \mathbf x_{j_0} \in \underline{\mathbb D}_{j_0}^\circ \,\Big|\, \mathbf R_k^\ast=0, \cdots, \mathbf R_{j_0-1}^\ast=0 \right) > 0
\end{align}

For each $j$ in $\{j_0, \cdots, N\}$, let us define a function $\mathcal P_{j}^\circ: \mathbb R^2 \times [0, 2\pi) \to \{0, 1\}$ as follows:
\begin{align*}
 \mathcal P_j^\circ(x_j) = \begin{cases} 0 & \text{if } x_j \in \underline{\mathbb D}_{j}^\circ \\ 1 & \text{otherwise} \end{cases}
\end{align*}
where
\begin{align*}
 \underline{\mathbb D}_{j}^\circ = \left\{ x_{j} \in \mathbb R^2 \times [0, 2\pi) \,\Big|\, d^2 \left(x_{j}, \hat x_{j}^\circ \right) + \mathbb E_{\mathbf x_{j+1}} \left[ J_{j+1}^\ast \left( \mathbf x_{j+1}, \hat x_{j+1:N}^\circ \right) \,\Big|\, \mathbf x_{j} = x_{j} \right] < c_{j}' \right\}
\end{align*}
Let us select new policies $\boldsymbol{\mathcal P}_{k:N}'$ and estimates $\hat x_{k:N}'$ as follows: for each $j$ in $\{k, \cdots, N\}$,
\begin{subequations} \label{eq_new_policies_estimates}
  \begin{align}
    \boldsymbol{\mathcal P}_j' &= \begin{cases} \boldsymbol{\mathcal P}_j^\ast & \text{if } j \in \{k, \cdots, j_0-1\} \\ \mathcal P_j^\circ & \text{if } j \in \{j_0, \cdots, N\} \end{cases} \label{eq_new_policies}\\
    \hat x_j' &= \begin{cases} \hat x_j^\ast & \text{if } j \in \{k, \cdots, j_0-1\} \\ \hat x_j^\circ & \text{if } j \in \{j_0, \cdots, N\} \end{cases} \label{eq_new_estimates}
  \end{align}
\end{subequations}

By definition, under the new policies $\boldsymbol{\mathcal P}_{k:N}'$, it holds that
\begin{align}
  \mathbb P \left( \mathbf x_{j_0} \in \underline{\mathbb D}_{j_0}^\circ \,\Big|\, \mathbf R_k'=0, \cdots, \mathbf R_{j_0}'=1 \right) = 0
\end{align}
subject to $\mathbf R_j' = \boldsymbol{\mathcal P}_j' \left( \mathbf x_j \right)$ for each $j$ in $\{k, \cdots, N\}$. This implies that
\begin{align} \label{eq_prop_optimal_solution_non-degeneracy_02}
  & \mathbb P \left( \mathbf x_{j_0} \in \underline{\mathbb D}_{j_0}^\circ \,\Big|\, \mathbf R_k'=0, \cdots, \mathbf R_{j_0-1}'=0 \right) \nonumber \\
  &= \mathbb P \left( \mathbf x_{j_0} \in \underline{\mathbb D}_{j_0}^\circ \,\Big|\, \mathbf R_k'=0, \cdots, \mathbf R_{j_0}'=0 \right) \cdot \mathbb P \left( \mathbf R_{j_0}'=0 \,\Big|\, \mathbf R_k'=0, \cdots, \mathbf R_{j_0-1}'=0 \right)
\end{align}
By \eqref{eq_prop_optimal_solution_non-degeneracy_03}, \eqref{eq_new_policies}, and \eqref{eq_prop_optimal_solution_non-degeneracy_02}, we can see that 
\begin{align} \label{eq_prop_optimal_solution_non-degeneracy_04}
 \mathbb P \left( \mathbf R_{j_0}'=0 \,\Big|\, \mathbf R_k'=0, \cdots, \mathbf R_{j_0-1}'=0 \right) > 0
\end{align}

Due to \eqref{eq_prop_optimal_solution_non-degeneracy_01}, by Remark \ref{remark_degeneracy}, we can see that
\begin{align*}
 \mathbb E_{\mathbf x_{j_0}} \left[ J_{j_0} \left( \mathbf x_{j_0}, \boldsymbol{\mathcal P}_{j_0:N}^\ast, \hat x_{j_0:N}^\ast \right) \,\Big|\, \mathbf R_k^\ast=0, \cdots, \mathbf R_{j_0-1}^\ast=0 \right] = c_{j_0}'
\end{align*}
While, by \eqref{eq_cost-to-go_01}, \eqref{eq_new_policies_estimates}, and \eqref{eq_prop_optimal_solution_non-degeneracy_04}, we can see that
\begin{align*}
 &\mathbb E_{\mathbf x_{j_0}} \left[ J_{j_0}\left(\mathbf x_{j_0}, \boldsymbol{\mathcal P}_{j_0:N}', \hat x_{j_0:N}' \right) \,\Big|\, \mathbf R_k'=0, \cdots, \mathbf R_{j_0-1}'=0 \right] \\
 &=\mathbb E_{\mathbf x_{j_0}} \left[ J_{j_0}^\ast \left(\mathbf x_{j_0}, \hat x_{j_0:N}' \right) \,\Big|\, \mathbf R_k'=0, \cdots, \mathbf R_{j_0-1}'=0 \right] < c_{j_0}'
\end{align*}
These relations imply that
\begin{align*}
 & \mathbb E_{\mathbf x_{j_0}} \left[ J_{j_0}\left( \mathbf x_{j_0}, \boldsymbol{\mathcal P}_{j_0:N}', \hat x_{j_0:N}' \right) \,\Big|\, \mathbf R_k'=0, \cdots, \mathbf R_{j_0-1}'=0 \right] \\
 & < \mathbb E_{\mathbf x_{j_0}} \left[ J_{j_0} \left( \mathbf x_{j_0}, \boldsymbol{\mathcal P}_{j_0:N}^\ast, \hat x_{j_0:N}^\ast \right) \,\Big|\, \mathbf R_k^\ast=0, \cdots, \mathbf R_{j_0-1}^\ast=0 \right]
\end{align*}
Using the facts that $\boldsymbol{\mathcal P}_j' = \boldsymbol{\mathcal P}_j^\ast$ and $\hat x_j' = \hat x_j^\ast$ for each $j$ in ${\{k, \cdots, j_0-1\}}$, and $j_0$ is the smallest integer for which \eqref{eq_prop_optimal_solution_non-degeneracy_01} holds, from \eqref{eq_cost-to-go_01}, we can infer that 
$$\mathbb E_{\mathbf x_k} \left[ J_k\left(\mathbf x_k, \boldsymbol{\mathcal P}_{k:N}', \hat x_{k:N}' \right) \right] < \mathbb E_{\mathbf x_k} \left[ J_k\left(\mathbf x_k, \boldsymbol{\mathcal P}_{k:N}^\ast, \hat x_{k:N}^\ast \right) \right] $$
which violates $\hat x_{k:N}^\ast$ is a global minimizer. Therefore, every global minimizer has to be non-degenerate. \hfill \QED

\begin{lemma} \label{lemma_bounded_estimates}
 Consider policies $\left\{ \boldsymbol{\mathcal P}_{k:N}^{(i)} \right\}_{i \in \mathbb N}$ and estimates $\left\{ \hat x_{k:N}^{(i)} \right\}_{i \in \mathbb N}$ satisfying ${\boldsymbol{\mathcal P}_{k:N}^{(i)} \in \boldsymbol{\mathfrak P}\left( \hat x_{k:N}^{(i-1)} \right)}$ for all $i$ in $\{k, \cdots, N\}$. Suppose that there exists a positive constant $\epsilon$ for which the following holds for all $i$ in $\mathbb N$ and $j$ in $\{k, \cdots, j_0\}$:
 \begin{align} \label{eq_lemma_bounded_estimates_01}
   \mathbb P \left( \mathbf R_j^{(i)}=0 \,\Big|\, \mathbf R_k^{(i)}=0, \cdots, \mathbf R_{j-1}^{(i)}=0 \right) \geq \epsilon
 \end{align}
 subject to $\mathbf R_j^{(i)} = \boldsymbol{\mathcal P}_j^{(i)} \left( \mathbf x_j \right)$ for each $i$ in $\mathbb N$ and $j$ in $\{k, \cdots, j_0\}$. Then the sequence $\left\{ \hat x_j^{(i)} \right\}_{i \in \mathbb N}$ is bounded for all $j$ in $\{k, \cdots, j_0\}$.
\end{lemma}
\begin{proof}
 By contradiction, suppose that there exists $j'$ in $\{k, \cdots, j_0\}$ such that for a subsequence $\left\{ \hat x_{j'}^{(i_l-1)} \right\}_{l \in \mathbb N}$ of $\left\{ \hat x_{j'}^{(i)} \right\}_{i \in \mathbb N}$, it holds that
 \begin{align} \label{eq_lemma_bounded_estimates_03}
   d\left( 0, \hat x_{j'}^{(i_l-1) }\right) \overset{l \to \infty}{\longrightarrow} \infty
 \end{align}

 For each $l$ in $\mathbb N$, let us choose a compact set $\mathbb K_{j'}^{(i_l)} = \left\{ x \in \mathbb R^2 \times [0, 2\pi) \,\Big|\, d^2 \left( x, \hat x_{j'}^{(i_l-1)}\right) \leq c_j' \right\}$. Then, according to Proposition \ref{prop_decision_set_containment}, the following holds for all $l$ in $\mathbb N$:
 \begin{align}
   \mathbb P \left( \mathbf x_{j'} \in \mathbb K_{j'}^{(i_l)} \,\Big|\, \mathbf R_k^{(i_l)}=0, \cdots, \mathbf R_{j'}^{(i_l)}=0 \right) = 1
 \end{align}
 subject to $\mathbf R_{j}^{(i_l)} = \boldsymbol{\mathcal P}_{j}^{(i_l)}\left( \mathbf x_{j} \right)$ for each $l$ in $\mathbb N$ and $j$ in $\{k, \cdots, j'\}$. Using \eqref{eq_lemma_bounded_estimates_01}, we can derive the following:
 \begin{align} \label{eq_lemma_bounded_estimates_02}
   &\mathbb P \left( \mathbf R_{j'}^{(i_l)}=0 \,\Big|\, \mathbf R_k^{(i_l)}=0, \cdots, \mathbf R_{j'-1}^{(i_l)}=0 \right) \nonumber \\
   &= \frac{\mathbb P \left( \mathbf R_{j'}^{(i_l)}=0 \,\Big|\, \mathbf x_{j'} \in \mathbb K_{j'}^{(i_l)}, \mathbf R_k^{(i_l)}=0, \cdots, \mathbf R_{j'-1}^{(i_l)}=0 \right)}{\mathbb P \left( \mathbf x_{j'} \in \mathbb K_{j'}^{(i_l)} \,\Big|\, \mathbf R_k^{(i_l)}=0, \cdots, \mathbf R_{j'}^{(i_l)}=0 \right)}  \cdot \mathbb P \left( \mathbf x_{j'} \in \mathbb K_{j'}^{(i_l)} \,\Big|\, \mathbf R_k^{(i_l)}=0, \cdots, \mathbf R_{j'-1}^{(i_l)}=0 \right) \nonumber \\
   &\leq \mathbb P \left( \mathbf x_{j'} \in \mathbb K_{j'}^{(i_l)} \,\Big|\, \mathbf R_k^{(i_l)}=0, \cdots, \mathbf R_{j'-1}^{(i_l)}=0 \right) \nonumber \\
   &= \frac{\mathbb P \left(\mathbf x_{j'} \in \mathbb K_{j'}^{(i_l)}, \mathbf R_k^{(i_l)}=0, \cdots, \mathbf R_{j'-1}^{(i_l)}=0 \right)}{\prod_{j=k}^{j'-1} \mathbb P \left( \mathbf R_j^{(i_l)}=0 \,\Big|\, \mathbf R_k^{(i_l)}=0, \cdots, \mathbf R_{j-1}^{(i_l)}=0 \right)} \nonumber \\
   & \leq \epsilon^{k-j'} \cdot \mathbb P \left( \mathbf x_{j'} \in \mathbb K_{j'}^{(i_l)} \right)
 \end{align}
 holds for all $l$ in $\mathbb N$. Hence, by Lemma \ref{lemma_tightness} and \eqref{eq_lemma_bounded_estimates_03}, we can see that
 \begin{align}
   \lim_{l \to \infty} \mathbb P \left( \mathbf x_{j'} \in \mathbb K_{j'}^{(i_l)} \right)=0
 \end{align}
 In conjunction with \eqref{eq_lemma_bounded_estimates_02}, we conclude that
 \begin{align} \label{eq_proof_lemma_optimal_solution_containment_06}
   \lim_{l \to \infty} \mathbb P \left( \mathbf R_{j'}^{(i_l)}=0 \,\Big|\, \mathbf R_k^{(i_l)}=0, \cdots, \mathbf R_{j'-1}^{(i_l)}=0 \right) = 0
 \end{align}
 This contradicts the fact that \eqref{eq_lemma_bounded_estimates_01} holds for all $i$ in $\mathbb N$ and $j$ in $\{k, \cdots, j_0\}$.
\end{proof}

\textit{Proof of Lemma \ref{lemma_optimal_solution_containment}:}
For a positive real $r$, let us define 
\begin{align}
 \mathbb K_r \overset{def}{=} \left\{ \hat x_{k:N} \in \left( \mathbb R^2 \times [0, 2\pi) \right)^{N-k+1} \,\Big|\, d \left( 0, \hat x_j \right) \leq r \text{ for all $j$ in $\{k, \cdots, N\}$}\right\}
\end{align}
To prove the lemma, it is sufficient to show that there exists $r>0$ for which with $\mathbb K = \mathbb K_r$, the statement of the lemma is true. By contradiction, suppose that there exists a sequence $\left\{ \hat x_{k:N}^{(i)}\right\}_{i \in \mathbb N} \subset \left( \mathbb R^2 \times [0, 2\pi) \right)^{N-k+1}$ that satisfies the following hypotheses:
\begin{enumerate} [label=\bfseries (H\arabic*)]
\item For each element $\hat x_{k:N}^{(i)}$ of the sequence, it holds that $\hat x_{k:N}^{(i)} \notin \mathbb K_i$. \label{enum_h_01}

\item For every $\hat x_{k:N}$ in $\mathbb K_i$, it holds that $\mathcal G \left( \hat x_{k:N} \right) > \mathcal G \left( \hat x_{k:N}^{(i)} \right)$. \label{enum_h_02}
\end{enumerate}

We constructively prove that the hypothesis \ref{enum_h_02} is violated for sufficiently large $i$ in $\mathbb N$. To proceed, let us select policies $\left\{ \boldsymbol{\mathcal P}_{k:N}^{(i)} \right\}_{i \in \mathbb N}$ that satisfy $\boldsymbol{\mathcal P}_{k:N}^{(i)} \in \boldsymbol{\mathfrak P} \left( \hat x_{k:N}^{(i-1)} \right)$. Let $j_0 \in \{k, \cdots, N\}$ be the smallest integer for which there is a subsequence $\left\{ \boldsymbol{\mathcal P}_{k:j_0}^{(i_l)} \right\}_{l \in \mathbb N}$ of $\left\{ \boldsymbol{\mathcal P}_{k:j_0}^{(i)} \right\}_{i \in \mathbb N}$ satisfying\footnote{Such $j_0$ always exists; otherwise according to Lemma \ref{lemma_bounded_estimates}, the sequence $\left\{ \hat x_{k:N}^{(i)} \right\}_{i \in \mathbb N}$ is bounded, which violates \ref{enum_h_01}.}
\begin{align} \label{eq_proof_lemma_optimal_solution_containment_01}
 \lim_{l \to \infty} \mathbb P \left( \mathbf R_{j_0}^{(i_l)}=0 \,\Big|\, \mathbf R_k^{(i_l)}=0, \cdots, \mathbf R_{j_0-1}^{(i_l)}=0 \right) = 0
\end{align}
subject to $\mathbf R_{j}^{(i_l)} = \boldsymbol{\mathcal P}_{j}^{(i_l)}\left( \mathbf x_{j} \right)$ for each $l$ in $\mathbb N$ and $j$ in $\{k, \cdots, j_0\}$. Note from \eqref{eq_cost-to-go_01} and \eqref{eq_proof_lemma_optimal_solution_containment_01}, we have that
\begin{align} \label{eq_proof_lemma_optimal_solution_containment_03}
 \lim_{l \to \infty} \mathbb E_{\mathbf x_{j_0}} \left[ J_{j_0} \left( \mathbf x_{j_0}, \boldsymbol{\mathcal P}_{j_0:N}^{(i_l)}, \hat x_{j_0:N}^{(i_l-1)} \right) \,\Big|\, \mathbf R_k^{(i_l)}=0, \cdots, \mathbf R_{j_0-1}^{(i_l)}=0 \right] = c_{j_0}'
\end{align}
Also, according to Lemma \ref{lemma_bounded_estimates}, the sequence $\left\{ \hat x_j^{(i)} \right\}_{i \in \mathbb N}$ is bounded for all $j$ in ${\{k, \cdots, j_0-1\}}$.

Using Lemma \ref{lemma_non-degenerate_policy}, let us choose $\hat x_{j_0:N}^\circ \in \left( \mathbb R^2 \times [0, 2\pi) \right)^{N-j_0+1}$ for which the set given by
\begin{align}
  \underline{\mathbb D}_{j_0}^\circ = \left\{ x_{j_0} \in \mathbb R^2 \times [0, 2\pi) \,\Big|\, d^2 \left(x_{j_0}, \hat x_{j_0}^\circ \right) + \mathbb E_{\mathbf x_{j_0+1}} \left[ J_{j_0+1}^\ast \left( \mathbf x_{j_0+1}, \hat x_{j_0+1:N}^\circ \right) \,\Big|\, \mathbf x_{j_0} = x_{j_0} \right] < c_{j_0}' \right\}
\end{align}
is non-empty, where $J_{j_0+1}^\ast$ is defined in \eqref{eq_cost-to-go_02}. Note that by Proposition \ref{prop_mapping_G_j_continuity}
\begin{align*}
  d^2 \left(x_{j_0}, \hat x_{j_0}^\circ \right) + \mathbb E_{\mathbf x_{j_0+1}} \left[ J_{j_0+1}^\ast \left( \mathbf x_{j_0+1}, \hat x_{j_0+1:N}^\circ \right) \,\Big|\, \mathbf x_{j_0} = x_{j_0} \right]
\end{align*}
is a continuous function of $x_{j_0}$. Hence, for a positive constant $\epsilon$, the set defined by
\begin{align} \label{eq_proof_lemma_optimal_solution_containment_04}
  \mathbb B = \left\{ x_{j_0} \in \mathbb R^2 \times [0, 2\pi) \,\Big|\, d^2 \left(x_{j_0}, \hat x_{j_0}^\circ \right) + \mathbb E_{\mathbf x_{j_0+1}} \left[ J_{j_0+1}^\ast \left( \mathbf x_{j_0+1}, \hat x_{j_0+1:N}^\circ \right) \,\Big|\, \mathbf x_{j_0} = x_{j_0} \right] < c_{j_0}'-\epsilon \right\}
\end{align}
is non-empty and open.

For each $j$ in $\{j_0, \cdots, N\}$, let us define a function $\mathcal P_{j}^\circ: \mathbb R^2 \times [0, 2\pi)  \to \{0, 1\}$ as follows:
\begin{align*}
 \mathcal P_j^\circ(x_j) = \begin{cases} 0 & \text{if } x_j \in \underline{\mathbb D}_{j}^\circ \\ 1 & \text{otherwise} \end{cases}
\end{align*}
where
\begin{align*}
 \underline{\mathbb D}_{j}^\circ = \left\{ x_{j} \in \mathbb R^2 \times [0, 2\pi) \,\Big|\, d^2 \left(x_{j}, \hat x_{j}^\circ \right) + \mathbb E_{\mathbf x_{j+1}} \left[ J_{j+1}^\ast \left( \mathbf x_{j+1}, \hat x_{j+1:N}^\circ \right) \,\Big|\, \mathbf x_{j} = x_{j} \right] < c_{j}' \right\}
\end{align*}
Also, let us select sequences of policies $\left\{ \boldsymbol{\mathcal P}_{k:N}'^{(i)} \right\}_{i \in \mathbb N}$ and estimates $\left\{ \hat x_{k:N}'^{(i)} \right\}_{i \in \mathbb N}$ as follows: for each $i$ in $\mathbb N$ and $j$ in $\{k, \cdots, N\}$,
\begin{subequations} \label{eq_proof_lemma_optimal_solution_containment_07}
  \begin{align}
    \boldsymbol{\mathcal P}_j'^{(i)} &= \begin{cases} \boldsymbol{\mathcal P}_j^{(i)} & \text{if } j \in \{k, \cdots, j_0-1\} \\ \mathcal P_j^\circ & \text{if } j \in \{j_0, \cdots, N\} \end{cases} \\
    \hat x_j'^{(i)} &= \begin{cases} \hat x_j^{(i)} & \text{if } j \in \{k, \cdots, j_0-1\} \\ \hat x_j^\circ & \text{if } j \in \{j_0, \cdots, N\} \end{cases}
  \end{align}
\end{subequations}

We argue that $\hat x_{k:N}'^{(i-1)} \in \mathbb K_{i-1}$ and ${\mathcal G \left( \hat x_{k:N}'^{(i-1)} \right) < \mathcal G \left( \hat x_{k:N}^{(i-1)} \right)}$ hold for sufficiently large $i$. This contradicts the hypothesis \ref{enum_h_02}; hence it completes the proof of the lemma. In what follows, we show that this argument is valid. Note that by Lemma \ref{lemma_bounded_estimates}, the sequence $\left\{ \hat x_{j_0-1}^{(i-1)} \right\}_{i \in \mathbb N}$ is bounded. By Proposition \ref{prop_decision_set_containment} and by the fact that $\boldsymbol{\mathcal P}_j'^{(i)} = \boldsymbol{\mathcal P}_j^{(i)}$ for all $i$ in $\mathbb N$ and $j$ in $\{k, \cdots, j_0-1\}$, there exists a compact set $\mathbb K_{j_0-1}$ for which the following holds for all $i$ in $\mathbb N$:
\begin{align*}
  &\mathbb P \left( \mathbf x_{j_0-1} \in \mathbb K_{j_0-1} \,\Big|\, \mathbf R_k'^{(i)}=0, \cdots, \mathbf R_{j_0-1}'^{(i)}=0 \right) \\
  &= \mathbb P \left( \mathbf x_{j_0-1} \in \mathbb K_{j_0-1} \,\Big|\, \mathbf R_k^{(i)}=0, \cdots, \mathbf R_{j_0-1}^{(i)}=0 \right) = 1
\end{align*}
By Proposition \ref{prop_time_evolution}, we can write that
\begin{align*}
  &\mathbb P \left( \mathbf x_{j_0} \in \mathbb B \,\Big|\, \mathbf R_k'^{(i)}=0, \cdots, \mathbf R_{j_0-1}'^{(i)}=0 \right) \nonumber \\
  &= \int_{\mathbb K_{j_0-1}} \mathbb P \left( \mathbf x_{j_0} \in \mathbb B \,\Big|\, \mathbf x_{j_0-1} = x \right) \,\mathrm d\mu_{j_0-1|j_0-1}'^{(i)}
\end{align*}
where the probability measure $\mu_{j_0-1|j_0-1}'^{(i)}$ is defined as follows: For each $\mathbb A$ in $\mathfrak B$,
\begin{align*}
  \mu_{j_0-1|j_0-1}'^{(i)} \left( \mathbb A \right) = \mathbb P \left( \mathbf x_{j_0-1} \in \mathbb A \,\Big|\, \mathbf R_k'^{(i)}=0, \cdots, \mathbf R_{j_0-1}'^{(i)}=0 \right)
\end{align*}
Due to Assumption \ref{assump_transition_probability} and the compactness of $\mathbb K_{j_0-1}$, for a positive constant $\delta_{j_0}$ and for the set $\mathbb B$ given by \eqref{eq_proof_lemma_optimal_solution_containment_04}, the following holds for all $i$ in $\mathbb N$:
\begin{align} \label{eq_proof_lemma_optimal_solution_containment_02}
  &\mathbb P \left( \mathbf x_{j_0} \in \mathbb B \,\Big|\, \mathbf R_k'^{(i)}=0, \cdots, \mathbf R_{j_0-1}'^{(i)}=0 \right) \nonumber \\
  & \geq \delta_{j_0} \cdot \mu_{j_0-1|j_0-1}'^{(i)} \left( \mathbb K_{j_0-1} \right) = \delta_{j_0}
\end{align}

Since the sequence $\left\{ \hat x_j^{(i-1)} \right\}_{i \in \mathbb N}$ is bounded for all $j$ in $\{k, \cdots, j_0-1\}$, for sufficiently large $i$, we can see that $\hat x_{k:N}'^{(i-1)} \in \mathbb K_{i-1}$. In addition, by \eqref{eq_cost-to-go_01}, \eqref{eq_proof_lemma_optimal_solution_containment_04}, \eqref{eq_proof_lemma_optimal_solution_containment_07}, and \eqref{eq_proof_lemma_optimal_solution_containment_02}, we can see that the following relations hold for all $i$ in $\mathbb N$:
{ \allowdisplaybreaks
  \begin{align} \label{eq_proof_lemma_optimal_solution_containment_05}
    &\mathbb E_{\mathbf x_{j_0}} \left[ J_{j_0} \left( \mathbf x_{j_0}, \boldsymbol{\mathcal P}_{j_0:N}'^{(i)}, \hat x_{j_0:N}'^{(i-1)} \right) \,\Big|\, \mathbf R_k'^{(i)}=0, \cdots, \mathbf R_{j_0-1}'^{(i)}=0 \right] \nonumber \\
    &= \Big( \mathbb E_{\mathbf x_{j_0}} \left[ d^2\left( \mathbf x_{j_0}, \hat x_{j_0}^\circ \right) \,\Big|\, \mathbf R_k'^{(i)}=0, \cdots, \mathbf R_{j_0}'^{(i)}=0  \right] + \mathbb E_{\mathbf x_{j_0+1}} \left[ J_{j_0+1}^\ast \left( \mathbf x_{j_0+1}, \hat x_{j_0+1:N}^\circ \right) \,\Big|\, \mathbf R_k'^{(i)}=0, \cdots, \mathbf R_{j_0}'^{(i)}=0 \right] \Big) \nonumber \\
    &\quad \cdot \mathbb P\left(\mathbf R_{j_0}'^{(i)}=0 \,\Big|\, \mathbf R_k'^{(i)}=0, \cdots, \mathbf R_{j_0-1}'^{(i)}=0\right) + c_{j_0}' \cdot \left(1 - \mathbb P\left(\mathbf R_{j_0}'^{(i)}=0 \,\Big|\, \mathbf R_k'^{(i)}=0, \cdots, \mathbf R_{j_0-1}'^{(i)}=0\right) \right) \nonumber \\
    &\overset{\mathbf{(i)}}{\leq} \Big( \mathbb E_{\mathbf x_{j_0}} \left[ d^2\left( \mathbf x_{j_0}, \hat x_{j_0}^\circ \right) \,\Big|\, \mathbf R_k'^{(i)}=0, \cdots, \mathbf R_{j_0-1}'^{(i)}=0, \mathbf x_{j_0} \in \mathbb B  \right] \nonumber \\
    &\qquad + \mathbb E_{\mathbf x_{j_0+1}} \left[ J_{j_0+1}^\ast \left( \mathbf x_{j_0+1}, \hat x_{j_0+1:N}^o \right) \,\Big|\, \mathbf R_k'^{(i)}=0, \cdots, \mathbf R_{j_0-1}'^{(i)}=0, \mathbf x_{j_0} \in \mathbb B \right] \Big) \nonumber \\
    &\quad \cdot \mathbb P\left(\mathbf x_{j_0} \in \mathbb B \,\Big|\, \mathbf R_k'^{(i)}=0, \cdots, \mathbf R_{j_0-1}'^{(i)}=0\right) + c_{j_0}' \cdot \left(1 - \mathbb P\left(\mathbf x_{j_0} \in \mathbb B \,\Big|\, \mathbf R_k'^{(i)}=0, \cdots, \mathbf R_{j_0-1}'^{(i)}=0\right) \right) \nonumber \\
    &\overset{\mathbf{(ii)}}{<} \left( c_{j_0}' - \epsilon \right) \cdot \mathbb P\left(\mathbf x_{j_0} \in \mathbb B \,\Big|\, \mathbf R_k'^{(i)}=0, \cdots, \mathbf R_{j_0-1}'^{(i)}=0\right) + c_{j_0}' \cdot \left( 1 - \mathbb P\left(\mathbf x_{j_0} \in \mathbb B \,\Big|\, \mathbf R_k'^{(i)}=0, \cdots, \mathbf R_{j_0-1}'^{(i)}=0\right) \right) \nonumber \\
    &\overset{\mathbf{(iii)}}{\leq} c_{j_0}'-\epsilon \cdot \delta_{j_0}
  \end{align}}
To obtain $\mathbf{(i)}$, we use the fact that
\begin{align*}
 d^2 \left(\mathbf x_{j_0}, \hat x_{j_0}^\circ \right) + \mathbb E_{\mathbf x_{j_0+1}} \left[ J_{j_0+1}^\ast \left( \mathbf x_{j_0+1}, \hat x_{j_0+1:N}^\circ \right) \,\Big|\, \mathbf x_{j_0} \right] < c_{j_0}'
\end{align*}
if $\mathbf R_{j_0}'^{(i)} = 0$ (or equivalently $\mathbf x_{j_0} \in \underline{\mathbb D}_{j_0}^\circ$), and $\mathbb B$ is a subset of $\underline{\mathbb D}_{j_0}^\circ$; whereas $\mathbf{(ii)}$ and $\mathbf{(iii)}$ follow from \eqref{eq_proof_lemma_optimal_solution_containment_04} and \eqref{eq_proof_lemma_optimal_solution_containment_02}, respectively. By a similar argument as in the proof of Proposition \ref{prop_optimal_solution_non-degeneracy}, from \eqref{eq_proof_lemma_optimal_solution_containment_03} and \eqref{eq_proof_lemma_optimal_solution_containment_05}, we can observe that for sufficiently large $i$, there exists $\hat x_{k:N}^{(i-1)}$ in $\mathbb K_{i-1}$ for which it holds that
$$\mathcal G \left( \hat x_{k:N}'^{(i-1)} \right) \leq \mathbb E_{\mathbf x_k} \left[ J_k\left(\mathbf x_k, \boldsymbol{\mathcal P}_{k:N}'^{(i)}, \hat x_{k:N}'^{(i-1)} \right) \right] < \mathbb E_{\mathbf x_k} \left[ J_k\left(\mathbf x_k, \boldsymbol{\mathcal P}_{k:N}^{(i)}, \hat x_{k:N}^{(i-1)} \right) \right] = \mathcal G \left( \hat x_{k:N}^{(i-1)} \right)$$
\hfill \QED

\subsection{Proofs of Lemmas \ref{lemma_estimates_convergence_spp} and \ref{lemma_P_closed}} \label{proof_lemma_P_closed}
  \textit{Proof of Lemma \ref{lemma_estimates_convergence_spp}:}
  To prove the lemma, we first claim that the following hold for all $j$ in $\{k, \cdots, N\}$:
  \begin{subequations} \label{eq_prop_estimates_convergence_spp_01}
    \begin{align}
      \lim_{l \to \infty} \mathbb E \left[ \mathbf p_{1,j} \,\Big|\, \mathbf R_k^{(i_l)}=0, \cdots, \mathbf R_j^{(i_l)}=0 \right] &= \mathbb E \left[ \mathbf p_{1,j} \,\Big|\, \mathbf R_k=0, \cdots, \mathbf R_j=0 \right] \label{eq_prop_estimates_convergence_spp_01_a} \\
      \lim_{l \to \infty} \mathbb E \left[ \mathbf p_{2,j} \,\Big|\, \mathbf R_k^{(i_l)}=0, \cdots, \mathbf R_j^{(i_l)}=0 \right] &= \mathbb E \left[ \mathbf p_{2,j} \,\Big|\, \mathbf R_k=0, \cdots, \mathbf R_j=0 \right] \\
      \lim_{l \to \infty} \mathbb E \left[ \sin \boldsymbol \theta_j \,\Big|\, \mathbf R_k^{(i_l)}=0, \cdots, \mathbf R_j^{(i_l)}=0 \right] &= \mathbb E \left[ \sin \boldsymbol \theta_j \,\Big|\, \mathbf R_k=0, \cdots, \mathbf R_j=0 \right]\\
      \lim_{l \to \infty} \mathbb E \left[ \cos \boldsymbol \theta_j \,\Big|\, \mathbf R_k^{(i_l)}=0, \cdots, \mathbf R_j^{(i_l)}=0 \right] &= \mathbb E \left[ \cos \boldsymbol \theta_j \,\Big|\, \mathbf R_k=0, \cdots, \mathbf R_j=0 \right]
    \end{align}
  \end{subequations}
  subject to $\mathbf R_j^{(i_l)} = \boldsymbol{\mathcal P}_j^{(i_l)} \left( \mathbf x_j \right)$ and $\mathbf R_j = \boldsymbol{\mathcal P}_j \left( \mathbf x_j \right)$ for each $l$ in $\mathbb N$ and $j$ in $\{k, \cdots, N\}$.

  For notational convenient, let us define $$\alpha = \mathbb E^2 \left[ \sin \boldsymbol \theta_j \,\Big|\, \mathbf R_k=0, \cdots, \mathbf R_j=0 \right] + \mathbb E^2 \left[ \cos \boldsymbol \theta_j \,\Big|\, \mathbf R_k=0, \cdots, \mathbf R_j=0 \right]$$
   If $\alpha$ is non-zero, then by Corollary \ref{cor_best_response_mapping_X}, $\hat x_j^{(i_l)} = \begin{pmatrix} \hat p_{1,j}^{(i_l)} & \hat p_{2,j}^{(i_l)} & \hat{\theta}_j^{(i_l)} \end{pmatrix}^T$ and $\hat x_j = \begin{pmatrix} \hat p_{1,j} & \hat p_{2,j} & \hat{\theta}_j \end{pmatrix}^T$ satisfy
  \begin{align*}
    p_{1,j}^{(i_l)} &= \mathbb E \left[ \mathbf p_{1,j} \,\Big|\, \mathbf R_k^{(i_l)}=0, \cdots, \mathbf R_j^{(i_l)}=0 \right] \\
    p_{2,j}^{(i_l)} &= \mathbb E \left[ \mathbf p_{2,j} \,\Big|\, \mathbf R_k^{(i_l)}=0, \cdots, \mathbf R_j^{(i_l)}=0 \right] \\
    \sin \hat{\theta}_j^{(i_l)} &= \frac{\mathbb E \left[ \sin \boldsymbol \theta_j \,\Big|\, \mathbf R_k^{(i_l)}=0, \cdots, \mathbf R_j^{(i_l)}=0 \right]}{\mathbb E^2 \left[ \sin \boldsymbol \theta_j \,\Big|\, \mathbf R_k^{(i_l)}=0, \cdots, \mathbf R_j^{(i_l)}=0 \right] + \mathbb E^2 \left[ \cos \boldsymbol \theta_j \,\Big|\, \mathbf R_k^{(i_l)}=0, \cdots, \mathbf R_j^{(i_l)}=0 \right]}\\
    \cos \hat{\theta}_j^{(i_l)} &= \frac{\mathbb E \left[ \cos \boldsymbol \theta_j \,\Big|\, \mathbf R_k^{(i_l)}=0, \cdots, \mathbf R_j^{(i_l)}=0 \right]}{\mathbb E^2 \left[ \sin \boldsymbol \theta_j \,\Big|\, \mathbf R_k^{(i_l)}=0, \cdots, \mathbf R_j^{(i_l)}=0 \right] + \mathbb E^2 \left[ \cos \boldsymbol \theta_j \,\Big|\, \mathbf R_k^{(i_l)}=0, \cdots, \mathbf R_j^{(i_l)}=0 \right]}
  \end{align*}
  and
  \begin{align} \label{eq_prop_estimates_convergence_spp_02}
    \lim_{l \to \infty} d\left( \hat x_j^{(i_l)}, \hat x_j\right) = 0
  \end{align}

  For each $j$ in $\{k, \cdots, N\}$, let us define $\hat x_j^\ast = \begin{pmatrix} \hat p_{1,j}^\ast & \hat p_{2,j}^\ast & \hat{\theta}_j^\ast \end{pmatrix}$ as
  \begin{align*}
    p_{1,j}^\ast &= \mathbb E \left[ \mathbf p_{1,j} \,\Big|\, \mathbf R_k=0, \cdots, \mathbf R_j=0 \right] \\
    p_{2,j}^\ast &= \mathbb E \left[ \mathbf p_{2,j} \,\Big|\, \mathbf R_k=0, \cdots, \mathbf R_j=0 \right] \\
    \sin \hat{\theta}_j^\ast &= \alpha^{-1} \cdot \mathbb E \left[ \sin \boldsymbol \theta_j \,\Big|\, \mathbf R_k=0, \cdots, \mathbf R_j=0 \right] \\
    \cos \hat{\theta}_j^\ast &= \alpha^{-1} \cdot \mathbb E \left[ \cos \boldsymbol \theta_j \,\Big|\, \mathbf R_k=0, \cdots, \mathbf R_j=0 \right]
  \end{align*}
  Note that by Corollary \ref{cor_best_response_mapping_X}, it holds that $\hat x_{k:N}^\ast \in \mathfrak X \left( \boldsymbol{\mathcal P}_{k:N} \right)$. From \eqref{eq_prop_estimates_convergence_spp_01} and \eqref{eq_prop_estimates_convergence_spp_02}, we can observe that 
  \begin{align*}
    d\left( \hat x_j, \hat x_j^\ast \right) \leq d\left( \hat x_j^{(i_l)}, \hat x_j^\ast \right) + d\left( \hat x_j^{(i_l)}, \hat x_j \right) \overset{l \to \infty}{\longrightarrow} 0
  \end{align*}
  Therefore, we conclude that  $\hat x_{k:N} \in \mathfrak X \left( \boldsymbol{\mathcal P}_{k:N} \right)$.

  Otherwise, if $\alpha$ is zero then the value of 
  $$\mathbb E_{\mathbf x_j} \left[ d^2 \left( \mathbf x_j, \hat x_j \right) \,\Big|\, \mathbf R_k=0, \cdots, \mathbf R_j=0 \right]$$
  does not depend on $\hat \theta_j$, and by Proposition \ref{prop_best_response_mapping_X} and Corollary \ref{cor_best_response_mapping_X}, we can show that $\hat x_{k:N} \in \mathfrak X \left( \boldsymbol{\mathcal P}_{k:N} \right)$ if the following hold for all $j$ in $\{k, \cdots, N\}$:
  \begin{align*}
    p_{1,j} &= \mathbb E \left[ \mathbf p_{1,j} \,\Big|\, \mathbf R_k=0, \cdots, \mathbf R_j=0 \right] \\
    p_{2,j} &= \mathbb E \left[ \mathbf p_{2,j} \,\Big|\, \mathbf R_k=0, \cdots, \mathbf R_j=0 \right]
  \end{align*}
  This can be verified by similar arguments given above and \eqref{eq_prop_estimates_convergence_spp_01}. It remains to prove the claim.

  \textbf{Proof of the Claim:} Notice that $\mathbf p_{1,j}, \mathbf p_{2,j}, \sin \hat{\boldsymbol \theta}_j, \cos \hat{\boldsymbol \theta}_j$ are all continuous functions of $\mathbf x_j$. Hence to show that \eqref{eq_prop_estimates_convergence_spp_01} is true, it is sufficient to show that the following holds for any continuous function $g: \mathbb R^2 \times [0, 2\pi) \to \mathbb R$:
  \begin{align}
    \lim_{l \to \infty} \mathbb E \left[ g \left( \mathbf x_j \right) \,\Big|\, \mathbf R_k^{(i_l)}=0, \cdots, \mathbf R_j^{(i_l)}=0 \right] = \mathbb E \left[ g\left( \mathbf x_j \right) \,\Big|\, \mathbf R_k=0, \cdots, \mathbf R_j=0 \right]
  \end{align}

  Recall the definitions of $\mu_{j|j}^{(i_l)}$ and $\mu_{j|j}$ given in \eqref{eq_probablity_measures}. Since $\left\{ \boldsymbol{\mathcal P}_{k:N}^{(i_l)} \right\}_{l \in \mathbb N}$ converges to $\boldsymbol{\mathcal P}_{k:N}$, by Definition \ref{def_policy_convergence}, it holds that $\mu_{j|j}^{(i_l)} \rightarrow \mu_{j|j}$ for all $j$ in $\{k, \cdots, N\}$. Since $\left( \mathbb R^2 \times [0, 2\pi), d \right)$ is a complete, separable metric space, by the convergence of $\left\{ \mu_{j|j}^{(i_l)} \right\}_{i \in \mathbb N}$ and the Skorokhod representation theorem \cite{skorokhod1956_tpa}, there exist a sequence of random variables $\left\{ \mathbf y_j^{(i_l)} \right\}_{l \in \mathbb N}$ and a random variable $\mathbf y_j$ all defined on a common probability space $\left( \Omega, \mathfrak{F}, \nu \right)$ in which the following three facts are true:
  \begin{enumerate}  [label=\bfseries (F\arabic*)]
  \item $\mu_{j|j}^{(i_l)}$ is the probability measure of $\mathbf y_j^{(i_l)}$, i.e., $\nu \left( \left\{ \omega \in \Omega \,\Big|\, \mathbf y_j^{(i_l)}(\omega) \in \mathbb A \right\} \right) = \mu_{j|j}^{(i_l)} \left( \mathbb A\right)$ for each $\mathbb A$ in $\mathfrak B$. \label{enum_01}

  \item $\mu_{j|j}$ is the probability measure of $\mathbf y_j$, i.e., $\nu \left( \left\{ \omega \in \Omega \,\Big|\, \mathbf y_j(\omega) \in \mathbb A \right\} \right) = \mu_{j|j} \left( \mathbb A\right)$ for each $\mathbb A$ in $\mathfrak B$. \label{enum_02}

  \item $\left\{ \mathbf y_j^{(i_l)} \right\}_{i \in \mathbb{N}}$ converges to $\mathbf y_j$ almost surely. \label{enum_03}
  \end{enumerate}
  
  Since $\left\{ \hat x_j^{(i_l-1)} \right\}_{l \in \mathbb N}$ is a convergent sequence, according to Proposition \ref{prop_decision_set_containment},  there is a compact set $\mathbb K_j$ for which $\mu_{j|j}^{(i_l)} \left( \mathbb K_j \right) = 1$ for all $l$ in $\mathbb N$. Hence, by \textbf{(F1)}, the following holds for a positive real $\beta$:
  \begin{align} \label{eq_prop_estimates_convergence_linear_02}
    \int_{\left\{ \omega \in \Omega \,\Big|\, \left| g\left(\mathbf y_j^{(i_l)} (\omega)\right) \right| > \beta \right\} } \left| g\left(\mathbf y_j^{(i_l)} (\omega) \right) \right| \,\mathrm d\nu = \int_{\left\{ x \in \mathbb R^2 \times [0, 2\pi) \,\big|\, \left| g\left(x\right) \right| > \beta \right\}} \left| g\left(x\right)\right| \,\mathrm d\mu_{j|j}^{(i_l)} = 0
  \end{align}
  for all $l$ in $\mathbb N$. In conjunction with \ref{enum_03}, by an application of Theorem 10.3.6 in \cite{dudley_9780511755347}, we have that
  \begin{align} \label{eq_prop_estimates_convergence_linear_03}
    \lim_{l \to \infty } \int_{\Omega} g\left(\mathbf y_j^{(i_l)} (\omega) \right) \,\mathrm d\nu = \int_{\Omega} g\left(\mathbf y_j (\omega) \right) \,\mathrm d\nu
  \end{align}
  Therefore, from \textbf{(F2)} and \eqref{eq_prop_estimates_convergence_linear_03}, we can see that
  \begin{align}
    \lim_{l \to \infty} \int_{\mathbb R^2 \times [0, 2\pi)} g \left( x \right)\, \mathrm d \mu_{j|j}^{(i_l)} &= \lim_{l \to \infty } \int_{\Omega} g\left(\mathbf y_j^{(i_l)} (\omega) \right) \,\mathrm d\nu \nonumber \\
                                                                                                             &= \int_{\Omega} g\left(\mathbf y_j (\omega) \right) \,\mathrm d\nu \nonumber \\
                                                                                                             &= \int_{\mathbb R^2 \times [0, 2\pi) } g\left( x\right) \,\mathrm d\mu_{j|j}
  \end{align}
  This proves the claim. \hfill \QED


\begin{lemma} \label{lemma_set_convergence}
 Let $\left\{ \hat x_{k:N}^{(i)} \right\}_{i \in \mathbb N}$ be a sequence of estimates that converges to $\hat x_{k:N}$. The following inclusions hold for all $j$ in $\{k, \cdots, N\}$:
 \begin{subequations}
   \begin{align}
     \overline{\mathbb D}_j &\supset \bigcap_{i \in \mathbb N} \bigcup_{l \geq i} \overline{\mathbb D}_j^{(l)} \label{eq_lemma_set_convergence_01}\\
     \underline{\mathbb D}_j &\subset \bigcup_{i \in \mathbb N} \bigcap_{l \geq i} \underline{\mathbb D}_j^{(l)} \label{eq_lemma_set_convergence_02}
   \end{align}
 \end{subequations}
 where
 \begin{align*} 
   \overline{\mathbb D}_j &= \left\{ x_j \in \mathbb R^2 \times [0, 2\pi) \,\Big|\, d^2 \left( x_j, \hat x_j \right) + \mathbb E_{\mathbf x_{j+1}} \left[ J_{j+1}^\ast \left( \mathbf x_{j+1}, \hat x_{j+1:N} \right) \,\Big|\, \mathbf x_j = x_j \right] \leq c_j' \right\} \\
   \underline{\mathbb D}_j &= \left\{ x_j \in \mathbb R^2 \times [0, 2\pi) \,\Big|\, d^2 \left( x_j, \hat x_j \right) + \mathbb E_{\mathbf x_{j+1}} \left[ J_{j+1}^\ast \left( \mathbf x_{j+1}, \hat x_{j+1:N} \right) \,\Big|\, \mathbf x_j = x_j \right] < c_j' \right\}
 \end{align*}
 and
 \begin{align*} 
   \overline{\mathbb D}_j^{(i)} &= \left\{ x_j \in \mathbb R^2 \times [0, 2\pi) \,\Big|\, d^2 \left( x_j, \hat x_j^{(i)} \right) + \mathbb E_{\mathbf x_{j+1}} \left[ J_{j+1}^\ast \left( \mathbf x_{j+1}, \hat x_{j+1:N}^{(i)} \right) \,\Big|\, \mathbf x_j = x_j \right] \leq c_j' \right\} \\
   \underline{\mathbb D}_j^{(i)} &= \left\{ x_j \in \mathbb R^2 \times [0, 2\pi) \,\Big|\, d^2 \left( x_j, \hat x_j^{(i)} \right) + \mathbb E_{\mathbf x_{j+1}} \left[ J_{j+1}^\ast \left( \mathbf x_{j+1}, \hat x_{j+1:N}^{(i)} \right) \,\Big|\, \mathbf x_j = x_j \right] < c_j' \right\}
 \end{align*}
 where $J_{j+1}^\ast$ is defined in \eqref{eq_cost-to-go_02}.
\end{lemma}
\begin{proof}
 Let $x_j$ be an element of $\bigcap_{i \in \mathbb N} \bigcup_{l \geq i} \overline{\mathbb D}_j^{(l)}$. By definition, there exists an infinite index set $\left\{ i_l \right\}_{l \in \mathbb N}$ for which $x_j \in \overline{\mathbb D}_j^{(i_l)}$ holds for all $l$ in $\mathbb N$. Hence, we can see that
 \begin{align}
   d^2 \left( x_j, \hat x_j^{(i_l)} \right) + \mathbb E_{\mathbf x_{j+1}} \left[ J_{j+1}^\ast \left( \mathbf x_{j+1}, \hat x_{j+1:N}^{(i_l)} \right) \,\Big|\, \mathbf x_j = x_j \right] \leq c_j'
 \end{align}
 holds for all $l$ in $\mathbb N$. Using Proposition \ref{prop_mapping_G_j_continuity} and by the fact that $\left\{ \hat x_{k:N}^{(i)} \right\}_{i \in \mathbb N}$ converges to $\hat x_{k:N}$, we can derive
 \begin{align}
   d^2 \left( x_j, \hat x_j \right) + \mathbb E_{\mathbf x_{j+1}} \left[ J_{j+1}^\ast \left( \mathbf x_{j+1}, \hat x_{j+1:N} \right) \,\Big|\, \mathbf x_j = x_j \right] \leq c_j'
 \end{align}
 which shows that $x_j \in \overline{\mathbb D}_j$. This proves \eqref{eq_lemma_set_convergence_01}.

 To show that \eqref{eq_lemma_set_convergence_02} is true, we consider
 \begin{align}
   \underline{\mathbb D}_j^c \supset \bigcap_{i \in \mathbb N} \bigcup_{l \geq i} \left( \underline{\mathbb D}_j^{(l)} \right)^c
 \end{align}
 As the rest of the proof is similar to the above arguments, we omit the detail for brevity.
\end{proof}

\textit{Proof of Lemma \ref{lemma_P_closed}:}
For every $\mathbb A$ in $\mathfrak B$, let us define
\begin{subequations}
 \begin{align}
   \mu_{j|j-1}^{(i)} \left( \mathbb A \right) &= \mathbb P \left( \mathbf x_j \in \mathbb A \,\Big|\, \mathbf R_k^{(i)} = 0, \cdots, \mathbf R_{j-1}^{(i)} = 0 \right) \\
   \mu_{j|j}^{(i)} \left( \mathbb A \right) &= \mathbb P \left( \mathbf x_j \in \mathbb A \,\Big|\, \mathbf R_k^{(i)} = 0, \cdots, \mathbf R_j^{(i)} = 0 \right)
 \end{align}
\end{subequations}
subject to $\mathbf R_j^{(i)} = \boldsymbol{\mathcal P}_j^{(i)} \left( \mathbf x_j \right)$ for all $i$ in $\mathbb N$ and $j$ in $\{k, \cdots, N\}$. Since $\left\{ \hat x_{k:N}^{(i_l-1)} \right\}_{l \in \mathbb N}$ is a convergent sequence, according to Proposition \ref{prop_decision_set_containment}, there exist compact subsets $\left\{ \mathbb K_j \right\}_{j=k}^N$ for which the following holds for all $l$ in $\mathbb N$ and $j$ in $\{k, \cdots, N\}$:
\begin{align*}
  \mu_{j|j}^{(i_l)} \left( \mathbb K_j \right) = 1
\end{align*}
Hence, for all $j$ in $\{k, \cdots, N\}$, the probability measures $\left\{ \mu_{j|j}^{(i_l)} \right\}_{l \in \mathbb N}$ are uniformly tight in the sense of Definition~\ref{def_uniform_tightness}. By Theorem 11.5.4 in \cite{dudley_9780511755347}, for each $j$ in $\{k, \cdots, N\}$, there exists a subsequence $\left\{ \mu_{j|j}^{(i_l')} \right\}_{l \in \mathbb N}$ that converges to a probability measure $\mu_{j|j}$ defined on $\left( \mathbb R^2 \times [0, 2\pi), \mathfrak B \right)$. In addition, note that $\mathbb P \left( \mathbf R_j^{(i)} = 0 \,\Big|\, \mathbf R_k^{(i)} = 0, \cdots, \mathbf R_{j-1}^{(i)} \right)$ takes a value in a compact set $[\epsilon, 1]$ for all $i$ in $\mathbb N$ and $j$ in $\{k, \cdots, N\}$. Hence, there is an infinite index set $\left\{ i_l'' \right\}_{l \in \mathbb N}$ of $\left\{ i_l' \right\}_{l \in \mathbb N}$ for which the following holds for all $j$ in $\{k, \cdots, N\}$:
\begin{align*}
  \lim_{l \to \infty} \mathbb P \left( \mathbf R_j^{(i_l'')} = 0 \,\Big|\, \mathbf R_k^{(i_l'')} = 0, \cdots, \mathbf R_{j-1}^{(i_l'')} = 0 \right) = q_j
\end{align*}
where $q_j$ belongs to $[\epsilon, 1]$. For clear and simple presentation, without loss of generality, we prove the lemma by imposing the following three assumptions: For each $j$ in $\{k, \cdots, N\}$,
\begin{enumerate} [label=\bfseries (A\arabic*)]
\item The estimates $\left\{ \hat x_{k:N}^{(i-1)} \right\}_{i \in \mathbb N}$ converge to $\hat x_{k:N}'$.

\item The probability measures $\left\{ \mu_{j|j}^{(i)} \right\}_{i \in \mathbb N}$ converge to $\mu_{j|j}$. \label{enum_F1}

\item $\lim_{i \to \infty} \mathbb P \left( \mathbf R_j^{(i)} = 0 \,\Big|\, \mathbf R_k^{(i)} = 0, \cdots, \mathbf R_{j-1}^{(i)} = 0 \right) = q_j$ holds, where $q_j$ belongs to $[\epsilon, 1]$ \label{enum_F2}
\end{enumerate}

To complete the proof of the lemma, it is sufficient to show that there exist policies $\boldsymbol{\mathcal P}_{k:N}$ for which
\begin{enumerate} [label=\bfseries (F\arabic*)]
\item For every $\mathbb A$ in $\mathfrak B$, it holds that
 \begin{subequations}
   \begin{align}
     \mu_{j|j} \left( \mathbb A \right) = \mathbb P \left( \mathbf x_j \in \mathbb A \,\Big|\, \mathbf R_k = 0, \cdots, \mathbf R_j = 0 \right)
   \end{align}
   and
   \begin{align}
     q_j = \mathbb P \left( \mathbf R_j=0 \,\Big|\, \mathbf R_k = 0, \cdots, \mathbf R_{j-1} = 0 \right)
   \end{align}

 \end{subequations}
 subject to $\mathbf R_j = \boldsymbol{\mathcal P}_j\left( \mathbf x_j \right)$ for all $j$ in $\{k, \cdots, N\}$. \label{enum_A1}

\item The policies $\boldsymbol{\mathcal P}_{k:N}$ belong to $\boldsymbol{\mathfrak P} \left( \hat x_{k:N}' \right)$, where $\hat x_{k:N}'$ is the limit of $\left\{ \hat x_{k:N}^{(i-1)} \right\}_{i \in \mathbb N}$. \label{enum_A2}
\end{enumerate}
For notational convenience, let us define
\begin{subequations}
 \begin{align} 
   \overline{\mathbb D}_j &= \left\{ x_j \in \mathbb R^2 \times [0, 2\pi) \,\Big|\, d^2 \left( x_j, \hat x_j' \right) + \mathbb E_{\mathbf x_{j+1}} \left[ J_{j+1}^\ast \left( \mathbf x_{j+1}, \hat x_{j+1:N}' \right) \,\Big|\, \mathbf x_j = x_j \right] \leq c_j' \right\} \label{def_best_decision_sets_limits_01_appendix} \\
   \underline{\mathbb D}_j &= \left\{ x_j \in \mathbb R^2 \times [0, 2\pi) \,\Big|\, d^2 \left( x_j, \hat x_j' \right) + \mathbb E_{\mathbf x_{j+1}} \left[ J_{j+1}^\ast \left( \mathbf x_{j+1}, \hat x_{j+1:N}' \right) \,\Big|\, \mathbf x_j = x_j \right] < c_j' \right\} \label{def_best_decision_sets_limits_02_appendix}
 \end{align}
\end{subequations}
and
\begin{subequations}
 \begin{align} 
   \overline{\mathbb D}_j^{(i)} &= \left\{ x_j \in \mathbb R^2 \times [0, 2\pi) \,\Big|\, d^2 \left( x_j, \hat x_j^{(i-1)} \right) + \mathbb E_{\mathbf x_{j+1}} \left[ J_{j+1}^\ast \left( \mathbf x_{j+1}, \hat x_{j+1:N}^{(i-1)} \right) \,\Big|\, \mathbf x_j = x_j \right] \leq c_j' \right\} \label{def_best_decision_sets_01_appendix} \\
   \underline{\mathbb D}_j^{(i)} &= \left\{ x_j \in \mathbb R^2 \times [0, 2\pi) \,\Big|\, d^2 \left( x_j, \hat x_j^{(i-1)} \right) + \mathbb E_{\mathbf x_{j+1}} \left[ J_{j+1}^\ast \left( \mathbf x_{j+1}, \hat x_{j+1:N}^{(i-1)} \right) \,\Big|\, \mathbf x_j = x_j \right] < c_j' \right\} \label{def_best_decision_sets_02_appendix}
 \end{align}
\end{subequations}
for each $i$ in $\mathbb N$ and $j$ in $\{k, \cdots, N\}$, where $J_{j+1}^\ast$ is defined in \eqref{eq_cost-to-go_02}. Note that according to Proposition \ref{prop_mapping_G_j_continuity}, the sets $\overline{\mathbb D}_j$ and $\overline{\mathbb D}_j^{(i)}$ are closed, and the sets $\underline{\mathbb D}_j$ and $\underline{\mathbb D}_j^{(i)}$ are open.

We first make the following two claims to show that \ref{enum_A1} is true.

\noindent \textbf{Claim 1:} For each $\mathbb A$ in $\mathfrak B$, let us define 
\begin{align} \label{eq_claim1_01}
  \mu_{j|j-1} \left( \mathbb A \right) \overset{def}{=} \int_{\mathbb R^2 \times [0, 2\pi)} \mathbb P \left( \mathbf x_j \in \mathbb A \,\Big|\, \mathbf x_{j-1} = x \right) \,\mathrm d\mu_{j-1|j-1}
\end{align}
Then, $\mu_{j|j-1}$ is a probability measure on $\left( \mathbb R^2 \times [0, 2\pi), \mathfrak B\right)$, and it holds that
$$\lim_{i \to \infty} \mu_{j|j-1}^{(i)}\left( \mathbb A \right) = \mu_{j|j-1}\left( \mathbb A \right)$$
for all $\mathbb A$ in $\mathfrak B$.

To prove the claim, based on Proposition \ref{prop_time_evolution}, we note that
\begin{align}
 \mu_{j|j-1}^{(i)} \left( \mathbb A \right) &= \int_{\mathbb R^2 \times [0, 2\pi)} \mathbb P \left( \mathbf x_j \in \mathbb A \,\Big|\, \mathbf x_{j-1} = x \right) \,\mathrm d\mu_{j-1|j-1}^{(i)}
\end{align}
holds for each $\mathbb A$ in $\mathfrak B$. By definition, for fixed $x$ in $\mathbb R^2 \times [0, 2\pi)$, the mapping $\mathbb A \mapsto \mathbb P \left( \mathbf x_j \in \mathbb A \,\Big|\, \mathbf x_{j-1} = x \right)$ defines a probability measure on $\left( \mathbb R^2 \times [0, 2\pi), \mathfrak B\right)$. In conjunction with Assumption \ref{assump_transition_probability}, we can see that ${x \mapsto \mathbb P \left( \mathbf x_j \in \mathbb A \,\Big|\, \mathbf x_{j-1} = x \right)}$ is a bounded, continuous function. Hence, using \ref{enum_F1}, we have that
\begin{align} \label{eq_claim_01}
 \lim_{i \to \infty} \mu_{j|j-1}^{(i)} \left( \mathbb A \right) &= \lim_{i \to \infty} \int_{\mathbb R^2 \times [0, 2\pi)} \mathbb P \left( \mathbf x_j \in \mathbb A \,\Big|\, \mathbf x_{j-1} = x \right) \,\mathrm d\mu_{j-1|j-1}^{(i)} \nonumber \\
 &= \int_{\mathbb R^2 \times [0, 2\pi)} \mathbb P \left( \mathbf x_j \in \mathbb A \,\Big|\, \mathbf x_{j-1} = x \right) \,\mathrm d\mu_{j-1|j-1} = \mu_{j|j-1} \left( \mathbb A \right)
\end{align}

Lastly, the argument that $\mu_{j|j-1}$ is a probability measure on $\left( \mathbb R^2 \times [0, 2\pi), \mathfrak B \right)$ follows from \eqref{eq_claim_01} and the fact that $\mathbb A \mapsto \mathbb P \left( \mathbf x_j \in \mathbb A \,\Big|\, \mathbf x_{j-1} = x \right)$ is a probability measure on $\left( \mathbb R^2 \times [0, 2\pi), \mathfrak B\right)$.
\hfill $\square$

\noindent \textbf{Claim 2:} There exists a measurable function $f_j: \mathbb R^2 \times [0, 2\pi) \to \mathbb [0, 1]$ for which
\begin{align}
 \mu_{j|j} \left( \mathbb A \right) = \frac{\int_{\mathbb A} f_j \,\mathrm d\mu_{j|j-1}}{q_j}
\end{align}
holds for all $\mathbb A$ in $\mathfrak B$, where $\mu_{j|j-1}$ is defined in \eqref{eq_claim1_01}.

Based on Lemma \ref{thm_portmanteau} and Proposition \ref{prop_time_evolution}, for any open set $\mathbb O$, we can see that the following relations hold:
\begin{align} \label{eq_proof_lemma_P_closed_01}
 \mu_{j|j} \left( \mathbb O \right) \leq \liminf_{i \to \infty} \mu_{j|j}^{(i)} \left( \mathbb O \right) &\overset{\mathbf{(i)}}{\leq} \lim_{i \to \infty} \frac{\mu_{j|j-1}^{(i)} \left( \mathbb O \right)}{\mathbb P \left( \mathbf R_j^{(i)}=0 \,\Big|\, \mathbf R_k^{(i)}=0, \cdots, \mathbf R_{j-1}^{(i)}=0 \right)} \nonumber \\
                                                                                                         &\overset{\mathbf{(ii)}}{=} \frac{\mu_{j|j-1}\left( \mathbb O \right)}{q_j}
\end{align}
where $\mathbf{(i)}$ follows from Proposition \ref{prop_time_evolution}, and $\mathbf{(ii)}$ follows from Claim 1 and \ref{enum_F2}. We argue that the following holds for any set $\mathbb A$ in $\mathfrak B$:
\begin{align} \label{eq_proof_lemma_P_closed_02}
 \mu_{j|j} \left( \mathbb A \right) \leq \frac{\mu_{j|j-1} \left( \mathbb A \right)}{q_j}
\end{align}
To justify the argument, by contradiction, suppose that for a set $\mathbb A$ in $\mathfrak B$ it holds that
\begin{align}
 \mu_{j|j} \left( \mathbb A \right) > \frac{\mu_{j|j-1} \left( \mathbb A \right)}{q_j}
\end{align}
By the closed regularity theorem (see Theorem 7.1.3 in \cite{dudley_9780511755347}) and Remark \ref{remark_outer_regularity}, we can choose an open set $\mathbb O$ containing $\mathbb A$ for which the following holds:
\begin{align}
 \mu_{j|j} \left( \mathbb O \right) \geq \mu_{j|j} \left( \mathbb A \right) &> \frac{\mu_{j|j-1} \left( \mathbb O \right)}{q_j} \nonumber \\
                                                                            &\geq \frac{\mu_{j|j-1} \left( \mathbb A \right)}{q_j}
\end{align}
This contradicts \eqref{eq_proof_lemma_P_closed_01}. 

Notice that \eqref{eq_proof_lemma_P_closed_02} implies that $\mu_{j|j}$ is \textit{absolutely continuous} with respect to $\mu_{j|j-1}$. According to the Radon-Nikodym theorem, there is a measurable function $f_j: \mathbb R^2 \times [0, 2\pi) \to \mathbb R_+$ for which the following holds for all $\mathbb A$ in $\mathfrak B$:
\begin{align} \label{eq_proof_lemma_P_closed_04}
 \mu_{j|j} \left( \mathbb A \right) = \frac{\int_{\mathbb A} f_j \,\mathrm d\mu_{j|j-1}}{q_j}
\end{align}
In addition, it can be verified that $f_j(x) \leq 1$ for almost every $x$ in $\mathbb R^2 \times [0, 2\pi)$; otherwise \eqref{eq_proof_lemma_P_closed_02} would be violated. \hfill $\square$

\noindent \textbf{Proof of \ref{enum_A1}:} Using the function $f_j$ obtained in Claim 2, let us define policies $\boldsymbol{\mathcal P}_{k:N}$ as follows: For each $j$ in $\{k, \cdots, N\}$
\begin{align} \label{eq_proof_lemma_P_closed_05}
  \boldsymbol{\mathcal P}_j(x) = \begin{cases} 0 & \text{with probability } f_j(x) \\ 1 & \text{with probability } 1-f_j(x)
  \end{cases}
\end{align}
In conjunction with \eqref{eq_proof_lemma_P_closed_04}, we can verify that under the policies $\boldsymbol{\mathcal P}_{k:N}$, the following holds:
\begin{align} \label{eq_proof_A1_01}
  \mu_{j|j} \left( \mathbb A \right) = \frac{\int_{\mathbb A} \mathbb P \left( \boldsymbol{\mathcal P}_j (\mathbf x_j) = 0 \,\Big|\, \mathbf x_j=x\right) \,\mathrm d\mu_{j|j-1}}{q_j}
\end{align}


Using \eqref{eq_proof_A1_01} and Proposition \ref{prop_time_evolution}, it can be verified that the following hold for every $\mathbb A$ in $\mathfrak B$:
\begin{subequations}
 \begin{align*}
   \mu_{j|j} \left( \mathbb A \right) = \mathbb P \left( \mathbf x_j \in \mathbb A \,\Big|\, \mathbf R_k = 0, \cdots, \mathbf R_j = 0 \right)
 \end{align*}
 and
 \begin{align*}
   q_j = \mathbb P \left( \mathbf R_j=0 \,\Big|\, \mathbf R_k = 0, \cdots, \mathbf R_{j-1} = 0 \right)
 \end{align*}
\end{subequations}
subject to $\mathbf R_j = \boldsymbol{\mathcal P}_j\left( \mathbf x_j \right)$ for all $j$ in $\{k, \cdots, N\}$. This completes the proof. \hfill \QED


Henceforth, we make two additional claims under the policies $\boldsymbol{\mathcal P}_{k:N}$ determined as in \eqref{eq_proof_lemma_P_closed_05} to show that \ref{enum_A2} is true.

\noindent \textbf{Claim 3:} For any Borel measurable subset $\mathbb A$ contained in $\overline{\mathbb D}_j^c$, it holds that 
$$\mathbb P \left( \mathbf x_j \in \mathbb A \,\Big|\, \mathbf R_k = 0, \cdots, \mathbf R_j = 0 \right)=0$$
where the set $\overline{\mathbb D}_j$ is defined in \eqref{def_best_decision_sets_limits_01_appendix}.

Notice that
\begin{align} \label{eq_proof_lemma_P_closed_03}
  &\mathbb P \left( \mathbf x_j \in \mathbb A \,\Big|\, \mathbf R_k=0, \cdots, \mathbf R_j=0\right) \nonumber \\
  &= \frac{\mathbb P \left( \mathbf x_j \in \mathbb A \,\Big|\, \mathbf R_k=0, \cdots, \mathbf R_{j-1}=0\right)}{\mathbb P \left( \mathbf R_j=0 \,\Big|\, \mathbf R_k=0, \cdots, \mathbf R_{j-1}=0 \right)} \nonumber \\
  &\quad -\frac{\mathbb P \left( \mathbf x_j \in \mathbb A \,\Big|\, \mathbf R_k=0, \cdots, \mathbf R_{j}=1\right) \cdot \mathbb P \left( \mathbf R_j=1 \,\Big|\, \mathbf R_k=0, \cdots, \mathbf R_{j-1}=0 \right)}{\mathbb P \left( \mathbf R_j=0 \,\Big|\, \mathbf R_k=0, \cdots, \mathbf R_{j-1}=0 \right)}
\end{align}

To prove the claim, let $\mathbb O$ be an open set contained in $\overline{\mathbb D}_j^c$. By Lemma \ref{lemma_optimal_policies} and Proposition \ref{prop_time_evolution}, we can derive the following:
\begin{align*}
 \mu_{j|j}^{(i)} \left( \mathbb O \right) = \mu_{j|j}^{(i)} \left( \mathbb O \cap \overline{\mathbb D}_j^{(i)} \right) \leq \frac{\mu_{j|j-1}^{(i)} \left( \mathbb O \cap \overline{\mathbb D}_j^{(i)} \right)}{\mathbb P \left( \mathbf R_j^{(i)}=0 \,\Big|\, \mathbf R_k^{(i)}=0, \cdots, \mathbf R_{j-1}^{(i)}=0 \right)}
\end{align*}
where the set $\overline{\mathbb D}_j^{(i)}$ is defined in \eqref{def_best_decision_sets_01_appendix}. By applying Theorem \ref{thm_portmanteau}, we can show that the following holds for all $i_0$ in $\mathbb N$:
\begin{align*}
 \mu_{j|j} \left( \mathbb O \right) &\leq \liminf_{i \to \infty} \mu_{j|j}^{(i)} \left( \mathbb O \right) \nonumber \\
                                    &\leq \liminf_{i \to \infty} \frac{\mu_{j|j-1}^{(i)} \left( \mathbb O \cap \overline{\mathbb D}_j^{(i)} \right)}{\mathbb P \left( \mathbf R_j^{(i)}=0 \,\Big|\, \mathbf R_k^{(i)}=0, \cdots, \mathbf R_{j-1}^{(i)}=0 \right)} \nonumber \\
                                    &\leq \liminf_{i \to \infty} \frac{\mu_{j|j-1}^{(i)} \left( \mathbb O \cap \left( \bigcup_{l \geq i} \overline{\mathbb D}_j^{(l)} \right) \right)}{\mathbb P \left( \mathbf R_j^{(i)} = 0 \,\Big|\, \mathbf R_k^{(i)}=0, \cdots, \mathbf R_{j-1}^{(i)}=0 \right)} \nonumber \\
                                    &\overset{\mathbf{(i)}}{\leq} \frac{\mu_{j|j-1} \left( \mathbb O \cap \left( \bigcup_{l \geq i_0} \overline{\mathbb D}_j^{(l)} \right) \right)}{\mathbb P \left( \mathbf R_j = 0 \,\Big|\, \mathbf R_k=0, \cdots, \mathbf R_{j-1}=0 \right)}
\end{align*}
where $\mathbf{(i)}$ follows from Claim 1, \ref{enum_A1}, and the fact that $\left\{ \bigcup_{l \geq i} \overline{\mathbb D}_j^{(l)} \right\}_{i \in \mathbb N}$ is a decreasing sequence of measurable sets. Hence, from Lemma \ref{lemma_set_convergence}, we have that
\begin{align*}
 \mu_{j|j} \left( \mathbb O \right) \leq \frac{\mu_{j|j-1} \left( \mathbb O \cap \left( \bigcap_{i \in \mathbb N} \bigcup_{l \geq i} \overline{\mathbb D}_j^{(l)} \right) \right)}{\mathbb P \left( \mathbf R_j = 0 \,\Big|\, \mathbf R_k=0, \cdots, \mathbf R_{j-1}=0 \right)} = 0
\end{align*}
Since $\overline{\mathbb D}_j^c$ is an open set, by selecting $\mathbb O = \overline{\mathbb D}_j^c$, we conclude that the following holds for every Borel measurable subset $\mathbb A$ of $\overline{\mathbb D}_j^c$:
\begin{align*}
 \mu_{j|j} \left( \mathbb A \right) \leq \mu_{j|j} \left( \overline{\mathbb D}_j^c \right) = 0
\end{align*}
 \hfill $\square$

\noindent \textbf{Claim 4:} Suppose that $\mathbb P \left( \mathbf R_j=1 \,\Big|\, \mathbf R_k=0, \cdots, \mathbf R_{j-1}= 0 \right)$ is non-zero. Then, for any Borel measurable subset $\mathbb A$ contained in the set $\underline{\mathbb D}_j$ given in \eqref{def_best_decision_sets_limits_02_appendix}, it holds that
$$\mathbb P \left( \mathbf x_j \in \mathbb A \,\Big|\, \mathbf R_k=0, \cdots, \mathbf R_j=1 \right) = 0$$

To prove the claim, let $\mathbb F$ be a closed set contained in $\underline{\mathbb D}_j$. Notice that, by Lemma \ref{lemma_optimal_policies}, the following holds for all $i$ in $\mathbb N$:
\begin{align} \label{eq_proof_lemma_P_closed_03}
  &\mathbb P \left( \mathbf x_j \in \mathbb F \cap \underline{\mathbb D}_j^{(i)} \,\Big|\, \mathbf R_k^{(i)}=0, \cdots, \mathbf R_j^{(i)}=0 \right) \nonumber \\
  &= \frac{\mathbb P \left( \mathbf x_j \in \mathbb F \cap \underline{\mathbb D}_j^{(i)} \,\Big|\, \mathbf R_k^{(i)}=0, \cdots, \mathbf R_{j-1}^{(i)}=0\right)}{\mathbb P \left( \mathbf R_j^{(i)}=0 \,\Big|\, \mathbf R_k^{(i)}=0, \cdots, \mathbf R_{j-1}^{(i)}=0 \right)}
\end{align}
Using \eqref{eq_proof_lemma_P_closed_03} and Theorem \ref{thm_portmanteau}, we can show that the following holds for all $i_0$ in $\mathbb N$:
\begin{align*}
 \mu_{j|j} \left( \mathbb F \right) &\geq \limsup_{i \to \infty} \mu_{j|j}^{(i)} \left( \mathbb F \right) \nonumber \\
                                    &\geq \limsup_{i \to \infty} \mu_{j|j}^{(i)} \left( \mathbb F \cap \underline{\mathbb D}_j^{(i)}\right) \nonumber \\
                                    &= \limsup_{i \to \infty} \frac{\mu_{j|j-1}^{(i)} \left( \mathbb F \cap \underline{\mathbb D}_j^{(i)} \right)}{\mathbb P \left( \mathbf R_j^{(i)}=0 ~\big|~ \mathbf R_k^{(i)}=0, \cdots, \mathbf R_{j-1}^{(i)}=0 \right)} \nonumber \\
                                    &\geq \limsup_{i \to \infty} \frac{\mu_{j|j-1}^{(i)} \left( \mathbb F \cap \left( \bigcap_{l \geq i} \underline{\mathbb D}_j^{(l)} \right) \right)}{\mathbb P \left( \mathbf R_j^{(i)}=0 ~\big|~ \mathbf R_k^{(i)}=0, \cdots, \mathbf R_{j-1}^{(i)}=0 \right)} \nonumber \\
                                    &\overset{\mathbf{(i)}}{\geq} \frac{\mu_{j|j-1} \left( \mathbb F \cap \left(\bigcap_{l \geq i_0} \underline{\mathbb D}_j^{(l)} \right) \right)}{\mathbb P \left( \mathbf R_j=0 ~\big|~ \mathbf R_k=0, \cdots, \mathbf R_{j-1}=0 \right)}
\end{align*}
where $\mathbf{(i)}$ follows from Claim 1, \ref{enum_A1}, and the fact that $\left\{ \bigcap_{l \geq i} \underline{\mathbb D}_j^{(l)} \right\}_{i \in \mathbb N}$ is an increasing sequence of measurable sets. Hence, from Lemma \ref{lemma_set_convergence}, we have that
\begin{align*}
 \mu_{j|j} \left( \mathbb F \right) &\geq \frac{\mu_{j|j-1} \left( \mathbb F \cap \left(\bigcup_{i \in \mathbb N} \bigcap_{l \geq i} \underline{\mathbb D}_j^{(l)} \right) \right)}{\mathbb P \left( \mathbf R_j=0 ~\big|~ \mathbf R_k=0, \cdots, \mathbf R_{j-1}=0 \right)} \nonumber \\
                                    &= \frac{\mu_{j|j-1} \left( \mathbb F \right)}{\mathbb P \left( \mathbf R_j=0 ~\big|~ \mathbf R_k=0, \cdots, \mathbf R_{j-1}=0 \right)}
\end{align*}
Using this relation, we can see that
\begin{align*}
  &\mu_{j|j-1} \left( \mathbb F \right) \nonumber \\
  &= \mathbb P \left( \mathbf x_j \in \mathbb F \,\Big|\, \mathbf R_k = 0, \cdots, \mathbf R_{j-1} = 0 \right) \nonumber \\
  &= \mathbb P \left( \mathbf x_j \in \mathbb F \,\Big|\, \mathbf R_k = 0, \cdots, \mathbf R_j = 0 \right) \cdot \mathbb P \left( \mathbf R_j = 0 \,\Big|\, \mathbf R_k = 0, \cdots, \mathbf R_{j-1} = 0 \right) \nonumber \\
  & \quad + \mathbb P \left( \mathbf x_j \in \mathbb F \,\Big|\, \mathbf R_k = 0, \cdots, \mathbf R_j = 1 \right) \cdot \mathbb P \left( \mathbf R_j = 1 \,\Big|\, \mathbf R_k = 0, \cdots, \mathbf R_{j-1} = 0 \right) \nonumber \\
  &\geq \mu_{j|j-1} \left( \mathbb F \right) \nonumber \\
  & \quad + \mathbb P \left( \mathbf x_j \in \mathbb F \,\Big|\, \mathbf R_k = 0, \cdots, \mathbf R_j = 1 \right) \cdot \mathbb P \left( \mathbf R_j = 1 \,\Big|\, \mathbf R_k = 0, \cdots, \mathbf R_{j-1} = 0 \right)
\end{align*}
Using the fact that $\mathbb P \left( \mathbf R_j = 1 \,\Big|\, \mathbf R_k = 0, \cdots, \mathbf R_{j-1} = 0 \right)$ is non-zero, we obtain
\begin{align*}
  \mathbb P \left( \mathbf x_j \in \mathbb F \,\Big|\, \mathbf R_k = 0, \cdots, \mathbf R_j = 1 \right) = 0
\end{align*}
Based on the closed regularity theorem (see Theorem 7.1.3 in \cite{dudley_9780511755347}), we can see that the following holds for any Borel measurable set $\mathbb A$ contained in $\underline{\mathbb D}_j$:
\begin{align*}
  \mathbb P \left( \mathbf x_j \in \mathbb A \,\Big|\, \mathbf R_k = 0, \cdots, \mathbf R_j = 1 \right) = 0
\end{align*}
\hfill $\square$

\noindent \textbf{Proof of \ref{enum_A2}: } Recall that $\mathbb E_{\mathbf x_j} \left[ J_j \left(\mathbf x_j, \boldsymbol{\mathcal P}_{j:N}, \hat x_{j:N}' \right) \,\Big|\, \mathbf R_k=0, \cdots, \mathbf R_{j-1}=0 \right]$ and $J_j^\ast$ are defined in \eqref{eq_cost-to-go_01} and \eqref{eq_cost-to-go_02}, respectively, and that $\mathbf R_j = \boldsymbol{\mathcal P}_j\left( \mathbf x_j \right)$ for all $j$ in $\{k, \cdots, N\}$. We will use the mathematical induction to show that the following is true: 
\begin{align*}
  \mathbb E_{\mathbf x_k} \left[ J_k \left(\mathbf x_k, \boldsymbol{\mathcal P}_{k:N}, \hat x_{k:N}' \right) \right] = \min_{\boldsymbol{\mathcal P}_{k:N}'}\mathbb E_{\mathbf x_k} \left[ J_k \left(\mathbf x_k, \boldsymbol{\mathcal P}_{k:N}', \hat x_{k:N}' \right) \right]
\end{align*}

Using Claim 3 and Claim 4, we can derive the following:
\begin{subequations} \label{eq_proof_A2_01}
 \begin{align}
   &\mathbb E_{\mathbf x_N} \left[ d^2 \left( \mathbf x_N, \hat x_N' \right) \,\Big|\, \mathbf R_N = 0 \right] \nonumber \\
   &= \mathbb E_{\mathbf x_N} \left[ \min \left\{ d^2 \left( \mathbf x_N, \hat x_N' \right), c_N' \right\} \,\Big|\, \mathbf R_N = 0 \right] \nonumber \\
   &=\mathbb E_{\mathbf x_N} \left[ J_N^\ast \left(\mathbf x_N, \hat x_{N}' \right) \,\Big|\, \mathbf R_N = 0 \right]
 \end{align}
 and
 \begin{align}
   c_N' &= \mathbb E_{\mathbf x_N} \left[ \min \left\{ d^2 \left( \mathbf x_N, \hat x_N' \right), c_N' \right\} \,\Big|\, \mathbf R_N = 1 \right] \nonumber \\
   &=\mathbb E_{\mathbf x_N} \left[ J_N^\ast \left(\mathbf x_N, \hat x_{N}' \right) \,\Big|\, \mathbf R_N = 1 \right]
 \end{align}
\end{subequations}
provided that $\mathbb P \left( \mathbf R_N=1 \,\Big|\, \mathbf R_k=0, \cdots, \mathbf R_{N-1}=0\right)$ is nonzero. From \eqref{eq_cost-to-go_01}, \eqref{eq_cost-to-go_02}, and \eqref{eq_proof_A2_01}, we can derive that
\begin{align}
 &\mathbb E_{\mathbf x_N} \left[ J_N \left(\mathbf x_N, \boldsymbol{\mathcal P}_{N}, \hat x_{N}' \right) \,\Big|\, \mathbf R_{N-1}=0 \right] \nonumber \\
 &=\mathbb E_{\mathbf x_N} \left[ J_N^\ast \left(\mathbf x_N, \hat x_{N}' \right) \,\Big|\, \mathbf R_{N-1}=0 \right]
\end{align}

Suppose that the following relation holds:
\begin{align}
 &\mathbb E_{\mathbf x_{j+1}} \left[ J_{j+1} \left(\mathbf x_{j+1}, \boldsymbol{\mathcal P}_{j+1:N}, \hat x_{j+1:N}' \right) \,\Big|\, \mathbf R_k=0, \cdots, \mathbf R_{j}=0 \right] \nonumber \\
 &=\mathbb E_{\mathbf x_{j+1}} \left[ J_{j+1}^\ast \left(\mathbf x_{j+1}, \hat x_{j+1:N}' \right) \,\Big|\, \mathbf R_k=0, \cdots, \mathbf R_{j}=0 \right]
\end{align}
Then, using Claim 3 and Claim 4, we can derive the following:
\begin{subequations} \label{eq_proof_A2_02}
 {\small \begin{align}
   &\mathbb E_{\mathbf x_j} \left[ d^2 \left( \mathbf x_j, \hat x_j' \right) + \mathbb E_{\mathbf x_{j+1}} \left[ J_{j+1} \left(\mathbf x_{j+1}, \boldsymbol{\mathcal P}_{j+1:N}, \hat x_{j+1:N}' \right) \,\Big|\, \mathbf x_j \right] \,\Big|\, \mathbf R_k = 0, \cdots, \mathbf R_j = 0 \right] \nonumber \\
   &=\mathbb E_{\mathbf x_j} \left[ \min \left\{ d^2 \left( \mathbf x_j, \hat x_j' \right) + \mathbb E_{\mathbf x_{j+1}} \left[ J_{j+1}^\ast \left(\mathbf x_{j+1}, \hat x_{j+1:N}' \right) \,\Big|\, \mathbf x_j \right], c_j' \right\} \,\Big|\, \mathbf R_k = 0, \cdots, \mathbf R_j = 0 \right] \nonumber \\
   &=\mathbb E_{\mathbf x_j} \left[ J_j^\ast \left(\mathbf x_j, \hat x_{j:N}' \right) \,\Big|\, \mathbf R_k = 0, \cdots, \mathbf R_j = 0 \right]
 \end{align}}
 and
 {\small \begin{align}
   c_j' &=\mathbb E_{\mathbf x_j} \left[ \min \left\{ d^2 \left( \mathbf x_j, \hat x_j' \right) + \mathbb E_{\mathbf x_{j+1}} \left[ J_{j+1}^\ast \left(\mathbf x_{j+1}, \hat x_{j+1:N}' \right) \,\Big|\, \mathbf x_j \right], c_j' \right\} \,\Big|\, \mathbf R_k = 0, \cdots, \mathbf R_j = 1 \right] \nonumber \\
        &= \mathbb E_{\mathbf x_j} \left[ J_j^\ast \left(\mathbf x_j, \hat x_{j:N}' \right) \,\Big|\, \mathbf R_k = 0, \cdots, \mathbf R_j = 1 \right]
 \end{align}}
\end{subequations}
 provided that $\mathbb P \left( \mathbf R_j=1 \,\Big|\, \mathbf R_k=0, \cdots, \mathbf R_{j-1}=0\right)$ is non-zero. From \eqref{eq_cost-to-go_01}, \eqref{eq_cost-to-go_02}, and \eqref{eq_proof_A2_02}, we can derive that
\begin{align} \label{eq_proof_A2_03}
 &\mathbb E_{\mathbf x_j} \left[ J_j \left(\mathbf x_j, \boldsymbol{\mathcal P}_{j:N}, \hat x_{j:N}' \right) \,\Big|\, \mathbf R_k=0, \cdots, \mathbf R_{j-1}=0 \right] \nonumber \\
 &=\mathbb E_{\mathbf x_j} \left[ J_j^\ast \left(\mathbf x_j, \hat x_{j:N}' \right) \,\Big|\, \mathbf R_k=0, \cdots, \mathbf R_{j-1}=0 \right]
\end{align}

By induction, we conclude that \eqref{eq_proof_A2_03} holds for all $j$ in $\{k, \cdots, N\}$. By Definition~\ref{def_best_response_mapping_P} and the fact that
\begin{align*}
  \min_{\boldsymbol{\mathcal P}_{k:N}'}\mathbb E_{\mathbf x_k} \left[ J_k \left(\mathbf x_k, \boldsymbol{\mathcal P}_{k:N}', \hat x_{k:N}' \right) \right] =\mathbb E_{\mathbf x_k} \left[ J_k^\ast \left(\mathbf x_k, \hat x_{k:N}' \right) \right]
\end{align*}
we conclude that the policies $\boldsymbol{\mathcal P}_{k:N}$ belong to $\boldsymbol{\mathfrak P} \left( \hat x_{k:N}' \right)$. \hfill \QED

\end{appendix}
\bibliographystyle{IEEEtran}
\bibliography{IEEEabrv,spark2016_remote_estimation}

\begin{thebibliography}{10}
\providecommand{\url}[1]{#1}
\csname url@rmstyle\endcsname
\providecommand{\newblock}{\relax}
\providecommand{\bibinfo}[2]{#2}
\providecommand\BIBentrySTDinterwordspacing{\spaceskip=0pt\relax}
\providecommand\BIBentryALTinterwordstretchfactor{4}
\providecommand\BIBentryALTinterwordspacing{\spaceskip=\fontdimen2\font plus
\BIBentryALTinterwordstretchfactor\fontdimen3\font minus
  \fontdimen4\font\relax}
\providecommand\BIBforeignlanguage[2]{{%
\expandafter\ifx\csname l@#1\endcsname\relax
\typeout{** WARNING: IEEEtran.bst: No hyphenation pattern has been}%
\typeout{** loaded for the language `#1'. Using the pattern for}%
\typeout{** the default language instead.}%
\else
\language=\csname l@#1\endcsname
\fi
#2}}

\bibitem{qiang_du1999_siam_review}
Q.~Du, V.~Faber, and M.~Gunzburger, ``Centroidal voronoi tessellations:
  Applications and algorithms,'' \emph{SIAM Review}, vol.~41, no.~4, pp.
  637--676, 1999.

\bibitem{molin2012_ifac}
A.~Molin and S.~Hirche, ``An iterative algorithm for optimal event-triggered
  estimation,'' in \emph{4th IFAC Conference on Analysis and Design of Hybrid
  Systems}, June 2012, pp. 64--69.

\bibitem{lipsa2011_ieee_tac}
G.~M. Lipsa and N.~C. Martins, ``Remote state estimation with communication
  costs for first-order {LTI} systems,'' \emph{{IEEE} Trans. Automat. Contr.},
  vol.~56, no.~9, pp. 2013--2025, Sept 2011.

\bibitem{nayyar2013_ieee_tac}
A.~Nayyar, T.~Başar, D.~Teneketzis, and V.~V. Veeravalli, ``Optimal strategies
  for communication and remote estimation with an energy harvesting sensor,''
  \emph{IEEE Transactions on Automatic Control}, vol.~58, no.~9, pp.
  2246--2260, Sept 2013.

\bibitem{spark7040013}
S.~Park and N.~Martins, ``Individually optimal solutions to a remote state
  estimation problem with communication costs,'' in \emph{Decision and Control
  (CDC), 2014 IEEE 53rd Annual Conference on}, Dec 2014, pp. 4014--4019.

\bibitem{yonggang_xu2004_ieee_cdc}
Y.~Xu and J.~P. Hespanha, ``Optimal communication logics in networked control
  systems,'' in \emph{Decision and Control, 2004. CDC. 43rd IEEE Conference
  on}, vol.~4, Dec 2004, pp. 3527--3532.

\bibitem{cogil2007_acc}
R.~Cogill, S.~Lall, and J.~P. Hespanha, ``A constant factor approximation
  algorithm for event-based sampling,'' in \emph{American Control Conference,
  2007. ACC '07}, July 2007, pp. 305--311.

\bibitem{lichun_li2011_ieee_cdc-ecc}
L.~Li and M.~Lemmon, ``Performance and average sampling period of sub-optimal
  triggering event in event triggered state estimation,'' in \emph{Decision and
  Control and European Control Conference (CDC-ECC), 2011 50th IEEE Conference
  on}, Dec 2011, pp. 1656--1661.

\bibitem{lichun_li2013_acc}
L.~Li, Z.~Wang, and M.~Lemmon, ``Polynomial approximation of optimal event
  triggers for state estimation problems using {SOSTOOLS},'' in \emph{American
  Control Conference (ACC), 2013}, June 2013, pp. 2699--2704.

\bibitem{dudley_9780511755347}
R.~M. Dudley, \emph{Real Analysis and Probability}, 2nd~ed.\hskip 1em plus
  0.5em minus 0.4em\relax Cambridge University Press, 2002.

\bibitem{mcclintock2012_em}
B.~T. McClintock, R.~King, L.~Thomas, J.~Matthiopoulos, B.~J. McConnell, and
  J.~M. Morales, ``A general discrete-time modeling framework for animal
  movement using multistate random walks,'' \emph{Ecological Monographs},
  vol.~82, no.~3, pp. 335--349, 2012.

\bibitem{billingsley1995_wiley}
P.~Billingsley, \emph{Probability and measure}, 3rd~ed.\hskip 1em plus 0.5em
  minus 0.4em\relax Wiley, 1995.

\bibitem{skorokhod1956_tpa}
A.~V. Skorokhod, ``Limit theorems for stochastic processes,'' \emph{Theory of
  Probability \& Its Applications}, vol.~1, no.~3, pp. 261--290, 1956.

\end{thebibliography}

\end{document}